\documentclass{amsart}

\usepackage{cleveref}
\crefformat{equation}{(#2#1#3)}
\usepackage{url}
\usepackage{cite}
\usepackage{color}

\usepackage{mathtools}
\usepackage{amssymb}
\usepackage{mathrsfs}
\usepackage{amsthm}

\makeatletter
\let\save@mathaccent\mathaccent
\newcommand*\if@single[3]{%
  \setbox0\hbox{${\mathaccent"0362{#1}}^H$}%
  \setbox2\hbox{${\mathaccent"0362{\kern0pt#1}}^H$}%
  \ifdim\ht0=\ht2 #3\else #2\fi
  }
\newcommand*\rel@kern[1]{\kern#1\dimexpr\macc@kerna}
\newcommand*\widebar[1]{\@ifnextchar^{{\wide@bar{#1}{0}}}{\wide@bar{#1}{1}}}
\newcommand*\wide@bar[2]{\if@single{#1}{\wide@bar@{#1}{#2}{1}}{\wide@bar@{#1}{#2}{2}}}
\newcommand*\wide@bar@[3]{%
  \begingroup
  \def\mathaccent##1##2{%
    \let\mathaccent\save@mathaccent
    \if#32 \let\macc@nucleus\first@char \fi
    \setbox\z@\hbox{$\macc@style{\macc@nucleus}_{}$}%
    \setbox\tw@\hbox{$\macc@style{\macc@nucleus}{}_{}$}%
    \dimen@\wd\tw@
    \advance\dimen@-\wd\z@
    \divide\dimen@ 3
    \@tempdima\wd\tw@
    \advance\@tempdima-\scriptspace
    \divide\@tempdima 10
    \advance\dimen@-\@tempdima
    \ifdim\dimen@>\z@ \dimen@0pt\fi
    \rel@kern{0.6}\kern-\dimen@
    \if#31
      \overline{\rel@kern{-0.6}\kern\dimen@\macc@nucleus\rel@kern{0.4}\kern\dimen@}%
      \advance\dimen@0.4\dimexpr\macc@kerna
      \let\final@kern#2%
      \ifdim\dimen@<\z@ \let\final@kern1\fi
      \if\final@kern1 \kern-\dimen@\fi
    \else
      \overline{\rel@kern{-0.6}\kern\dimen@#1}%
    \fi
  }%
  \macc@depth\@ne
  \let\math@bgroup\@empty \let\math@egroup\macc@set@skewchar
  \mathsurround\z@ \frozen@everymath{\mathgroup\macc@group\relax}%
  \macc@set@skewchar\relax
  \let\mathaccentV\macc@nested@a
  \if#31
    \macc@nested@a\relax111{#1}%
  \else
    \def\gobble@till@marker##1\endmarker{}%
    \futurelet\first@char\gobble@till@marker#1\endmarker
    \ifcat\noexpand\first@char A\else
      \def\first@char{}%
    \fi
    \macc@nested@a\relax111{\first@char}%
  \fi
  \endgroup
}
\makeatother

\DeclarePairedDelimiter\norm\lVert\rVert%
\newtheorem{theorem}{Theorem}[section]
\newtheorem{lemma}[theorem]{Lemma}
\newtheorem{corr}[theorem]{Corollary}

\theoremstyle{definition}

\newtheorem{defi/}{Definition}

\newenvironment{defi}
  {%
   \pushQED{\qed}\begin{defi/}}
  {\popQED\end{defi/}}

\DeclareMathOperator{\diam}{diam}

\DeclareMathOperator{\supp}{supp}
\renewcommand{\i}{{\mathrm{i}}}
\renewcommand{\d}{{\mathrm{d}}}
\newcommand{\e}{{\mathrm{e}}}
\newcommand{\bfr}{{\mathbf{r}}}

\newcommand{\vertiii}[1]{{\left\vert\kern-0.25ex\left\vert\kern-0.25ex\left\vert #1
    \right\vert\kern-0.25ex\right\vert\kern-0.25ex\right\vert}}

\theoremstyle{remark}

\newtheorem{rem/}[theorem]{Remark}

\newenvironment{rem}
  {%
   \pushQED{\qed}\begin{rem/}}
  {\popQED\end{rem/}}

\numberwithin{equation}{section}

\makeatletter
\renewcommand\subsection{\@startsection{subsection}{2}%
  \z@{.5\linespacing\@plus.7\linespacing}{-.5em}%
  {\normalfont\scshape}}
\makeatother

\begin{document}

\title{``Bootstrap domain of dependence'':
  bounds and time decay of solutions of
  the wave equation}
\author{Thomas G. Anderson}
\address{Applied \& Comp.\ Mathematics, California Institute of
Technology, Pasadena, CA}
\curraddr{Department of Mathematics, University of Michigan, Ann Arbor, MI}
\email{tganders@umich.edu}

\author{Oscar P. Bruno}
\address{Applied \& Comp.\ Mathematics, California Institute of
Technology, Pasadena, CA}
\email{obruno@caltech.edu}
\thanks{This work was supported by NSF and DARPA through contracts DMS-1714169
and HR00111720035, and the NSSEFF Vannevar Bush Fellowship under
contract number N00014-16-1-2808.}

\subjclass[2010]{Primary 35B40, 35L05, 45M05}
\date{\today}
\keywords{boundary integral equations, wave equation, decay, Huygens' principle}
\begin{abstract}
  This article introduces a novel ``bootstrap domain-of-dependence''
  concept, according to which, for all time following a given
  illumination period of arbitrary duration, the wave field scattered
  by an obstacle is encoded in the history of boundary scattering
  events for a time-length equal to the diameter of the obstacle,
  measured in time units. Resulting solution bounds provide estimates
  on the solution values in terms of a short-time history record, and
  they establish super-algebraically fast decay (i.e., decay faster
  than any negative power of time) for a wide range of scattering
  obstacles---including certain types of ``trapping'' obstacles whose
  periodic trapped orbits span a set of positive volumetric measure,
  and for which no previous fast-decay theory was available.  The
  results, which do not rely on consideration of the Lax-Phillips
  complex-variables scattering framework and associated resonance-free
  regions in the \emph{complex} plane, utilize only \emph{real-valued}
  frequencies, and follow from use of Green functions and boundary
  integral equation representations in the frequency and time domains,
  together with a certain $q$-growth condition on the frequency-domain
  operator resolvent.
\end{abstract}

\maketitle
\markright{\MakeUppercase{Bootstrap Domain-of-dependence bounds \& decay for wave
solutions}}

\section{Introduction\label{Introduction}}

We present a novel ``bootstrap domain-of-dependence'' concept
(bootstrap DoD) and associated bounds on solutions of the scattering
problem for the wave equation on an exterior region $\Omega^c$.  The
classical domain-of-dependence concept~\cite{John} for a space-time
point $(\bfr_0, T_0)$ concerns the initial and boundary values that
determine the free-space solution at that space-time point. The
bootstrap DoD ending at a given time $T_0$, in contrast, involves
field values on the scattering boundary
$\Gamma = \partial \Omega= \partial \Omega^c$ over the length of time
$T_*$ preceding $T_0$, where $T_*$ equals the amount of time that is
sufficient for a wave in free space to traverse a distance equal to
the diameter of the obstacle. As discussed in this paper, in absence
of additional illumination after time $t = T_0 - T_*$, the bootstrap
DoD completely determines the scattered field on $\Gamma$ for all
times $t\geq T_0$.  Further, solution bounds that result from the
bootstrap DoD approach provide estimates on the solution values in
terms of a short-time history record, and they establish
super-algebraically fast decay (i.e., decay faster than any negative
power of time) for a wide range of scattering obstacles---including
certain types of trapping obstacles for which no previous fast-decay
theory was available. These results, whose proofs do not rely on
consideration of the Lax-Phillips complex-variables scattering
framework and associated resonance-free regions in the \emph{complex}
plane, utilize only \emph{real-valued} frequencies, and follow from
use of Green functions and boundary integral equation representations
in the frequency and time domains, together with a certain $q$-growth
condition on the frequency-domain operator resolvent. As discussed in
\Cref{3d_decay_rmk_iii}, on the basis of some of the materials in
\Cref{sec:theory_part_i} and most of those in
\Cref{sec:theory_part_ii}, related but less informative
super-algebraically decaying solution bounds can also be obtained by
relying on the complete time-history of the incident field instead of
the bootstrap DoD.

(Estimates based on the {\em classical} concept of domain of
dependence have previously been provided, including, in particular, a
``domain-of-dependence inequality'' for the problem of scattering by
obstacles~\cite[Thm.\ 5.2]{Wilcox:62}. We note however, that that is
not a decay result and, in fact, it plays an important but very
different role: by establishing that at most exponential {\em growth}
can occur, it provides the necessary stability elements in a proof of
existence and uniqueness based on energy considerations.)

In addition to the decay-problem application, the bootstrap DoD
estimates introduced in this paper provide a valuable tool in
connection with the numerical analysis of certain frequency-time
``hybrid'' numerical methods introduced recently~\cite{Anderson:20}
for the time-domain wave equation, which proceed by Fourier
transformation from the time domain into the frequency domain.  These
hybrid solvers evaluate solutions of the wave equation by partitioning
incident wave fields into a sequence of smooth compactly-supported
wave packets followed by Fourier transformation in time for each
packet, and they incorporate a number of novel approaches designed to
accurately capture all high-frequency behavior---both in time, for
high frequencies, and in frequency, for large time. The overall
solution is then reconstructed as a sum of many wave equation
solutions, each with a distinct center in time. Unlike the
aforementioned complete time-history bounds, the bootstrap DoD bounds
introduced in this paper provide a natural theoretical basis for
efficient truncation of this sum while meeting a prescribed accuracy
tolerance.

Returning to the decay problem we note that, in contrast with previous
approaches, that typically rely on energy arguments and/or on analytic
continuation of the frequency-domain resolvent, the method proposed in
this paper is based on use of boundary integral equations for the
frequency-domain and time-domain problems (along with the Fourier
transform that relates them), and it characterizes the
multiple-scattering effects and trapping character of domain
boundaries (which are described by means of billiard-ball
trajectories, see Section~\ref{sec:detailed_review}) in terms of the
growth of the norm of the resolvent for the frequency-domain problem
as the {\em real} frequency $\omega$ grows without bound. In detail,
without recourse to complex-analytic methods, the new approach to the
decay problem proceeds on the basis of the bootstrap DoD formulation
introduced in Section~\ref{sec:theory_part_i}, which, as mentioned
above, captures the impact of the complete time history up to a given
time $t$ in terms of the history over the time interval, immediately
preceding $t$, of time-length $T_*$ required for free-space
propagation across a distance equal to the diameter of the
obstacle. The resulting bounds provide super-algebraically-fast
time-decay energy estimates (decaying faster than any negative power
of time) for a wide range of (both trapping and non-trapping) obstacles
satisfying a certain $q$-growth condition on the resolvent operator as
a function of the \emph{real} frequency $\omega$. In particular, this
theory establishes the first rapid decay estimates for specific types
of trapping obstacles, such as those depicted in
\Cref{fig:3d_connected_trapping}, which are not equal to unions of
convex obstacles---and whose periodic trapped orbits, in fact, span a
set of positive volumetric measure.

\subsection{Overview}\label{sec:overview}
This paper is organized as follows. A brief but somewhat detailed
overview of previous decay results for the obstacle scattering problem
is presented in \Cref{sec:detailed_review}.  Then, after preliminaries
presented in \Cref{sec:prelim}, \Cref{sec:theory_part_i} introduces
the bootstrap domain-of-dependence formulation and it presents
\Cref{3d_decay_thm}, which shows that for obstacles satisfying the
$q$-growth condition, if the Neumann trace is ``small'' on a (slightly
extended) bootstrap DoD time interval, then, absent additional
illumination during and after that interval, it must remain
``permanently small'', that is, small for all times subsequent to that
interval.  \Cref{sec:theory_part_ii} then extends the results of
\Cref{sec:theory_part_i}, establishing, in \Cref{3d_decay_thm_ii} and
Corollaries~\ref{decay_corr} and~\ref{decay_corr_energy}, various
super-algebraically fast-decay results, including super-algebraic decay
of the local energy (eq.~\eqref{local_energy} below) for all
obstacles satisfying the $q$-growth condition.

\subsection{Additional Background on Decay
  Theory}\label{sec:detailed_review}

In order to provide relevant background concerning decay estimates we
briefly review the literature on this long-standing problem, and we
note the significant role played in this context by the shape of the
obstacle $\Omega$. Previous studies of the decay problem establish
exponential decay of solutions for certain classes of domain
shapes---including star-shaped domains~\cite{Morawetz:61}, domains
that are ``non-trapping'' with respect to
rays~\cite{Morawetz:77,Melrose:79}, and unions of strictly convex
domains that satisfy certain spacing
criteria~\cite{Ikawa:82,Ikawa:88}. (A domain is non-trapping if each
billiard ball traveling in the exterior of $\Omega$, which bounces off
the boundary $\Gamma = \partial \Omega$ in accordance with the law of
specular reflection, and which starts within any given ball $B_R$ of
radius $R$ containing $\Omega$, eventually escapes $B_R$ with bounded
trajectory length~\cite{Melrose:79}.)

Early results~\cite{Morawetz:61,Lax:63,Morawetz:66}, obtained on the
basis of energy estimates in the domains of time and frequency,
establish exponential decay for ``star-shaped'' obstacles, that is,
obstacles $\Omega$ which, for a certain $\mathbf{r}_0 \in \Omega$,
contain the line segment connecting $\mathbf{r}_0$ and any other point
$\mathbf{r}\in\Omega$. As noted in~\cite{Lax:63,Lax:67}, exponential
decay generally implies analyticity of the resolvent operator in a
strip around the real $\omega$ frequency axis. A significant
generalization of these results was achieved in~\cite{Morawetz:77}
using a hypothesis somewhat more restrictive than the non-trapping
condition, while~\cite{Melrose:79} established
exponential
local-energy decay for all non-trapping obstacles.  All of these
results establish exponential decay
\begin{equation}\label{energy_decay}
  E(u, D, t) \le C\e^{-\alpha t} E(u, \infty, 0), \quad \alpha > 0,
\end{equation}
for the local energy
\begin{equation}\label{local_energy}
  E(u, D, t) = \int_D \left|\nabla u(\bfr, t)\right|^2 + \left|u_t(\bfr,
  t)\right|^2\,\d V(\bfr)
\end{equation}
contained in a compact region $D \subset \Omega^c$ in terms of the
energy $E(u, \infty, 0)$ contained in all of $\Omega^c$---the latter
one of which is, of course, conserved.  As reviewed in
Remark~\ref{ralston}, a uniform decay estimate of the
form~\eqref{energy_decay} cannot hold for trapping obstacles.

In this connection it is tangentially relevant to consider
reference~\cite{Datchev:12}, which establishes sub-exponential decay
for the problem of scattering by a globally defined smooth
potential. The method utilized in that work relies on use of simple
resolvent manipulations for the differential scattering operator
$P =-\frac{1}{\omega^2}\Delta + V(\bfr)$, where
$V\in C_0^\infty(\mathbb{R}^n)$ denotes the globally defined
potential. Such manipulations are not applicable in the context of the
Green function-based operators for the impenetrable-scattering
problem, for which, in particular, the frequency $\omega$ is featured
in the {\em Green function exponent} of the {\em integral} scattering
operator instead of a {\em linear factor} $\frac{1}{\omega^2}$ in the
corresponding {\em differential} operator.

As mentioned above, a decay estimate of the form~\eqref{energy_decay}
cannot hold for trapping obstacles. However, by relying on analytic
continuation of the frequency-domain resolvent into a strip
surrounding the real-axis in the complex frequency domain $\omega$,
exponential decay (of a different character than expressed
in~\eqref{energy_decay}; see \Cref{ralston}) has been established for
certain trapping geometries~\cite{Ikawa:82,Ikawa:88,Fahry:91}.  In
view of previous work leading to results of exponential decay, even
for trapping geometries, the question may arise as to whether the
results of super-algebraically fast convergence presented in this paper
could actually be improved to imply, for example, exponentially fast
decay for trapping geometries which merely satisfy the $q$-growth
condition. A definite answer in the negative to such a question is
provided in~\cite[Thm.\ 1]{Ikawa:85}. This contribution exhibits an
example for which, consistent with earlier general
suggestions~\cite[p.\ 158]{Lax:67}, the sequence of imaginary parts of
the pole-locations of the scattering matrix in the complex plane tends
to zero as the corresponding sequence of real parts tends to infinity;
clearly, such a domain cannot exhibit exponential decay in view of the
Paley-Wiener theorem~\cite[Thm.\ I]{PaleyWiener:34}. In detail, the
example presented in~\cite{Ikawa:85} concerns a domain $\Omega$
consisting of a union of two disjoint convex obstacles, where the
principal curvature of each connected component vanishes at the
closest point between the obstacles.  Reference~\cite{Burq:98}
provides a general-obstacle inverse polylogarithmic decay estimate
that is the only previous decay result applicable to such a trapping
domain. In view of this background it may be suggested that the
real-$\omega$ decay analysis presented in this paper provides
significant progress, as it establishes super-algebraic decay for a
wide range of obstacles not previously treated by classical scattering
theory, including the aforementioned vanishing-curvature
example~\cite{Ikawa:85} and the connected and significantly more
strongly trapping structures depicted in
\Cref{fig:3d_connected_trapping}, for which the trapped rays form a
set of positive volumetric measure.

The aforementioned references~\cite{Ikawa:82,Ikawa:88,Fahry:91}
establish exponential decay for wave scattering for certain trapping
structures consisting of unions of disjoint convex obstacles (but see
Remark~\ref{ralston}), and thereby answer in the negative a conjecture
by Lax and Phillips~\cite[p.  158]{Lax:67} according to which
exponential decay could not occur for any trapping structure (in view
of the Lax/Phillips conjectured existence, for all trapping obstacles,
of a sequence of resonances $\lambda_j$ for which
$\mathrm{Im}\,\lambda_j \to 0^-$ as $j \to \infty$). The trapping
structures with exponential decay are taken to equal a disjoint union
of two smooth strictly convex obstacles in~\cite{Ikawa:82}, and
otherwise unions of disjoint convex obstacles in~\cite{Ikawa:88}; in
all cases the geometries considered give rise to sets of trapping rays
spanning three-dimensional point sets of zero volumetric measure:
single rays in~\cite{Ikawa:82,Fahry:91}, and countable sets of
primitive trapped rays in~\cite{Ikawa:88} as implied by Assumption
(H.2) in that reference. To the authors' knowledge, these are the only
known results on fast decay of solutions in trapping
geometries. Despite these known exceptions to the Lax-Phillips
conjecture, it has been surmised~\cite{Ikawa:85} that ``it seems very
sure that the conjecture remains to be correct for a great part of
trapping obstacles,''---which would disallow exponential decay for
most trapping obstacles---thus providing an interesting perspective on
the main results presented in this paper---which, in particular,
establish super-algebraic decay for certain obstacles for which the set
of trapped rays span a set of positive measure.

\section{Preliminaries}\label{sec:prelim}
\subsection{Dirichlet problems for the three-dimensional wave-equation}\label{sec:prelim_wave}

We consider the problem of scattering of an incident field $b$ by an
(open) bounded obstacle $\Omega\subset \mathbb{R}^3$ with Lipschitz
boundary. More precisely, letting
$\Box = \frac{\partial^2}{\partial t^2} - c^2\Delta$ denote the
d'Alembertian operator with wave-speed $c > 0$ and given an open set
$\Omega^\textit{inc}$ containing the closure $\widebar{\Omega}$ of
$\Omega$, $\widebar{\Omega} \subset \Omega^\textit{inc}$, we study the
solution $u$ of the initial and boundary value problem
\begin{subequations}\label{eq:w_eq}
  \begin{align}
    \Box\,u(\mathbf{r}, t)
    &= 0,\quad\mbox{for}\quad (\bfr, t) \in \widebar{\Omega}^c \times (0, \infty),\label{eq:w_eq_a}\\
    u(\mathbf{r},0) &= \frac{\partial u}{\partial t}(\mathbf{r}, 0)
                      = 0,\quad\mbox{for}\quad\mathbf{r} \in \Omega^c,\\
    u(\mathbf{r}, t) &=  -\gamma^+ b(\mathbf{r}, t),
                       \quad\mbox{for}\quad(\mathbf{r},t)\in\Gamma\times (0,\infty),\label{eq:w_eq_c}
    \end{align}
  \end{subequations}
  on the complement $\widebar{\Omega}^c$ of $\widebar{\Omega}$, for a
  given incident-field function
\begin{equation}\label{eq:bdef}
b: \Omega^\textit{inc}\times\mathbb{R} \to \mathbb{R}\;\mbox{satisfying}\;
b \in C^2(\Omega^\textit{inc}\times\mathbb{R})\; \mbox{and}\; \Box b = 0\;\mbox{in}\; \Omega^\textit{inc}\times\mathbb{R},
\end{equation}
(cf.\ \Cref{rem:wave_eq_sob_assump} below), where
$\Gamma = \partial\Omega^c= \partial\Omega$ denotes the boundary of
the obstacle $\Omega$, and where $\gamma^+$ (\Cref{trace_def} in
\Cref{sob-boch}) denotes the exterior trace operator.  Compatibility
of the boundary values with the initial condition requires
$b(\bfr, 0) = \frac{\partial}{\partial t} b(\bfr, 0) = 0$ for $\bfr \in \Gamma$. In
fact, for compatibility with the integral equation
formulation~\eqref{eq:tdie_sl} below, letting
$\mathbb{R}_0^- =\{t\in\mathbb{R}\ : \ t\leq 0 \}$, throughout this
paper we assume that
\begin{equation}\label{all_t_b}
  b \in C^2(\Omega^\textit{inc} \times \mathbb{R})\cap L^2(\Omega^\textit{inc} \times \mathbb{R})\;\mbox{and}\; b(\bfr, t) = 0\;\mbox{for}\; (\bfr, t) \in \widebar{\Omega} \times \mathbb{R}_0^-.
\end{equation}
The solution $u$ is the ``scattered'' component of the total field
$u^\textit{tot} = u + b$; with these conventions, we clearly have
$\gamma^+ u^\textit{tot}(\mathbf{r}, t) = 0$ for $\mathbf{r}$ on $\Gamma$.

\begin{rem}\label{rem:w_eq_ivp_ibvp_equiv}
  In the case $\Omega^\textit{inc}= \mathbb{R}^3$, the function $v = u^\textit{tot} = u + b$, satisfies
  the initial-value problem
\begin{subequations}\label{eq:w_eq_ivp}
    \begin{align}
      \Box &v(\mathbf{r}, t) = 0,\quad\mbox{for}\quad (\bfr, t) \in
            \widebar{\Omega}^c\times (0, \infty),\label{eq:w_eq_ivp_a}\\
        &v(\mathbf{r},0) = g_0(\bfr),\, \frac{\partial
        v}{\partial t}(\mathbf{r}, 0)
            = g_1(\bfr),\quad\mbox{for}\quad\mathbf{r} \in
            \Omega^c,\label{eq:w_eq_ivp_b}\\
        &v(\mathbf{r}, t) = 0,
            \quad\mbox{for}\quad(\mathbf{r},t)\in\Gamma\times (0,\infty),\label{eq:w_eq_ivp_c}
    \end{align}
\end{subequations}
in that domain, where $g_0(\mathbf{r}) = u^\textit{tot}(\bfr, 0)$ and
$g_1(\mathbf{r}) = \frac{\partial}{\partial t} u^\textit{tot}(\bfr, 0)$. Conversely,
given smooth data $g_0$ and $g_1$ in all of
$\Omega^c \cap \Omega^\textit{inc}$, an incident field $b$ can be
obtained by first extending $g_0$ and $g_1$ to all of
space~\cite[Thm.\ 3.10]{Necas:11} and using the extended data as
initial data for a free-space wave equation with solution $b$. In
other words, the problems~\eqref{eq:w_eq} and~\eqref{eq:w_eq_ivp} are
equivalent.
\end{rem}

\begin{rem}
  The classical literature on decay rates for the wave
  equation~\cite{Lax:67} concerns problem~\eqref{eq:w_eq_ivp} with the
  additional assumption that $g_0$ and $g_1$ are compactly
  supported. In view of the strong Huygens principle~\cite{John},
  which is valid in the present three-dimensional context, the
  procedure described in \Cref{rem:w_eq_ivp_ibvp_equiv} translates the
  spatial compact-support condition on $g_0$ and $g_1$ into a temporal
  compact-support condition for the function $b$
  in~\eqref{eq:w_eq}---i.e.\ that the incident field vanishes on the
  boundary $\Gamma = \partial \Omega$ after a certain initial
  ``illumination time period'' has elapsed.
\end{rem}

\begin{rem}\label{rem:wave_eq_sob_assump}
  Even though the main problem considered in this paper,
  problem~\eqref{eq:w_eq}, is solely driven by the Dirichlet
  data~\eqref{eq:w_eq_c}, our analysis relies on assumptions that are
  imposed not only on the function $\gamma^+b$, but also on the values
  $\gamma^+\partial_\mathbf{n} b$ of the normal derivative of $b$ on
  $\Gamma$. Specifically, various results presented in this paper
  assume that, for an integer $s\geq 0$ the incident field $b$
  satisfies the ``$s$-regularity conditions''
  \begin{equation}\label{eq:gamma_Hs_assump}
    \gamma^+b \in H^{s+1}(\mathbb{R};\,L^2(\Gamma))\quad\mbox{and}\quad
    \gamma^+\partial_\mathbf{n} b \in H^s(\mathbb{R};L^2(\Gamma)),
  \end{equation}
  where, for an integer $p$, $H^p(\mathbb{R};\,L^2(\Gamma))$ denotes
  the Sobolev-Bochner space of order $p$ with values in $L^2(\Gamma)$,
  as defined in~\Cref{sob-boch}.  Clearly, for the most relevant
  incident fields $b$, such as those mentioned in
  \Cref{rem:w_eq_ivp_ibvp_equiv}, namely, solutions of the wave
  equation that are smooth and compactly supported in time over any
  compact subset of $\Omega^{\textit{inc}}$, the $s$-regularity
  conditions hold for all non-negative integers $s$.
\end{rem}

\subsection{Time-domain single layer potential}\label{sec:prelim_td_layer}

The boundary value problem~\cref{eq:w_eq} can be
reduced~\cite{HaDuong:86} to an equivalent time-domain integral
equation formulation, in which a boundary integral density
$\psi = \psi(\mathbf{r}, t)$ (defined to vanish for all $t < 0$) is
sought, that satisfies the spatio-temporal boundary integral equation
\begin{equation}\label{eq:tdie_sl}
  \left(S \psi \right)(\mathbf{r}, t) = \gamma^+ b(\mathbf{r}, t)\quad\mbox{for}\quad (\mathbf{r},
  t) \in \Gamma \times \mathbb{R}.
\end{equation}
Here, calling
\begin{equation}\label{eq:green_fnct_time}
  G(\mathbf{r}, t; \mathbf{r}', t') = \frac{\delta\left((t - t') - |\mathbf{r} -
  \mathbf{r}'|/c\right)}{4\pi |\mathbf{r} - \mathbf{r}'|}
\end{equation}
the Green function for the three-dimensional wave equation,
$S = \gamma^+ \mathscr{S}$ denotes the trace of the time-domain
single-layer potential
\begin{equation}\label{eq:single_layer_pot_time}
  (\mathscr{S}\mu)(\mathbf{r}, t) = \int_{-\infty}^t \int_\Gamma G(\mathbf{r}, t;
  \mathbf{r}', t') \mu(\mathbf{r}', t')\,\d\sigma(\mathbf{r}')\,\d
  t',\quad(\mathbf{r},
  t) \in \mathbb{R}^3 \times \mathbb{R}.
\end{equation}
Note that the single-layer potential and its restriction to the
boundary may be expressed without recourse to distributions in the
forms
\begin{equation}\label{eq:single_layer_pot_time_conv}
    (\mathscr{S}\mu)(\mathbf{r}, t) = \int_\Gamma \frac{\mu(\mathbf{r}', t - |\mathbf{r} - \mathbf{r}'|/c)}{4\pi|\mathbf{r} - \mathbf{r}'|}\,\d\sigma(\mathbf{r}'),\quad\bfr\in\mathbb{R}^3,
\end{equation}
and
\begin{equation}\label{eq:single_layer_op_time_conv}
  \left(S\mu\right)(\mathbf{r}, t) = \int_\Gamma \frac{\mu(\mathbf{r}', t - |\mathbf{r} - \mathbf{r}'|/c)}{4\pi|\mathbf{r} - \mathbf{r}'|}\,\d\sigma(\mathbf{r}'),\quad \mathbf{r} \in \Gamma,
\end{equation}
respectively.  Clearly, these operators are well-defined, for example,
for densities $\mu \in \mathcal{S}(\mathbb{R}; L^2(\Gamma))$, where
$\mathcal{S}(\mathbb{R}; L^2(\Gamma))$ denotes~\cite[Def.\
2.4.21]{Weis} the Schwartz space of smooth and rapidly decaying
functions of $t\in \mathbb{R}$ with values in $L^2(\Gamma)$.  The
potential~\eqref{eq:single_layer_pot_time_conv} is also well-defined
for functions $\mu$ in the space $L^2(\mathbb{R}; L^2(\Gamma))$, a
fact that follows easily for $\mathbf{r}\not\in\Gamma$, on account of
the smoothness of the Green-function kernel in the
integrand~\eqref{eq:single_layer_pot_time} for such values of
$\mathbf{r}$. As discussed in \Cref{sec:prelim_spaces}, the same is
true for $\mathbf{r}\in\Gamma$. More precisely, as shown in
\Cref{fd_layer_continuity_lemma}, the integral
operator~\eqref{eq:single_layer_op_time_conv} maps
$L^2(\mathbb{R}; L^2(\Gamma))$ continuously into itself.

As is well known~\cite{HaDuong:86}, problem~\eqref{eq:tdie_sl} admits
a unique solution. (In fact, \Cref{3d_decay_lemma_2ndkind_wellposed}
below shows that $\psi \in L^2(\mathbb{R}; L^2(\Gamma))$ for obstacles
$\Omega$ satisfying the $q$-growth condition (\Cref{q-nontrapp}).)
Once $\psi$ has been obtained, the solution $u$ of~\eqref{eq:w_eq} is
given by
\begin{equation}\label{eq:kirchhoff_3d_soft}
  u(\mathbf{r}, t) = \left(\mathscr{S} \psi \right)(\mathbf{r}, t), \quad \mathbf{r} \in
  \Omega^c.
\end{equation}
As is well known, further, the solution $\psi$ of~\eqref{eq:tdie_sl}
equals the Neumann trace of the solution $u$:
\begin{equation}\label{psi_dnu}
  \psi(\mathbf{r}, t) = \gamma^+ \frac{\partial
    u^\textit{tot}}{\partial \mathbf{n}}(\mathbf{r}, t),\quad\mathbf{r}\in\Gamma.
\end{equation}
In the functional setting utilized in this paper, the
result~\eqref{psi_dnu} follows directly from
\Cref{td_sl_lemma} below together with the corresponding
result~\cite[Thm.\ 2.44]{ChandlerWilde:12} for solutions of the
frequency-domain single-layer equation (which is uniquely solvable for
all frequencies except for the measure-zero set of square roots of
Laplace eigenvalues in the domain $\Omega$).

As shown in \Cref{decay_corr_energy}, in view
of~\eqref{eq:kirchhoff_3d_soft}, spatio-temporal estimates and
temporal decay rates for the density $\psi$ imply corresponding decay
properties for the energy $E(u, D, t)$ in~\Cref{local_energy} over any
given compact set $D \subset \Omega^c$.
\begin{rem}\label{operator_notat}
  Throughout this paper three different kinds of notation are used to
  denote the application of an operator to a function, namely, e.g. in
  the case of the operator in~\eqref{eq:tdie_sl},
  \begin{equation}
    \left(S \psi \right) = S[\psi] = S\psi.\qedhere
  \end{equation}
\end{rem}
\subsection{The Fourier transform and frequency-domain layer potentials}\label{sec:prelim_fd_layer}
As indicated above, this paper presents time-decay estimates on the
solutions $\psi$ of the integral equation problem~\eqref{eq:tdie_sl},
including results for certain classes of trapping obstacles. Our
analysis is based on consideration of frequency-domain counterparts of
problems~\eqref{eq:w_eq} and~\eqref{eq:tdie_sl}. On one hand, the
frequency-domain counterpart of~\eqref{eq:w_eq} a given frequency
$\omega$ is the Dirichlet problem for the Helmholtz equation with
wavenumber $\kappa = \kappa(\omega)=\omega/c$,
\begin{subequations}\label{helmholtz}
  \begin{alignat}{2}
    \Delta &U^f (\mathbf{r}, \omega)+\kappa^2(\omega) U^f (\mathbf{r},\omega)\label{helmholtz_eq}
    = 0,\quad&&\mbox{for}\quad \bfr \in \widebar{\Omega}^c,\\
    & U^f(\mathbf{r}, \omega) =  -\gamma^+ B^f(\mathbf{r}, \omega),
                       \quad&&\mbox{for}\quad\mathbf{r}\in\Gamma,
  \end{alignat}
\end{subequations}
with unknown and
incident field given by
\begin{equation}\label{freq_u_b}
  U^f=\mathcal{F} [u]\quad\mbox{and}\quad B^f=\mathcal{F} [b],
\end{equation}
respectively (see Remark~\ref{FT_conv} below). The frequency-domain
counterpart of~\eqref{eq:tdie_sl}, in turn, is the equation
$S_\omega \psi^f = \gamma^+ B^f$, where $S_\omega$ denotes the
frequency-domain single-layer operator introduced in
equation~\eqref{eq:single_layer_op} below. But the latter equation is
not uniquely solvable at some frequencies, and we thus use the
uniquely solvable frequency-domain combined-field integral equation
\begin{equation}\label{CFIE_direct}
  A_{\omega,\eta} \psi^f = \gamma^+ \partial_\mathbf{n} B^f - \i \eta \gamma^+ B^f
\end{equation}
for the same unknown, where the operator $A_{\omega,\eta}$ is
presented in
\Cref{Aop_def_eqn}. \Cref{3d_decay_lemma_2ndkind_wellposed} below
shows that, indeed, for each $\omega\in\mathbb{R}$, the solution
$\psi^f = \psi^f(\mathbf{r}, \omega)$ of \Cref{CFIE_direct} coincides
with the temporal Fourier transform of $\psi = \psi(\mathbf{r}, t)$:
$ \psi^f(\mathbf{r}, \omega) =\mathcal{F}[ \psi](\mathbf{r}, \omega)$,
$\mathbf{r}\in\Gamma$---where, using~\eqref{eq:fourier_transf} below,
the notational convention
$\mathcal{F}[ \psi](\mathbf{r}, \omega) = \mathcal{F}[\psi(\mathbf{r},
\cdot)](\omega)$ has been introduced.
\begin{rem}\label{FT_conv}
  Throughout this article the superscript $f$ is often used to
  emphasize the dependence of a given function on the temporal
  frequency $\omega$. The situation occurs frequently in connection
  with the use of the Fourier transform.  For example, for a function
  $h(t)$ of the time variable $t$, the Fourier transform of $h$ could
  be denoted by
\begin{equation}\label{eq:fourier_transf}
  H^f(\omega) = \mathcal{F} [h](\omega) = \int_{-\infty}^\infty h(t) \e^{-\i\omega t}\,\d t.
\end{equation}
Here $h$ and $H^f = \mathcal{F} [h]$ could denote either
complex-valued scalar functions or functions with values on a Banach
space $X$ over the complex numbers. In the latter case the integral on
the right-hand side of~\eqref{eq:fourier_transf} indicates integration
in the sense of Bochner~\cite{Hille,DunfordSchwartz,Weis}.
\end{rem}

Our analysis relates the decay problem under consideration to the
growth of norms of associated frequency-domain solution operators as
the frequency grows without bound. Motivated in part by work
concerning numerical analysis of integral equations, studies of such
norm growths, which provide an indicator of the energy-trapping
character of obstacles, have been undertaken over the last several
decades, and have resulted in frequency-growth estimates for
geometries that exhibit a variety of trapping
behavior~\cite{Popov:91,Cardoso:02,ChandlerWilde:09,BetckeChandlerWilde:10,Spence:16,Spence:20}. The
relevant frequency-domain operators are introduced in the following
definition.

\begin{defi}[Frequency-domain operators]\label{Aop_def}
  Let $\Omega$ denote a Lipschitz domain with boundary $\Gamma$.
  Then, calling $G_\omega$ the Green function for the Helmholtz
  equation~\eqref{helmholtz_eq} with wavenumber $\kappa(\omega) = \omega / c$,
  $G_\omega(\mathbf{r}, \mathbf{r}') = \frac{\e^{\i\frac{\omega}{c}
      |\mathbf{r} - \mathbf{r}'|}}{4\pi|\mathbf{r} - \mathbf{r}'|}$,
  we define the single-layer potential $\mathscr{S}_\omega$, and the
  single-layer and adjoint double-layer operators, $S_\omega$ and
  $K_\omega^*$ , respectively (see e.g.~\cite{ChandlerWilde:12}),
\begin{equation}\label{eq:single_layer_pot}
    (\mathscr{S}_\omega\mu)(\mathbf{r}) = \int_\Gamma G_{\omega}(\mathbf{r}, \mathbf{r'})
    \mu(\mathbf{r'})\,\d\sigma(\mathbf{r'}),\quad  \mathbf{r} \in \mathbb{R}^3,
  \end{equation}
\begin{equation}\label{eq:single_layer_op}
    (S_\omega\mu)(\mathbf{r}) = \int_\Gamma G_{\omega}(\mathbf{r}, \mathbf{r'})
    \mu(\mathbf{r'})\,\d\sigma(\mathbf{r'}),\quad  \mathbf{r} \in \Gamma, \quad \mbox{and}
  \end{equation}
\begin{equation}\label{eq:adjoint_double_layer_op}
  (K^*_\omega\mu)(\mathbf{r}) = \int_\Gamma \frac{\partial G_{\omega}(\mathbf{r},
    \mathbf{r'})}{\partial n(\mathbf{r})}
  \mu(\mathbf{r'})\,\d\sigma(\mathbf{r'}), \quad \mathbf{r} \in \Gamma.
  \end{equation}
  Well known expressions for the exterior and interior traces of the
  the single layer potential and its normal
  derivative~\cite[p. 219]{McLean} tell us that, for a given
  $\eta\in\mathbb{R}$,
\begin{equation}\label{jump_cond}
  \left(\gamma^\pm \partial_\mathbf{n} -
    \i\eta\gamma^\pm \right)(\mathscr{S}_\omega\mu) =  \left(\mp\frac{1}{2}I + K_\omega^* -
    \i\eta
    S_\omega\right).
\end{equation}
In what follows we utilize the interior-trace instance of
\eqref{jump_cond} for a real coupling parameter $\eta \ne 0$, and we thus define
the combined-field operator
\begin{equation}\label{Aop_def_eqn}
  A_{\omega,\eta} = \frac{1}{2}I + K_\omega^* - \i\eta
  S_\omega,\quad\eta\ne 0.
\end{equation}
For a given $\omega_0 > 0$ and for $\omega \ge 0$ we also define
\begin{equation}\label{Aomega_def_eqn}
  A_\omega = A_{\omega,\eta_0(\omega)} \quad\mbox{where}\quad \eta_0 (\omega)=
\begin{cases} 1, &~\mbox{if}~0 \le \omega < \omega_0\\
\omega, &~\mbox{if}~\omega \ge \omega_0,
\end{cases}
\end{equation}
where the dependence of $A_\omega$ on $\omega_0$ has been suppressed.
\end{defi}

\subsection{Time- and frequency-domain functional spaces}\label{sec:prelim_spaces}

This paper develops an $L^2(\Gamma)$-theory for time-domain
scattering. We rely, in part, on classical results for frequency-domain layer
potentials on Lipschitz domains. The associated literature is a storied
one, and it contains well known contributions as described
in~\cite[p.\ 209]{McLean}; the corresponding results required in this
paper are outlined in the following lemma.
\begin{lemma}\label{fd_layer_continuity_lemma}
  Let $\Gamma$ denote a Lipschitz boundary. Then, the operators
  $S_\omega$, $K_\omega^*$ and $A_\omega$ are continuous linear
  operators on $L^2(\Gamma)$ for each $\omega \ge 0$. Further, the
  frequency-domain operator $\widetilde S$ given by
  \begin{equation}\label{eq:Stilde_def}
    \widetilde S[\mu] (\omega,\mathbf{r}) = S_\omega[\mu](\mathbf{r}), \quad\mathbf{r}\in\Gamma,
  \end{equation}
  is a continuous operator on the space $L^2(\mathbb{R}; L^2(\Gamma))$
  (i.e., the space $H^r(\mathbb{R}; H^s(\mathcal{U})$ in
  Definition~\ref{def:sob_bochner} with $r=s=0$ and
  $\mathcal{U} = \Gamma$):
  \begin{equation}\label{SL_cont_Bochner}
    \widetilde S: L^2(\mathbb{R}; L^2(\Gamma)) \to L^2(\mathbb{R}; L^2(\Gamma)).
  \end{equation}
  Finally, for $\mathrm{Re}(\eta) \ne 0$ the operator
  $A_{\omega,\eta}$ is invertible for all $\omega \ge 0$.
\end{lemma}
\begin{proof}
  Proofs for the results concerning the mapping properties of the
  operators $S_\omega$, $K_\omega^*$ and $A_\omega$ can be found
  in~\cite{Costabel:88} and~\cite[Thm.\ 6.12 and pp. 209]{McLean}; and
  the invertibility of $A_\omega$ is established \ e.g. in~\cite[Thm.\
  2.27]{ChandlerWilde:12}. (The continuity result for $S_\omega$ with
  $\omega = 0$, which easily implies the result for $\omega\ne 0$, was
  initially established in~\cite{Verchota:84}.)

  The continuity of the operator $\widetilde S$, finally, can be
  obtained from the norm-boundedness relation
  \begin{equation}\label{single_layer_uniform}
    \|S_\omega\|_{L^2(\Gamma) \to L^2(\Gamma)} \le C\quad
    \mbox{for all}\quad \omega \in\mathbb{R},
  \end{equation}
  which is given e.g. in~\cite[Thm.\ 3.3]{ChandlerWilde:09} for
  $\omega \geq 0$ and which follows also for $\omega<0$ in view of the
  Hermitian symmetry relation
  $\overline{S_{\omega}\mu^f(\mathbf{r}, \omega)} =
  S_{-\omega}\overline{\mu^f(\mathbf{r}, \omega)}$.  Indeed, for
  $\mu^f \in L^2(\mathbb{R}; L^2(\Gamma))$,
  equation~\eqref{single_layer_uniform} tells us that
  \begin{equation}
    \int_\Gamma \left|\int_\Gamma \frac{\e^{\i\omega |\mathbf{r} - \mathbf{r}'|}}{4\pi
        |\mathbf{r} -
        \mathbf{r}'|}\mu^f(\mathbf{r}', \omega) \,\d\sigma(\mathbf{r}')\right|^2\,\d\sigma(\mathbf{r})
    \leq C \left\| \mu^f(\cdot, \omega)\right\|_{L^2(\Gamma)}^2,
  \end{equation}
  and, thus, integrating with respect to $\omega$, we obtain
  \begin{equation}\label{S_in_L2}
    \| \widetilde{S}\mu^f\|^2_{L^2(\mathbb{R}; L^2(\Gamma))}  \le C \int_{-\infty}^\infty \left\| \mu^f(\cdot, \omega)\right\|_{L^2(\Gamma)}^2\,\d \omega = C\left\| \mu^f\right\|_{L^2(\mathbb{R}; L^2(\Gamma))}^2,
  \end{equation}
  establishing the desired continuity property for the operator
  $\widetilde{S}$.
\end{proof}
As indicated in the following Lemma, the time-domain single-layer
operator is also continuous in the space
$L^2(\mathbb{R}; L^2(\Gamma))$, and the frequency- and time-domain
single-layer operators are related by the Fourier transform.
\begin{lemma}\label{td_sl_lemma}
  The time-domain single-layer boundary integral
  operator~\cref{eq:single_layer_op_time_conv} is well defined for
  $\mu$ in the function-valued Schwartz space
  $\mathcal{S}(\mathbb{R},L^2(\Gamma))$ of infinitely smooth and
  rapidly decaying functions on the real line, and it may be extended
  to a continuous operator
  \begin{equation}\label{S_oper_bochner}
    S: L^2(\mathbb{R}; L^2(\Gamma)) \to L^2(\mathbb{R}; L^2(\Gamma)).
  \end{equation}
  Further, for each $\mu\in L^2(\mathbb{R}; L^2(\Gamma))$ the temporal
  Fourier transform of the operator~\eqref{S_oper_bochner} applied to
  $\mu$ equals the frequency domain single layer
  operator~\eqref{SL_cont_Bochner} applied to the Fourier transform of
  $\mu$:
  \begin{equation}\label{eq:single_layer_op_ii}
    \mathcal{F}\left [ S[\mu]\right] = \widetilde S\left[\mathcal{F}[\mu]\right].
  \end{equation}
\end{lemma}
\begin{proof}
See \Cref{app_c}.
\end{proof}

\subsection{Growth of operators norms}\label{sec:prelim_norm_growth}

This article utilizes a certain ``$q$-growth condition'' which, for a
given $q\in \mathbb{R}$, may or may not be satisfied by a given
obstacle $\Omega$---namely, that a constant $C$ exists such that, for
all $\omega\in\mathbb{R}$, the relation~\eqref{eq:q-nontrapp} below
holds. The $q$-growth condition is fundamentally a statement on the
high-frequency character of the operator $A_\omega^{-1}$. Further, the
$q$-growth condition is indeed a condition on the domain $\Omega$
which is independent on the constant $\omega_0$
in~\eqref{Aomega_def_eqn}. This can be verified by noting that, given
any compact one-dimensional intervals $I_\eta$ and $I_\omega$, where
$I_\eta$ is bounded away from $0$, the norm
$\norm{A_{\omega,\eta}^{-1}}_{L^2(\Gamma)\to L^2(\Gamma)}$ is
uniformly bounded for $(\eta,\omega)\in I_\eta\times I_\omega$. The
latter property, in turn, can be established on the basis of the
theory~\cite{ChandlerWilde:12,McLean} for the combined field operator
on a Lipschitz domain, together with a compactness argument based on
Taylor expansions of the oscillatory exponential factor of the Green
function as well as a Neumann series for the resolvent around each
pair $(\eta,\omega)\in I_\eta\times I_\omega$.
\begin{defi}[$q$-growth condition]\label{q-nontrapp} For real $q$
  and $\omega_0 >0$, a Lipschitz domain $\Omega$ and its boundary
  $\Gamma$ are said to satisfy the $q$-\emph{growth} condition iff for
  all $\omega \ge 0$ the operator $A_\omega$ in~\eqref{Aomega_def_eqn}
  satisfies the bound
  \begin{equation}\label{eq:q-nontrapp}
    \norm{A_\omega^{-1}}_{L^2(\Gamma)\to L^2(\Gamma)} \le C (1 +
    \omega^2)^{q/2}
  \end{equation}
  for some constant $C > 0$, or, equivalently, the bounds
  \begin{equation}\label{eq:q-nontrapp-cases}
    \left\|A_\omega^{-1}\right\|_{L^2(\Gamma)\to L^2(\Gamma)} \le
    \begin{cases} C_1, &~\mbox{if}~\omega \le \omega_0\\
      C_2 \omega^q, &~\mbox{if}~\omega > \omega_0
    \end{cases}
  \end{equation}
  for certain constants $C_1 > 0$ and $C_2 > 0$.
\end{defi}
Per~\cite[Lem.\ 6.2]{Spence:20}, polynomially growing bounds on the
norm of this inverse operator in the high-frequency regime can be
obtained on the basis of corresponding polynomially growing bounds as
$\omega\to\infty$ on the norm of the resolvent operator
$(\Delta + \omega^2)^{-1}$ outside $\Omega$, with zero Dirichlet
boundary conditions on $\Gamma$. The results of the present article
thus imply super-algebraically-fast local energy decay of solutions to
\Cref{eq:w_eq} for any Lipschitz domain for which such real-axis
high-frequency resolvent bounds for the Helmholtz operator can be
established.

\begin{rem}\label{3d_decay_generalization}
  It is known that the $q$-growth condition is satisfied with a
  variety of $q$-values for various classes of
  obstacles~\cite{ChandlerWilde:08,Spence:14,Spence:16,Spence:20,ChandlerWilde:09}. For
  example, reference~\cite[Thm.\ 1.13]{Spence:16} shows that a smooth
  \emph{non-trapping} obstacle satisfies the $q$-growth condition with
  $q = 0$. A related $q = 0$ result is presented
  in~\cite{ChandlerWilde:08} for merely Lipschitz domains, but under
  the stronger assumption that the obstacle is star-shaped.
  Reference~\cite{Spence:20} shows that for ``hyperbolic'' trapping
  regions (in which all periodic billiard ball trajectories are
  unstable), a merely logarithmic growth in $\omega$ results, while
  for certain ``parabolic'' trapping regions, stronger ($q = 2$ or
  $q = 3$) growth takes place. It is also known that much more
  strongly-trapping obstacles exist, including obstacles for which
  exponentially-large inverse operator norms
  $\left\|A_\omega^{-1}\right\|_{L^2(\Gamma)\to L^2(\Gamma)}$
  occur~\cite{Cardoso:02, BetckeChandlerWilde:10, Lafontaine:20} (and
  which, therefore, do not satisfy the $q$-growth condition for any
  value of $q$).
\end{rem}

\begin{rem}\label{connected-trapping}
  Existing results can be used to obtain examples of connected
  trapping obstacles containing cavities which satisfy the $q$-growth
  condition in \Cref{q-nontrapp} for some value of $q$. For example,
  the obstacle
  \[
\Omega = [-1,1]\times[-1,1]\times[-1,1]\setminus [-1/2,1/2]\times[-1/2,1/2]\times[0,1]
    \]
    (a cube with a smaller cube removed from one of its faces, as
    displayed in the left panel of \Cref{fig:3d_connected_trapping}),
    is an $(R_0, R_1, a)$ parallel trapping obstacle in the sense
    of~\cite[Def.\ 1.9]{Spence:20}, for $R_1 > \e^{1/4} R_0$,
    $R_0 \ge \sqrt{3/2}$, and $a = 1$.  According to~\cite[Cor.\ 1.14
    and Rem.\ 1.16]{Spence:20}, $\Omega$ satisfies the $q$-growth
    condition with $q=3$.  Smoothing of the corners of this obstacle
    results in a connected trapping obstacle that satisfies a
    $q$-growth condition with $q=2$.
\end{rem}

\begin{figure}
  \centering
  \includegraphics[width=0.49\textwidth]{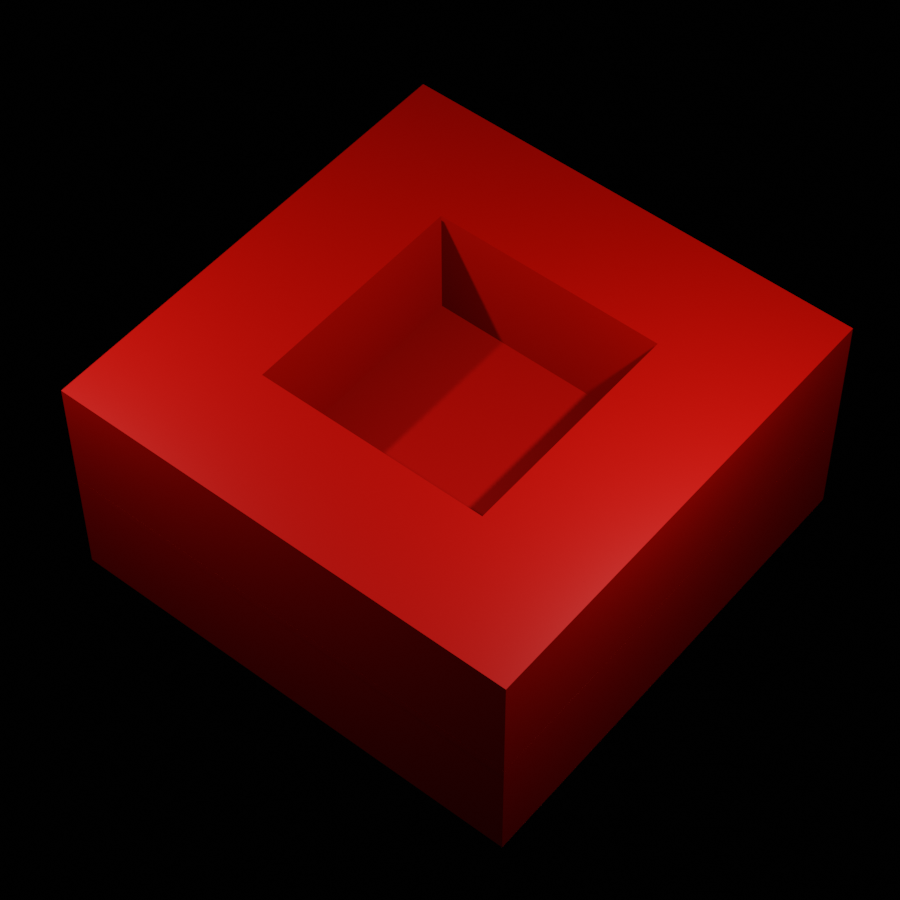}
  \includegraphics[width=0.49\textwidth]{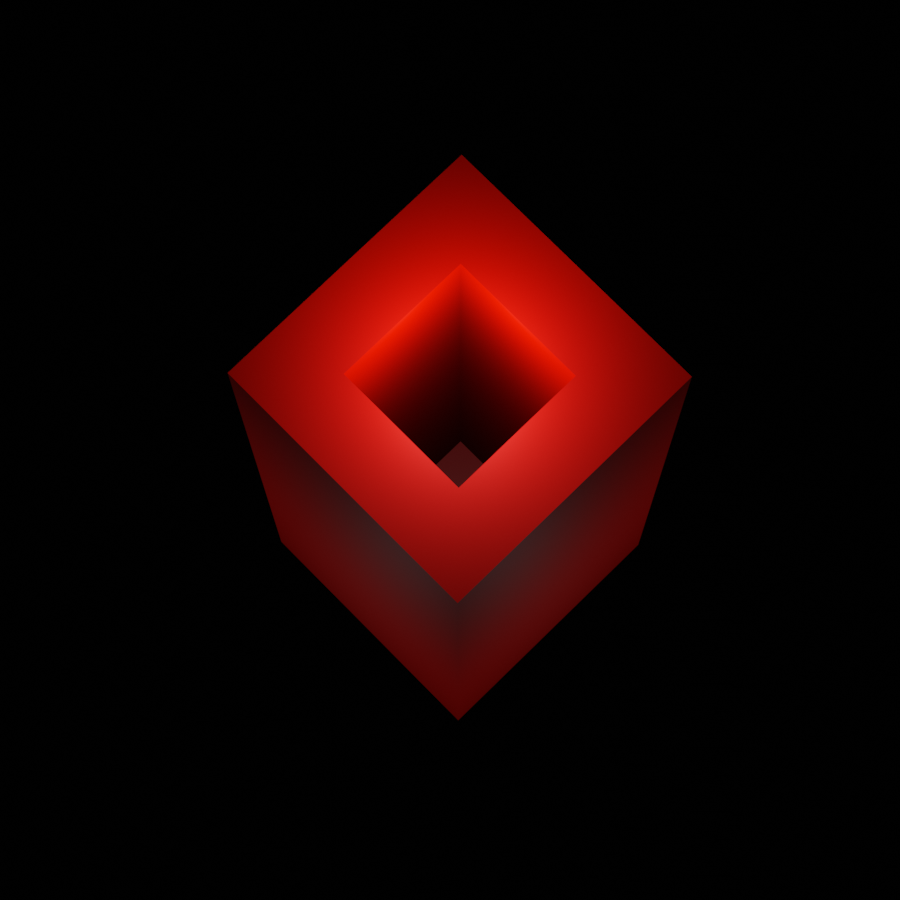}
  \caption[Example of connected trapping obstacles that satisfy a
  $q$-growth condition and for which wave equation decay rates are
  established.]{Examples of connected trapping obstacles that satisfy
    the $q$-growth condition ($q = 3$) and for which
    super-algebraically-fast wave equation time decay rates are
    established in this article. Left: Visualization of the obstacle
    mentioned in \Cref{connected-trapping}: a connected trapping
    obstacle satisfying the $q$-growth condition of \Cref{q-nontrapp}
    for $q=3$. Right: A trapping obstacle containing a deep cavity
    (arbitrary cavity depths are admissible), that satisfies the
    $q$-growth condition with the same value of
    $q$.}\label{fig:3d_connected_trapping}
\end{figure}

\section{Uniform ``bootstrap domain-of-dependence'' boundary density estimates}\label{sec:theory_part_i}

This section presents uniform-in-time ``bootstrap
domain-of-dependence'' estimates (\Cref{3d_decay_thm}) on the solution
$\psi$ of the time-domain boundary integral
equation~\cref{eq:tdie_sl}. In particular, these estimates show that
if $\psi$ is ``small'' throughout $\Gamma$ for any
domain-of-dependence time interval of length
\begin{equation}\label{Tstar_def}
 T_* \coloneqq \diam(\Gamma) / c = \max_{\mathbf{r},\mathbf{r}' \in \Gamma} |\mathbf{r} - \mathbf{r}'|/c
\end{equation}
after the $\Gamma$-values of the incident field have been turned off,
then $\psi$ will remain small for all time thereafter. At a
fundamental level, the domain-of-dependence analysis presented in this
section exploits an interesting property of
equation~\eqref{eq:tdie_sl} (made precise in
Lemmas~\ref{3d_decay_lemma_2ndkind} and~\ref{3d_decay_thm_h_equiv}),
namely, that, after the incident field $b(\mathbf{r}, t)$ has been
turned off permanently throughout $\Gamma$, the values of the solution
$\psi(\mathbf{r}, t)$ over any given domain-of-dependence time length
$T^*$ determine the solution uniquely for all subsequent times: the
boundary integral density over the interval $I_{T}$ encodes the
necessary information to reproduce all future scattering events and,
in particular, it encapsulates the effect of all previously imposed
boundary values, over time periods of arbitrarily long duration.

For technical reasons we utilize domain-of-dependence intervals
$I_{T}$, as detailed in \Cref{domainofdep}, of a length slightly
larger than $T^*$---larger by a small amount $2\tau>0$. Any positive
value of $\tau$ can be used, and the selection only affects the
multiplicative constants and integration domains in the main
domain-of-dependence estimates presented in this paper. The rest of
this section is organized as follows. Definitions~\ref{domainofdep}
and~\ref{timewinddens} lay down the conventions necessary to
subsequently state and prove the main result of the section, namely,
\Cref{3d_decay_thm}. After these definitions, the statement of the
theorem is introduced.  Lemmas~\ref{3d_decay_lemma_2ndkind_wellposed}
through~\ref{sob_lemma} then establish a series of results required in
the proof of the theorem, and then, concluding the section, the
theorem's proof is presented.

\begin{defi}[Domain-of-dependence interval]\label{domainofdep}
  Using the definition~\eqref{Tstar_def}
  $T_* \coloneqq \diam(\Gamma) / c$ for a given Lipschitz boundary
  $\Gamma$, and given real numbers $T$ and $\tau > 0$, the time
  interval
  \begin{equation}\label{time_interval_def}
    I_{T} = I_{T,T_*,\tau}= [T - T_* - 2\tau, T),
  \end{equation}
  will be referred to as the $\tau$-extended ``domain-of-dependence''
  interval ($\tau$DoD) relative to the ``observation'' time $T$.  (As
  suggested by the notations in~\eqref{time_interval_def}, the
  dependence of $I_T$ on the parameters $T_*$ and $\tau$ will not be
  made explicit in what follows.)
\end{defi}

\begin{rem}\label{ITvsI0}
  For future reference we mention here that the interval $I_T$ plays
  two important roles in our analysis. On one hand, this interval
  figures prominently, for a given value $T=T_0>0$, in the statements
  of the two main theorems of this paper, namely,
  Theorems~\ref{3d_decay_thm} and~\ref{3d_decay_thm_ii}: for any fixed
  given time $T_0>0$, the estimates provided by these theorems are
  valid for times past the upper endpoint $T_0$ of the interval
  $I_{T_0}$, provided the incident field vanishes at all times after
  the lower endpoint $T_0 - T_* - 2\tau$ of $I_{T_0}$. On the other
  hand, in order to adequately translate the fast oscillations that
  take place in inverse Fourier-transform integrands for large times
  into the decay rates claimed in \Cref{3d_decay_thm_ii}, a
  time-recentering technique is utilized that is embodied in
  \Cref{rem:breve} and \Cref{decay_estimate_L2}, and which is
  exercised in the portions of the proof of \Cref{3d_decay_thm_ii}
  containing equations~\eqref{rchipsi_L2_estimate},
  \eqref{L2_deriv_rchidecay} and~\eqref{Hp_tdecay_ii}. This
  time-recentering technique is based on consideration of the interval
  $I_T$, but, this time, with $T = 0 < T_0$. In this manner, the idea
  of recentering the problem in time around $t=0$, that is exploited
  in~\cite{Anderson:20} to reduce oscillations for algorithmic
  optimization purposes, is manifested in the present contribution in
  the use of the recentered interval $I_0$ equal to $I_T$ with $T=0$,
  to properly account for fast oscillations, as functions of
  frequency, that are observed for large time in frequency-domain
  inverse Fourier-transform integrands. From an analytical point of
  view, the time-recentering of frequency-domain integrals to the
  interval $I_0$ allows us to usefully exploit the frequency-time
  isometry inherent in the Plancherel theorem, to produce Sobolev
  estimates, but without incurring the uncontrollably large
  derivatives of the Fourier transform integrands with respect to
  frequency that result as $t$ grows, and which, instead of the
  Sobolev estimates provided in Lemma~\ref{Rderiv_sobolev_bound},
  would yield Sobolev estimates with constants $C$ that grow without
  bound as $T_0$ grows.
\end{rem}

\begin{defi}[Time-windowed solutions]\label{timewinddens}
  For a given solution $\psi$ of~\eqref{eq:tdie_sl} and a given
  $\tau$DoD interval $I_{T}$, using smooth non-negative window
  functions
  \begin{equation}\label{wtau_def}
  w_-(t) = \begin{cases} &1\mbox{~for~} t < -\tau\\
  &0\mbox{~for~} t \ge 0
  \end{cases},\quad
  w_+(t) = \begin{cases} &0\mbox{~for~} t < -\tau\\
  &1\mbox{~for~} t \ge 0,
  \end{cases}
  \end{equation}
  defined over the real line, which satisfy the ``Partition-of-unity''
  relation $w_- + w_+ = 1$, we define the auxiliary densities
  \begin{equation}\label{psipm_def}
    \psi_{-,T}(\bfr, t) = w_-(t - T)\psi(\bfr,
    t),\quad\psi_{+,T}(\bfr, t) = w_+(t - T)\psi(\bfr, t),
  \end{equation}
  whose temporal supports satisfy
  $\supp \psi_{-,T} \subset (-\infty, T]$ and
  $\supp \psi_{+,T} \subset [T -\tau, \infty)$, and for which the
  relation
\begin{equation}\label{pou_decomp}
  \psi(\cdot, t) = \psi_{-,T}(\cdot, t) + \psi_{+,T}(\cdot, t)
\end{equation}
holds for all real values of the time variables $t$ and $T$. We
further define
  \begin{equation}\label{psistar_def}
    \psi_{*,T}(\mathbf{r}, t) = w_+(t - T + T_* + \tau)w_-(t - T)\psi(\mathbf{r}, t),
  \end{equation}
  which is nonzero only in the interior of $I_{T}$.
\end{defi}

The smooth temporal decompositions introduced in \Cref{timewinddens}
play central roles in the derivations of the uniform-in-time bounds
and decay estimates presented in this paper. Noting that $\psi$ is
identical to $\psi_{+,T_0}$ for $t > T_0$, to produce such bounds and
estimates we first relate $\psi_{+,T_0}$ to $\psi_{-,T_0}$ and thereby
obtain a norm bound for $\psi_{+,T_0}$ over the slightly larger time
interval $[T_0 - \tau, \infty)$---which then yields, in particular,
the desired domain-of-dependence estimate for $\psi$ on the interval
$[T_0 , \infty)$.  The necessary bounds for $\psi_{+,T_0}$ are
produced via Fourier transformation into the frequency domain in
conjunction with use of the $q$-growth condition introduced in
\Cref{q-nontrapp}. A similar approach is followed for the
time-derivatives of the incident field data, leading to bounds in
temporal Sobolev norms of arbitrary orders, and, in particular, via
Sobolev embeddings, to uniform bounds in time.  As suggested above,
the approach intertwines the frequency and time domains, and it thus
incorporates frequency-domain estimates while also exploiting the time
domain Huygens' principle, cf.\ for example the
relations~\cref{h_equiv} and~\cref{HtoUbound0}.

The main result of this section, \Cref{3d_decay_thm}, which is stated
in what follows, relies on the definition of Sobolev-Bochner spaces
presented in \Cref{sob-boch}.
\begin{theorem}\label{3d_decay_thm}
  Let $p$ and $q$ denote non-negative integers, let $T_0 > 0$ and
  $\tau > 0$ be given, and assume (i) $\Gamma$ satisfies the
  $q$-growth condition (\Cref{q-nontrapp}); (ii) The incident field
  satisfies the $s$-regularity conditions~\eqref{eq:gamma_Hs_assump}
  with $s = p + 2q + 1$; and, (iii) The incident field
  $b = b(\bfr, t)$ satisfies \Cref{all_t_b} and it vanishes for
  $(\bfr, t) \in \widebar{\Omega} \times \left\lbrace I_{T_0} \cup
    [T_0, \infty)\right\rbrace$, with $I_{T_0}$ as in
  \Cref{domainofdep}. Then, the solution $\psi$ of \Cref{eq:tdie_sl}
  satisfies both the $H^p$ estimate
\begin{equation}\label{density_Hp_time_bound}
  \left\|\psi\right\|_{H^p([T_0, \infty);\,L^2(\Gamma))} \le
  C(\Gamma, \tau, p) \left\|\psi\right\|_{H^{p+q+1}(I_{T_0};\,L^2(\Gamma))}
  < \infty,
\end{equation}
and, in the case $p=1$ the time-uniform estimate
\begin{equation}\label{density_unif_time_bound}
  \left\|\psi(\cdot, t)\right\|_{L^2(\Gamma)} \le
  C(\Gamma, \tau)\left\|\psi\right\|_{H^{q+2}(I_{T_0};\,L^2(\Gamma))}\quad \mbox{for}\quad t > T_0,
\end{equation}
where the constants $C = C(\Gamma, \tau, p)$ and $C = C(\Gamma, \tau)$
are independent of $T_0>0$.
\end{theorem}

\begin{rem}\label{laplace_estimate_remark}
  Several results in the present contribution, including
  \Cref{3d_decay_thm}, \Cref{int_eq_pderiv}, \Cref{3d_decay_thm_ii}
  and \Cref{decay_estimate_L2}, utilize the $q$-growth condition-based
  estimate provided by \Cref{3d_decay_lemma_2ndkind_wellposed} below
  to conclude that $\psi \in H^p(\mathbb{R}; L^2(\Gamma))$ for certain
  values of $p$, and thus guarantee the less stringent condition
  $\psi \in H^p([\xi_1, \xi_2]; L^2(\Gamma))$ (for certain values
  $-\infty < \xi_1< \xi_2 < \infty$) that is actually needed to make
  the results meaningful. We note, however, that the regularity
  assumptions in these theorems and lemmas can be relaxed by using,
  instead, Laplace-domain bounds for the purpose of establishing the
  aforementioned less stringent condition. Indeed, in contrast to
  \Cref{3d_decay_lemma_2ndkind_wellposed}, on the basis of
  Laplace-domain estimates (in particular~\cite[Thm.\ 4.2]{Monk:14} in
  conjunction with~\cite[Lem.\ 2]{Chen:10}, cf.\
  also~\cite{HaDuong:86,Lubich:94}) such bounds show that if $b$ is
  causal and satisfies
    \[
        \gamma^+ b \in H^{r+1}([0, T]; L^2(\Gamma))\quad\mbox{and}\quad \gamma^+ \partial_\mathbf{n} b \in H^r([0, T]; L^2(\Gamma)),
    \]
    then $\psi \in H^r([0, T]; L^2(\Gamma))$. While we eschewed use of
    these bounds to achieve a simpler and more self-contained
    presentation, we note here how the regularity assumptions of the
    various lemmas and theorems would be relaxed by application of
    these auxiliary results. The $s$-regularity assumptions in
    \Cref{3d_decay_thm}, \Cref{int_eq_pderiv}, \Cref{3d_decay_thm_ii}
    and \Cref{decay_estimate_L2} would be relaxed, respectively, to
    requirements of $s$ regularity with $s = p + q + 1$, $s = p + 1$,
    $s = p + (n+1)(q+1)$, and $s = (n+1)(q+1)$.  We additionally note
    that none of the estimates are quantitatively affected---these
    assumptions are only used to ensure that the right-hand side
    quantities in the desired estimates are finite.  We finally
    mention that all of the assumptions mentioned here, including the
    original and the relaxed assumptions, are immediately satisfied
    for data $b$ that is smooth in time---see
    \Cref{rem:wave_eq_sob_assump}.
\end{rem}

The proof of \Cref{3d_decay_thm} is deferred to the end of this
section, following a series of eight preparatory Lemmas (the first and
last of which are Lemmas~\ref{3d_decay_lemma_2ndkind_wellposed} and
\ref{int_eq_pderiv}). The first of these lemmas relates the size of
the solution of equation~\eqref{eq:tdie_sl} to the size of the imposed
incident fields $b(\mathbf{r}, t)$ in various norms.

\begin{rem}\label{negative_freq}
  According to~\eqref{eq:bdef}, and without loss of generality,
  throughout this paper the incident field $b(\mathbf{r}, t)$ is
  assumed to be real-valued. It follows that the solution
  $\psi(\mathbf{r}, t)$ of \Cref{eq:tdie_sl} is real-valued, which
  implies that its Fourier transform $\psi^f(\mathbf{r}, \omega)$
  satisfies the Hermitian symmetry relation
  \begin{equation}\label{negative_freq_eq}
    \psi^f(\mathbf{r}, -\omega) = \overline{\psi^f(\mathbf{r}, \omega)}.
  \end{equation}
  Our studies of frequency-domain operator norms are therefore
  restricted to the range $\omega \geq 0$, as is common practice in
  the mathematical frequency-domain scattering
  literature~\cite{Spence:20,Spence:16,BetckeChandlerWilde:10}. Such
  symmetry relations apply to other quantities defined above and in
  what follows, including $\psi_{\pm,T}(\mathbf{r}, t)$,
  $\psi_{*,T}(\mathbf{r}, t)$, $h_T(\mathbf{r}, t)$, etc., that are
  defined in terms of the real-valued density $\psi(\mathbf{r}, t)$.
\end{rem}

\begin{rem}\label{rem:tilde_1}
  It will be necessary in what follows (specifically, in the proofs of
  the $H^p$ estimates in \Cref{3d_decay_thm} and
  \Cref{3d_decay_thm_ii} for $p > 0$, and associated preparatory
  lemmas), to allow for right-hand sides other than the specific
  incident-field function $b(\mathbf{r}, t)$ present in the Dirichlet
  problem \Cref{eq:w_eq} and the integral equation formulation
  \Cref{eq:tdie_sl}, and we thus consider scattering problems of an
  incident field function $\widetilde{b}(\mathbf{r}, t)$, similar to
  $b$, for which we have a corresponding integral equation formulation
\begin{equation}\label{eq:tdie_sl_generic}
    S\widetilde{\psi}(\mathbf{r}, t) = \gamma^+ \widetilde{b}(\mathbf{r}, t)\quad\mbox{for}\quad (\mathbf{r}, t) \in
    \Gamma\times\mathbb{R},
\end{equation}
with solution $\widetilde{\psi}$, where $S$ denotes the single layer
integral operator \Cref{eq:single_layer_op_time_conv}. Generalizing
\Cref{all_t_b}, and letting
$\mathbb{R}_\alpha^- =\{t\in\mathbb{R}\ : \ t\leq \alpha \}$,
throughout the remainder of this paper we assume that, for some
$\alpha \in \mathbb{R}$,
\begin{equation}\label{all_t_b_generic}
     \widetilde{b} \in C^2(\Omega^\textit{inc} \times \mathbb{R})\cap L^2(\Omega^\textit{inc} \times \mathbb{R})\;\mbox{and}\;  \widetilde{b}(\bfr, t) = 0\;\mbox{for}\; (\bfr, t) \in \widebar{\Omega} \times \mathbb{R}_\alpha^-
\end{equation}
(so that, in particular, $\widetilde{b}$ is square-integrable as a
function of time for $-\infty < t < \infty$). We also define
$\widetilde{\psi}_{-,T}$, $\widetilde{\psi}_{+,T}$, and
$\widetilde{\psi}_{*,T}$ paralleling the definitions of $\psi_{-,T}$,
$\psi_{+,T}$, and $\psi_{*,T}$, respectively, in
\Cref{timewinddens}. Similarly, with reference to~\eqref{freq_u_b}, we
write
\begin{equation}\label{freq_u_b_tilde} \quad \widetilde{B}^f =\mathcal{F}
  [\widetilde{b}].\qedhere
\end{equation}
\end{rem}

\begin{lemma}[Well-posedness of the time-domain integral equation~\Cref{eq:tdie_sl_generic}]\label{3d_decay_lemma_2ndkind_wellposed}
  Let $p$ and $q$ denote non-negative integers, and assume that the
  obstacle $\Omega$ satisfies the $q$-growth condition.  Let
  $\Omega^\textit{inc}$ denote an open domain containing
  $\widebar{\Omega}$, and assume the function $\widetilde{b}$ is a
  solution to the wave equation in $\Omega^\textit{inc}$, which
  satisfies~\eqref{all_t_b_generic} as well as the $s$-regularity
  conditions~\eqref{eq:gamma_Hs_assump} with $s = q$.  Then, for each
  fixed $\omega\in\mathbb{R}$ and each $\eta\in\mathbb{C}$ with
  $\mathrm{Re}(\eta) \ne 0$, the solution $\widetilde{\psi}^f$ of the
  $\eta$-dependent uniquely-solvable second-kind integral equation
\begin{equation}\label{CFIE_proof_generic_21}
  A_{\omega,\eta} \widetilde{\psi}^f (\mathbf{r}, \omega) =
  \gamma^+\partial_\mathbf{n} \widetilde{B}^f(\mathbf{r}, \omega)
  - \i  \eta \gamma^+ \widetilde{B}^f(\mathbf{r}, \omega), \quad\mathbf{r} \in \Gamma,
  \end{equation}
  is itself independent of $\eta$, and, therefore,
  using~\eqref{Aomega_def_eqn}, the $\eta$ independent function
  $\widetilde{\psi}^f$ also satisfies the uniquely-solvable
  second-kind integral equation
  \begin{equation}\label{CFIE_proof_generic_2}
  A_{\omega} \widetilde{\psi}^f (\mathbf{r}, \omega) =
    \gamma^+\partial_\mathbf{n} \widetilde{B}^f(\mathbf{r}, \omega)
    - \i  \eta_0(\omega) \gamma^+ \widetilde{B}^f(\mathbf{r}, \omega), \quad\mathbf{r} \in \Gamma.
\end{equation}
Further, there exists a unique solution
$\widetilde{\psi} \in L^2(\mathbb{R}; L^2(\Gamma))$ satisfying the
integral equation~\eqref{eq:tdie_sl_generic}, which is given by the
inverse Fourier transform, with respect to $\omega$, of the function
$\widetilde{\psi}^f = \widetilde{\psi}^f(\mathbf{r}, \omega)$.  If, in
addition, the incident field $\widetilde{b}$ satisfies the
$s$-regularity conditions~\eqref{eq:gamma_Hs_assump} with $s = p + q$,
then $\widetilde{\psi} \in H^p(\mathbb{R};L^2(\Gamma))$ and
  \begin{equation}\label{psi_Hp_bound}
    \left\|\widetilde{\psi}\right\|_{H^p(\mathbb{R};\,L^2(\Gamma))} \le
    C(\Gamma)\left( \left\|\gamma^+ \partial_\mathbf{n}
        \widetilde{b}\right\|_{H^{p+q}(\mathbb{R};\,L^2(\Gamma))}\\
      + \left\|\gamma^+
        \widetilde{b}\right\|_{H^{p+q+1}(\mathbb{R};\,L^2(\Gamma))}\right).
  \end{equation}
  Finally, if $\widetilde{b}$ satisfies the $s$-regularity
  conditions~\eqref{eq:gamma_Hs_assump} with $s = p + q + 1$, then
  $\widetilde{\psi} \in C^p(\mathbb{R};L^2(\Gamma))$, and the norm of
  $\widetilde{\psi}$ in $C^p(\mathbb{R};L^2(\Gamma))$ is bounded by
  the corresponding boundary data, as it follows
  from~\eqref{psi_Hp_bound} and the Sobolev lemma (\Cref{sob_lemma})
  below.
\end{lemma}
\begin{rem}\label{rem:psi_omega_independent}
  With reference to \Cref{rem:tilde_1}, we note that, for the
  particular case $\widetilde{b} = b$,
  \Cref{3d_decay_lemma_2ndkind_wellposed} tells us that the Fourier
  transform $\psi^f$ of the solution $\psi$ to
  \Cref{eq:tdie_sl} satisfies the frequency-domain equation
\begin{equation}\label{CFIE_proof_2}
  A_{\omega} \psi^f (\mathbf{r}, \omega) =
    \gamma^+\partial_\mathbf{n} B^f(\mathbf{r}, \omega)
    - \i \eta_0(\omega) \gamma^+B^f(\mathbf{r}, \omega), \quad\mathbf{r} \in \Gamma.\qedhere
\end{equation}
\end{rem}
\begin{proof}[Proof of \Cref{3d_decay_lemma_2ndkind_wellposed}.]
  As is well known~\cite[Thm.\ 2.27]{ChandlerWilde:12}, for any
  $\eta\in\mathbb{C}$ satisfying $\mathrm{Re}(\eta)\ne 0$, the
  integral equation~\eqref{CFIE_proof_generic_21} is uniquely
  solvable, its solution $\widetilde{\psi}^f$ does not depend on
  $\eta$, and $\widetilde{\psi}^f$ additionally satisfies~\cite[Thms.\
  2.44 and 2.46]{ChandlerWilde:12} the single-layer integral equation
\begin{equation}\label{muf_slie}
  S_\omega \widetilde{\psi}^
  f(\mathbf{r}, \omega) = \gamma^+ \widetilde{B}^f(\mathbf{r}, \omega), \quad (\mathbf{r}, \omega) \in  \Gamma \times \mathbb{R}.
\end{equation}
(The references cited establish the result for $\omega \ge 0$; the
statement for $\omega < 0$ then follows directly in view of the
aforementioned Hermitian symmetry property.)  Clearly, substituting
$\eta = \eta_0(\omega)$ in~\eqref{CFIE_proof_generic_21} tells us that
$\widetilde{\psi}^f$ is also a solution
of~\eqref{CFIE_proof_generic_2}.

To complete the proof of the lemma we now note that, for any given
non-negative integer $r$, \cref{CFIE_proof_generic_2} tells us that
\[
  (1 + \omega^2)^{r/2} \widetilde{\psi}^f = (1 + \omega^2)^{r/2}
    A_\omega^{-1}\left(\gamma^+ \partial_\mathbf{n} \widetilde{B}^f - \i \eta_0 \gamma^+\widetilde{B}^f \right),
  \]
  and, thus, in view of \eqref{eq:q-nontrapp} and taking into account
  the relation $|\eta_0| \le (1 + \omega^2)^{1/2}$ that follows
  from~\eqref{Aomega_def_eqn},
  we obtain
  \begin{equation}\label{psitilde_norm_bound_freq}
    \begin{split}
      \int_{-\infty}^\infty (1 + \omega^2)^r\left\|\widetilde{\psi}^f(\cdot, \omega)\right\|^2_{L^2(\Gamma)}\,\d\omega \le C &\int_{-\infty}^\infty \left((1 + \omega^2)^{r+q}\left\|\gamma^+\partial_\mathbf{n} \widetilde{B}^f(\cdot,
    \omega)\right\|^2_{L^2(\Gamma)}\right.\\    
    &\left.+ (1 + \omega^2)^{r+q+1}\left\|\gamma^+\widetilde{B}^f(\cdot,
      \omega)\right\|^2_{L^2(\Gamma)}\right)\,\d\omega.
      \end{split}
  \end{equation}
  In particular, equation~\eqref{psitilde_norm_bound_freq} with
  $r = 0$ tells us that that
  $\widetilde{\psi}^f \in L^2(\mathbb{R}; L^2(\Gamma))$, in view of
  the $s$-regularity conditions assumed with $s=q$. As a result, the
  inverse time Fourier transform
  $\widetilde{\psi} = \mathcal{F}^{-1}\left[
    \widetilde{\psi}^f\right]$ (see \Cref{eq:fourier_transf}) is
  well-defined. Since $\widetilde{\psi}^f$ satisfies \Cref{muf_slie}
  for every $\omega$, it follows from~\eqref{eq:single_layer_op_ii}
  that its inverse Fourier transform $\widetilde{\psi}$ satisfies
  \Cref{eq:tdie_sl_generic}, and, thus the existence of a solution to
  \Cref{eq:tdie_sl_generic}, given by the inverse Fourier transform of
  $\widetilde{\psi}^f$, follows. To establish the uniqueness of this
  solution in $ L^2(\mathbb{R}; L^2(\Gamma))$, we assume
  $ \lambda \in L^2(\mathbb{R}; L^2(\Gamma))$ satisfies
  $\left(S \lambda\right)(\mathbf{r}, t) = 0$ for
  $(\mathbf{r}, t) \in \Gamma \times
  \mathbb{R}$. Using~\eqref{eq:Stilde_def}
  and~\eqref{eq:single_layer_op_ii} we see that the Fourier transform
  $\lambda^f = \mathcal{F}[\lambda]$ satisfies the equation
  $\left(S_\omega \lambda^f(\cdot, \omega)\right)(\mathbf{r}) = 0$ for
  each $\omega \in \mathbb{R}$.  Since $S_\omega$ is invertible for
  almost every $\omega$ (or, more precisely, for all
  $\omega\in\mathbb{R}$ except for the measure-zero set of $\omega$
  for which $-\omega^2$ is a Dirichlet eigenvalue for the Laplace
  operator in the interior of $\Omega$~\cite[Thm.\
  2.25]{ChandlerWilde:09}), we conclude that
  $\lambda^f(r,\omega)\equiv 0$ for almost every
  $\omega\in\mathbb{R}$, and thus
  $\lambda = \mathcal{F}^{-1}[\lambda^f] \equiv 0$. The existence and uniqueness of solutions of \eqref{eq:tdie_sl_generic} thus follows.

  The estimate~\cref{psi_Hp_bound} is equivalent
  to~\eqref{psitilde_norm_bound_freq} with $r=p$, in view of the
  equivalence of the norms displayed in
  equations~\eqref{sob_bochner_norm} and~\eqref{sob_bochner_fourier}
  in~\Cref{sob-boch} below.  Under the additional assumption of
  $s$-regularity with $s = p + q + 1$, the fact that
  $\widetilde{\psi} \in C^p(\mathbb{R};\,L^2(\Gamma))$, together with
  an associated norm bound, result from an application
  of~\eqref{psi_Hp_bound} with $p$ substituted by $p+1$ and the use of
  the Sobolev embedding \Cref{sob_lemma}. The proof is now complete.
\end{proof}

In preparation for Lemma~\ref{3d_decay_lemma_2ndkind}
we note the relation
\begin{equation}\label{potent_time_freq}
  \mathcal{F}\left [ \mathscr{S}[\mu]\right] =
  \widetilde{\mathscr{S}}\left[\mathcal{F}[\mu]\right]
\end{equation}
that relates the time- and frequency-domain single layer
potentials~\eqref{eq:single_layer_pot_time}
and~\eqref{eq:single_layer_pot} (where we have set
\begin{equation}\label{tilde_scr_S}
  \widetilde{\mathscr{S}}[\mu](\mathbf{r},\omega) =
  \mathscr{S}_\omega[\mu(\cdot,\omega)](\mathbf{r}),\quad (\mathbf{r}, \omega) \in \mathbb{R}^3 \times \mathbb{R};
\end{equation}
cf. equations~\eqref{eq:Stilde_def}
and~\eqref{eq:single_layer_op_ii}), and which, as in \Cref{app_c},
follows via an application of the change of
variables~\eqref{tau_chvar} to the left-hand term of the equation.
Also, letting
\begin{equation}\label{umintilde_def}
    \widetilde{u}_{-,T}(\mathbf{r}, t) := (\mathscr{S} \widetilde{\psi}_{-,T})(\mathbf{r}, t), \quad (\mathbf{r}, t) \in \mathbb{R}^3 \times \mathbb{R},
  \end{equation}
we define the function
\begin{equation}\label{htildedef}
  \widetilde{h}_T(\mathbf{r}, t) := \widetilde{b}(\mathbf{r}, t) -
  \widetilde{u}_{-,T}(\mathbf{r}, t),\quad (\mathbf{r}, t) \in \Omega^{inc} \times \mathbb{R},
\end{equation}
and, using~\eqref{potent_time_freq}, its Fourier transform
\begin{equation}\label{Ht_def_generic}
  \widetilde{H}^f_T(\mathbf{r}, \omega) = \widetilde{B}^f(\mathbf{r}, \omega) - \widetilde{\mathscr{S}}\left[ 
  \widetilde{\psi}_{-,T}^f\right](\mathbf{r}, \omega),\quad \mathbf{r}\in\Omega^{inc} \times \mathbb{R}.
\end{equation}

\begin{rem}\label{rem:tilde_2}
  With reference to Remark~\ref{rem:tilde_1}, we mention that
  Lemma~\ref{3d_decay_lemma_2ndkind} and subsequent lemmas apply in
  particular to the function $\widetilde{b} = b$ in \Cref{eq:tdie_sl};
  when such applications arise we will accordingly use untilded
  notations such as e.g.
\begin{equation}\label{umin_def}
    u_{-,T}(\mathbf{r}, t) := (\mathscr{S} \psi
  _{-,T})(\mathbf{r}, t), \quad (\mathbf{r}, t) \in \mathbb{R}^3 \times \mathbb{R},
\end{equation}
instead of~\eqref{umintilde_def}, 
\begin{equation}\label{hdef}
  h_T(\mathbf{r}, t) := b(\mathbf{r}, t) - u_{-,T}(\mathbf{r}, t),\quad \mathbf{r}\in\Omega^{inc} \times \mathbb{R},
\end{equation}
instead of~\eqref{htildedef}, and
\begin{equation}\label{Ht_untilded_def}
    H^f_T(\mathbf{r}, \omega) = B^f(\mathbf{r}, \omega) - \widetilde{\mathscr{S}}\left[
      \psi_{-,T}^f\right] (\mathbf{r}, \omega),\quad \mathbf{r}\in\Omega^{inc} \times \mathbb{R},
\end{equation}
instead of~\eqref{Ht_def_generic}.
\end{rem}

Relying on the $\tau$DoD interval $I_{T}$ introduced in
\Cref{domainofdep}, as well as the windowed time-dependent boundary
densities $\widetilde{\psi}_{-,T}$, $\widetilde{\psi}_{+,T}$, and
$\widetilde{\psi}_{*,T}$ introduced in \Cref{timewinddens} and
\Cref{rem:tilde_1}, and the trace operators $\gamma^\pm$ on $\Gamma$
in \Cref{trace_def} (Appendix~\ref{sob-boch}),
Lemmas~\ref{3d_decay_lemma_2ndkind} to~\ref{per_freq_rhs_oper_bounds}
below develop estimates on the time-dependent density
$\widetilde{\psi}_{+,T}$ and bounds on the norms of certain
frequency-domain operators, all of which are used in the proof of
\Cref{3d_decay_thm}. The subsequent \Cref{sob_lemma} is the Sobolev
Lemma for functions with values in a Banach space
(i.e. $L^2(\Gamma)$), that is used, in particular, to establish the
theorem's uniform-in-time estimate~\eqref{density_unif_time_bound}.
\begin{lemma}[Direct second-kind integral equations for
  $\widetilde{\psi}_{+,T}^f$]\label{3d_decay_lemma_2ndkind}
  Assume that the obstacle $\Omega$ satisfies the $q$-growth condition
  for some integer $q \ge 0$.  Let $\widetilde{b}$ denote a function
  defined in an open set $\Omega^\textit{inc}$ containing
  $\widebar{\Omega}$, and assume that $\widetilde{b}$ is a solution to
  the wave equation in $\Omega^\textit{inc}$ that
  satisfies~\eqref{all_t_b_generic} as well as the $s$-regularity
  conditions~\eqref{eq:gamma_Hs_assump} with $s = q$.  Additionally,
  let $T$ denote a real number with associated time-windowed solutions
  $\widetilde{\psi}_{+,T}$ and $\widetilde{\psi}_{-,T}$ as in
  \Cref{timewinddens} and \Cref{rem:tilde_1}, where $\widetilde{\psi}$
  is the solution to \Cref{eq:tdie_sl_generic}. Then
  $\widetilde{\psi}_{+,T}$ solves the integral equation
    \begin{equation}\label{int_eq_lambda}
      (S\widetilde{\psi}_{+,T})(\mathbf{r}, t) = \gamma^+ \widetilde{h}_T(\mathbf{r}, t),
    \end{equation}
    (cf. \Cref{eq:tdie_sl}) where $\widetilde{h}_T$ is given by
    \Cref{htildedef}.  Further, for each fixed $\omega \ge 0$ the
    Fourier transform $\widetilde{\psi}_{+,T}^f$ of
    $\widetilde{\psi}_{+,T}$ satisfies the second-kind integral
    equation
\begin{equation}\label{CFIE_proof_generic}
  \left(A_{\omega} \widetilde{\psi}_{+,T}^f \right)(\mathbf{r}, \omega) =
    \gamma^-\partial_\mathbf{n} \widetilde{H}^f_T(\mathbf{r}, \omega)
    - \i  \eta_0(\omega) \gamma^- \widetilde{H}^f_T(\mathbf{r}, \omega), \quad\mathbf{r} \in \Gamma,
\end{equation}
where $A_{\omega}$ and $\widetilde{H}^f_T$ are given by
equations~\cref{Aomega_def_eqn} and \cref{Ht_def_generic},
respectively.
\end{lemma}
\begin{rem}\label{rem:tilde_3}
  With reference to \Cref{rem:tilde_1}, we note that, for the
  particular case $\widetilde{b} = b$, \Cref{3d_decay_lemma_2ndkind}
  tells us that the Fourier transform $\psi^f$ of the solution $\psi$ to \Cref{eq:tdie_sl} satisfies
\begin{equation}\label{CFIE_proof}
  \left(A_{\omega} \psi_{+,T}^f \right)(\mathbf{r}, \omega) =
    \gamma^-\partial_\mathbf{n} H^f_T(\mathbf{r}, \omega)
    - \i \eta_0(\omega) \gamma^-H^f_T(\mathbf{r}, \omega), \quad\mathbf{r} \in \Gamma,
\end{equation}
where
\begin{equation}\label{Ht_def}
    H^f_T(\mathbf{r}, \omega) = B^f(\mathbf{r}, \omega) - \left(\mathscr{S}_\omega
        \psi_{-,T}^f\right)(\mathbf{r}, \omega).\qedhere
\end{equation}
\end{rem}
\begin{rem}\label{rem:traces_interior_jump}
  It is important to note a certain structural difference between
  equations~\eqref{CFIE_proof_generic_2}
  and~\eqref{CFIE_proof_generic} (both of which concern the
  frequency-domain operator $A_{\omega}$), namely, that the right hand
  side in the first equation is expressed in terms of the exterior
  trace operator $\gamma^+$ while the one for the second utilizes the
  interior trace operator $\gamma^-$ instead. In the first case using
  $\gamma^-$ on the right hand side produces identical results, since,
  according to~\eqref{all_t_b_generic}, the incident field is assumed
  to be sufficiently smooth on the set $\Omega^\textit{inc}$ and,
  therefore, in a neighborhood of $\Gamma$. In the second case,
  however, as indicated in the proof of~\Cref{3d_decay_lemma_2ndkind},
  interior traces appear in view of the relation~\cite[p. 219]{McLean}
    \begin{equation}\label{gammam}
      \gamma^- \partial_\mathbf{n} \mathscr{S}_\omega = \frac{1}{2} I
      + K_\omega^*
    \end{equation}
    for the values of the interior normal-derivative of the
    single-layer potential on $\Gamma$, which happens to reproduce the
    $\frac{1}{2} I + K_\omega^*$ contribution in the
    expression~\eqref{Aop_def_eqn}-\eqref{Aomega_def_eqn} for the
    operator $A_\omega$, while use of an exterior normal derivative
    would yield a different result, on account of the jump conditions
    for the normal derivatives of the single layer potential across
    $\Gamma$. In a related interpretation, the presence of $\gamma^-$
    reflects the fact that the Green-formula derivations of
    \emph{direct} integral equation formulations, like the
    $A_\omega$-equations we use, rely on consideration of the interior
    problem for the single layer potential that represents the
    incident field.
\end{rem}
\begin{proof}[Proof of \Cref{3d_decay_lemma_2ndkind}.]
  Since $\widetilde{b}$ satisfies the assumptions in
  \Cref{3d_decay_lemma_2ndkind_wellposed}, an application of that
  lemma with with $p=0$ tells us that there exists a unique solution
  $\widetilde{\psi} \in L^2(\mathbb{R}; L^2(\Gamma))$ of equation
  \eqref{eq:tdie_sl_generic}. Using the partition of unity
  decomposition embodied in~\cref{psipm_def},
  equation~\Cref{eq:tdie_sl_generic} may be re-expressed in the
  form~\eqref{int_eq_lambda}, where $\widetilde{h}_T$ and the
  associated $\widetilde{u}_{-,T}$ are given by \Cref{htildedef} and
  \Cref{umintilde_def} respectively.

  To establish \Cref{CFIE_proof_generic} we use
  \Cref{3d_decay_lemma_2ndkind_wellposed} again which tells us that
  $\widetilde{\psi}^f = \mathcal{F}[\widetilde{\psi}]$ solves
  \Cref{CFIE_proof_generic_2}. But using the identity
  $\widetilde{\psi}^f = \widetilde{\psi}_{+,T}^f +
  \widetilde{\psi}_{-,T}^f$ (see \Cref{timewinddens} and
  \Cref{rem:tilde_1}), together with the relations
  $\gamma^- \partial_\mathbf{n} \widetilde{B}^f = \gamma^+
  \partial_\mathbf{n} \widetilde{B}^f$, and
  $\gamma^- \widetilde{B}^f = \gamma^+ \widetilde{B}^f$, which result
  in view of the assumed smoothness of $\widetilde{B}^f$ in a
  neighborhood of $\Gamma$, we thus obtain
  \begin{equation}\label{moved_2_RHS}
    \begin{split}
      \left(A_{\omega} \widetilde{\psi}_{+,T}^f \right)(\mathbf{r},
      \omega) = &-\left(A_{\omega} \widetilde{\psi}_{-,T}^f
      \right)(\mathbf{r}, \omega) \\&+ \gamma^-\partial_\mathbf{n}
      \widetilde{B}^f(\mathbf{r}, \omega) - \i \eta_0(\omega) \gamma^-
      \widetilde{B}^f(\mathbf{r}, \omega), \quad\mathbf{r} \in \Gamma.
        \end{split}
      \end{equation}
      In view of the definition~\cref{Aomega_def_eqn} of $A_\omega$
      and the relation~\eqref{gammam} we find that
    \begin{equation}
        \begin{split}
            \left(A_{\omega} \widetilde{\psi}_{+,T}^f \right)(\mathbf{r}, \omega) = &\gamma^- \partial_{\mathbf{n}} \left( \widetilde{B}^f - \mathscr{S}_\omega \widetilde{\psi}_{-,T}^f\right) 
            \\&- \i\eta_0(\omega) \gamma^- \left( \widetilde{B}^f - \mathscr{S}_\omega \widetilde{\psi}^f_{-,T}\right), \quad\mathbf{r} \in \Gamma.
        \end{split}
    \end{equation}
    Using the definition~\cref{Ht_def_generic} of $\widetilde{H}_T^f$
    together with~\cref{tilde_scr_S}, finally, the desired
    equation~\cref{CFIE_proof_generic} follows, and the proof is
    complete.
\end{proof}

\begin{lemma}[Frequency $L^2$ bounds on the solution
  of~\cref{CFIE_proof_generic}]\label{3d_decay_lemma_Ainv_freq_bounds}
  Let $T\in\mathbb{R}$ and $q \geq 0$, let $\Gamma = \partial \Omega$
  satisfy the $q$-growth condition, and assume $\widetilde{b}$
  satisfies the conditions of \Cref{3d_decay_lemma_2ndkind}. Then
  there exist constants $C_1 = C_1(\Gamma) > 0$ and $C_2 = C_2(\Gamma) > 0$
  such that $\widetilde{\psi}^f_{+,T}$ satisfies
\begin{equation}\label{psiplus_estimate_3}
    \begin{split}
    \left\|\widetilde{\psi}_{+,T}^f\right\|^2_{L^2(\mathbb{R};\,L^2(\Gamma))}
    &\le C_1\int_0^{\omega_0}
    \left\|\gamma^-\partial_\mathbf{n} \widetilde{H}^f_T(\cdot,
    \omega) - \i \gamma^-\widetilde{H}^f_T(\cdot,
    \omega)\right\|^2_{L^2(\Gamma)}\,\d\omega\\
    & + C_2\int_{\omega_0}^\infty\omega^{2q}
    \left\|\gamma^-\partial_\mathbf{n} \widetilde{H}^f_T(\cdot,
    \omega) - \i\omega \gamma^- \widetilde{H}^f_T(\cdot.
    \omega)\right\|^2_{L^2(\Gamma)}\,\d\omega,
    \end{split}
\end{equation}
    where $\widetilde{H}^f_T$ is given by \Cref{Ht_def_generic}.
\end{lemma}
\begin{proof}
  In view of \Cref{fd_layer_continuity_lemma}
  (cf. \Cref{3d_decay_lemma_2ndkind_wellposed}), $A_\omega$ is
  invertible for all $\omega \ge 0$, and, thus,
  using~\cref{CFIE_proof_generic} we obtain
  $\widetilde{\psi}_{+,T}^f = A_{\omega}^{-1} \left(
    \gamma^-\partial_\mathbf{n} \widetilde{H}^f_T - \i \eta_0 \gamma^-
    \widetilde{H}^f_T\right)$. It follows that, for all
  $\omega \ge 0$,
\begin{equation*}
\begin{split}
  \|\widetilde{\psi}_{+,T}^f(\cdot, \omega)\|_{L^2(\Gamma)} \le
    \|A_{\omega}^{-1}&\|_{L^2(\Gamma) \to L^2(\Gamma)}\cdot
    \left\|\gamma^-\partial_\mathbf{n} \widetilde{H}^f_T(\cdot,
    \omega) - \i\eta_0(\omega) \gamma^-\widetilde{H}^f_T(\cdot,
    \omega)\right\|_{L^2(\Gamma)}.\label{psiplus_estimate_1}
\end{split}
\end{equation*}
Using the $q$-growth condition~\Cref{eq:q-nontrapp-cases} we then
obtain
\begin{equation}
  \|\widetilde{\psi}_{+,T}^f(\cdot, \omega)\|_{L^2(\Gamma)} ^2\le
  \frac{C_1}{2}\left\|\gamma^-\partial_\mathbf{n} \widetilde{H}^f_T(\cdot,
    \omega) - \i \gamma^-\widetilde{H}^f_T(\cdot,
    \omega)\right\|_{L^2(\Gamma)}^2,\, (0 \le \omega < \omega_0)\label{psiplus_estimate_2b}
\end{equation}
    and
\begin{equation}
  \|\widetilde{\psi}_{+,T}^f(\cdot, \omega)\|_{L^2(\Gamma)}^2 \le
    \frac{C_2}{2}\omega^{2q}\left\|\gamma^-\partial_\mathbf{n} \widetilde{H}^f_T(\cdot,
    \omega) - \i\omega \gamma^-\widetilde{H}^f_T(\cdot,
    \omega)\right\|_{L^2(\Gamma)}^2,\, (\omega > \omega_0)\label{psiplus_estimate_2}
\end{equation}
for certain constants $C_1$ and $C_2$.  Noting that by Hermitian
symmetry (see \Cref{negative_freq_eq}) we have
$\left\|\widetilde{\psi}_{+,T}^f(\cdot, \omega)\right\|_{L^2(\Gamma)}
= \left\|\widetilde{\psi}_{+,T}^f(\cdot,
  |\omega|)\right\|_{L^2(\Gamma)}$ for $\omega \in \mathbb{R}$, it follows that
\begin{equation}
    \begin{split}
    \left\|\widetilde{\psi}_{+,T}^f\right\|^2_{L^2(\mathbb{R};\,L^2(\Gamma))}
      &= 2 \int_0^\infty \left\|\widetilde{\psi}_{+,T}^f(\cdot,
    \omega)\right\|^2_{L^2(\Gamma)}\,\d\omega \\
    &\le C_1\int_0^{\omega_0}
      \left\|\gamma^-\partial_\mathbf{n} \widetilde{H}^f_T(\cdot,
      \omega) - \i \gamma^-\widetilde{H}^f_T(\cdot,
    \omega)\right\|^2_{L^2(\Gamma)}\,\d\omega\\
    & + C_2\int_{\omega_0}^\infty
      \omega^{2q}\left\|\gamma^-\partial_\mathbf{n} \widetilde{H}^f_T(\cdot,
      \omega) - \i\omega \gamma^-\widetilde{H}^f_T(\cdot,
    \omega)\right\|^2_{L^2(\Gamma)}\,\d\omega,\\
  \end{split}
\end{equation}
as desired.
\end{proof}

\begin{lemma}[About $\widetilde{h}_T$: Limited temporal history
  horizon and bounded support]\label{3d_decay_thm_h_equiv}

  Let $T>0$ be given. Further, let $\widetilde{b}$ denote a function
  defined in the open set $\Omega^\textit{inc}$ containing
  $\widebar{\Omega}$, and assume that $\widetilde{b}$ is a solution of
  the wave equation in $\Omega^\textit{inc}$ that
  satisfies~\eqref{all_t_b_generic} for some $\alpha\in\mathbb{R}$
  and which vanishes for
  $(\bfr,t) \in \widebar\Omega \times \left\lbrace I_{T}\cup [T,
    \infty)\right\rbrace$ (that is, for $\bfr \in \widebar\Omega$ and
  $t \ge T - T_* - 2\tau$; see \Cref{domainofdep}). Then, the function
  $\widetilde{h}_T$ defined in \Cref{htildedef} satisfies
\begin{equation}\label{h_equiv}
    \widetilde{h}_T(\mathbf{r}, t) =
  \left\{\begin{aligned}
    -\widetilde{u}_{*,T}(\mathbf{r}, t), \quad &t \ge T - \tau\\
    0, \quad &t < T - \tau
  \end{aligned}\right.\quad\mbox{for}\quad\mathbf{r} \in \widebar{\Omega},
\end{equation}
  where, with reference to \Cref{timewinddens} and  \Cref{rem:tilde_1},
\begin{equation}\label{ustar_def}
    \widetilde{u}_{*,T}(\mathbf{r}, t) = (\mathscr{S} \widetilde{\psi}_{*,T})(\mathbf{r}, t).
\end{equation}
Further, the function $\widetilde{h}_T(\mathbf{r}, t)$ has a bounded
temporal support:
  \begin{equation}\label{h_compactsupp}
      \supp \widetilde{h}_T(\mathbf{r}, \cdot)\subset [T - \tau, T + T_*], \quad\mbox{for}\quad \mathbf{r} \in \widebar{\Omega}.
  \end{equation}
\end{lemma}
\begin{rem}\label{rem:h_equiv_rem_1}
  The boundedness of the temporal support of the function
  $\widetilde{h}_T(\mathbf{r}, t) = \widetilde{b} -
  \widetilde{u}_{-,T}$ provides an essential element in the proof of
  Theorem~\ref{3d_decay_thm_ii} and associated lemmas; additional
  details in this regard are provided in
  \Cref{breve{h}_0_impact_rem}. Note that, in particular, the
  right-hand side $\gamma^+\widetilde{h}_T(\mathbf{r}, t)$
  of~\eqref{int_eq_lambda} vanishes at times for which the solution
  $\widetilde{\psi}_{+,T}$ generically does not vanish.
\end{rem}
\begin{rem}\label{rem:h_equiv_rem_2}
  In view of~\eqref{psistar_def}, and taking into account
  \Cref{rem:tilde_1}, \Cref{3d_decay_thm_h_equiv} tells us that
  $\widetilde{h}_T$ is determined by the restriction of
  $\widetilde{\psi}(\bfr, t)$ to the interval $I_T$, and, thus, in
  view of~\cref{int_eq_lambda}, the same is true of
  $\widetilde{\psi}_{+,T}$. Since, by~\eqref{psipm_def},
  $\widetilde{\psi}(\bfr, t)$ coincides with
  $\widetilde{\psi}_{+,T}(\bfr, t)$ for all $t\geq T$, it follows
  that, for such values of $t$ the solution
  $\widetilde{\psi}(\bfr, t)$ is solely determined by the restriction
  of $\widetilde{\psi}(\bfr, t)$ to the interval $I_T$. In other
  words, for a given time $T$ and in absence of illumination for
  $t\geq T-T_* - 2\tau$, $\widetilde{\psi}(\bfr, t)$ is determined by
  the values of the same solution function $\widetilde{\psi}(\bfr, t)$
  in the bounded time interval $I_T$ immediately preceding time
  $T$. Thus, \Cref{3d_decay_thm_h_equiv} implies a bootstrap
  domain-of-dependence relation for solutions of the wave equation.
\end{rem}
\begin{proof}[Proof of \Cref{3d_decay_thm_h_equiv}.]
  We first establish~\eqref{h_equiv} for $t \geq T - \tau$ and
  $\bfr \in \widebar{\Omega}$. Since $\widetilde{b}$ vanishes for
  $t \geq T - T_* - 2\tau$, it suffices to show that
  $\widetilde{u}_{-,T}(\mathbf{r}, t) =
  \widetilde{u}_{*,T}(\mathbf{r}, t)$ for $t \geq T - \tau$, or,
  equivalently, in view of~\eqref{umintilde_def}
  and~\eqref{ustar_def}, that
  \begin{equation}\label{Spsimin_Spsistar}
  (\mathscr{S} \widetilde{\psi}_{-,T})(\mathbf{r}, t) = (\mathscr{S}
  \widetilde{\psi}_{*,T})(\mathbf{r}, t)\quad\mbox{ for }\quad t > T - \tau.
\end{equation}
Letting $\xi = t - |\mathbf{r} - \mathbf{r}'|/c$ denote the second
argument of the density factor
$\mu(\mathbf{r}', t - |\mathbf{r} - \mathbf{r}'|/c)$ in the
integrand~\eqref{eq:single_layer_pot_time_conv}, the density factors
in $(\mathscr{S} \widetilde{\psi}_{-,T})$ and
$(\mathscr{S} \widetilde{\psi}_{*,T}) $ are given by
$w_-(\xi - T)\psi(\bfr, \xi)$ and
$w_+(\xi - T + T_* + \tau) w_-(\xi - T)\psi(\mathbf{r}, \xi)$,
respectively. But, in view of~\eqref{wtau_def},
$w_+(\xi - T + T_* + \tau)=1$---since, using~\eqref{Tstar_def} we see
that
$\xi - T + T_* + \tau \geq \xi - T + \tau + |\mathbf{r} -
\mathbf{r}'|/c = t - T + \tau$, and, thus, $\xi - T + T_* + \tau > 0$
in the present case $t > T - \tau$. It follows that for such times the
two density factors coincide, showing that~\eqref{Spsimin_Spsistar}
holds, and, thus, that
$\widetilde{u}_{-,T}(\mathbf{r}, t) = \widetilde{u}_{*,T}(\mathbf{r},
t)$ for $t > T - \tau$---which establishes~\eqref{h_equiv} for
$t > T - \tau$.

  To complete the proof of \Cref{h_equiv} it remains to show that
  $\widetilde{h}_T(\mathbf{r}, t) = 0$ for $\bfr \in \widebar{\Omega}$
  and $t < T - \tau$. But~\eqref{int_eq_lambda} tells us that
  $\widetilde{h}_T = 0$ for $\bfr \in \Gamma$ and $t < T - \tau$,
  since, by definition, $\widetilde{\psi}_{+,T}(\mathbf{r}, t)$ itself
  vanishes for $\bfr \in \Gamma$ and $t < T - \tau$. In view
  of~\cref{htildedef} it follows that $\widetilde{h}_T(\bfr, t)$ is a solution
  of the wave equation for $\bfr \in \Omega$ with vanishing boundary values for
  $t\in (-\infty, T - \tau)$ which satisfies vanishing initial
  conditions throughout $\Omega$ for a sufficiently large negative
  value of $t$. Thus, $\widetilde{h}_T(\mathbf{r}, t)$ vanishes for
  $(\mathbf{r}, t) \in \widebar{\Omega} \times (-\infty, T - \tau)$,
  by solution uniqueness, and \Cref{h_equiv} follows.

  In order to establish \Cref{h_compactsupp}, finally, it suffices, in
  view of \Cref{h_equiv}, to show that
  $\widetilde{h}_T(\mathbf{r}, t) = 0$ for all
  $\mathbf{r} \in \widebar{\Omega}$ and $t > T + T_*$. To do this we
  note from \Cref{timewinddens} that
  $\supp \widetilde{\psi}_*(\mathbf{r}, \cdot) \subset I_T = [T - T_*
  - 2\tau, T)$. Thus, equation~\eqref{ustar_def}, which can be
  expressed in the form
  $\widetilde{u}_{*,T}(\mathbf{r}, t) = \int_\Gamma
  \frac{\widetilde{\psi}_{*,T}(\mathbf{r}', t - |\mathbf{r} -
    \mathbf{r}'|/c)}{4\pi |\mathbf{r} -
    \mathbf{r}'|}\,\d\sigma(\mathbf{r}')$, tells us that
  $\widetilde{h}_T(\mathbf{r}, t) = 0$ for
  $(\mathbf{r}, t) \in \widebar{\Omega} \times [T + T_*, \infty)$.
  Equation~\eqref{h_compactsupp} thus follows and the proof is
  complete.
  \end{proof}

\begin{lemma}\label{per_freq_rhs_oper_bounds}
  Let $\Gamma$ denote the boundary of a Lipschitz obstacle. Then there
  exist constants $C_1 = C_1(\Gamma) > 0$ and $C_2 = C_2(\Gamma) > 0$
  such that the operator norm bounds
  \begin{equation}\label{per_freq_rhs_oper_bound_i}
    \left\|\left(\gamma^-\partial_\mathbf{n} \pm
      \i\omega\gamma^- \right) \mathscr{S}_\omega\right\|_{L^2(\Gamma)\to L^2(\Gamma)} \le C_1(1 +
    \omega^2)^{1/2}
  \end{equation}
  and
  \begin{equation}\label{per_freq_rhs_oper_bound_ii}
    \left\|\left(\gamma^-\partial_\mathbf{n} \pm
      \i\gamma^- \right) \mathscr{S}_\omega\right\|_{L^2(\Gamma)\to L^2(\Gamma)} \le C_2(1 +
    \omega^2)^{1/2}
  \end{equation}
  hold for all $\omega \ge 0$.
\end{lemma}
\begin{proof}
  The relations~\eqref{per_freq_rhs_oper_bound_i}
  and~\eqref{per_freq_rhs_oper_bound_ii} follow directly from the bounds
\begin{equation}\label{Somega_Kstaromega_bounds}
    \left\|\gamma^- \mathscr{S}_\omega\right\|_{L^2(\Gamma) \to L^2(\Gamma)}
    \le D_1 ,\quad\mbox{and}\quad
    \left\|K_\omega^*\right\|_{L^2(\Gamma) \to L^2(\Gamma)}
    \le D_2 \omega + D_3,
\end{equation}
valid for all $\omega \geq 0$ (where $D_1, D_2, D_3 > 0$ are
constants), which are presented in reference~\cite[Thms.\ 3.3 and
3.5]{ChandlerWilde:09}. Indeed, in view of~\eqref{jump_cond} it follows
that there exist constants $\widetilde{D}_1 > 0$ and
$\widetilde{D}_2 > 0$ and $C_1 > 0$ such that
\begin{equation*}
  \begin{split}
      \left\|\left(\gamma^-\partial_\mathbf{n} \pm \i\omega\gamma^-\right) \mathscr{S}_\omega\right\|_{L^2(\Gamma)\to L^2(\Gamma)} &= \left\|\frac{1}{2}I + K_\omega^* \pm
    \i\omega S_\omega\right\|_{L^2(\Gamma)\to L^2(\Gamma)}\\
      &\le (\widetilde{D}_1 + \widetilde{D}_2\omega) \le C_1(1 + \omega^2)^{1/2}.
  \end{split}
\end{equation*}
Similarly,
\begin{equation*}
  \begin{split}
      \left\|\left(\gamma^-\partial_\mathbf{n} \pm \i\gamma^- \right)\mathscr{S}_\omega\right\|_{L^2(\Gamma)\to L^2(\Gamma)}
    &\le C_2(1 + \omega^2)^{1/2}
  \end{split}
\end{equation*}
for some constant $C_2 > 0$, and the result follows.
\end{proof}

\begin{lemma}\label{conv_H_to_psi}
  Let $q$ denote a non-negative integer, let $\Gamma = \partial\Omega$
  satisfy the $q$-growth condition (\Cref{q-nontrapp}), let $T_0$ denote a
  given real number such that $\widetilde{b}$ vanishes for
  $(\bfr,t) \in \widebar{\Omega} \times \left\lbrace I_{T_0}\cup [T_0,
    \infty)\right\rbrace$, and let $\widetilde{H}_{T_0}^f$ be defined
  by \Cref{Ht_def_generic} with $T = T_0$. Then for some constant
  $C > 0$ independent of $T_0$ and $\widetilde{b}$ we have
\begin{equation}\label{conv_H_to_psi_eq}
  \begin{alignedat}{2}
      \int_0^{\omega_0} &\, &&
      \left\|\left(\gamma^-\partial_\mathbf{n} - \i
      \gamma^-\right) \widetilde{H}^f_{T_0}(\cdot,
      \omega)\right\|^2_{L^2(\Gamma)}\,\d\omega \\
      +\,& &&\int_{\omega_0}^\infty \omega^{2q}
      \left\|\left(\gamma^-\partial_\mathbf{n} - \i\omega
      \gamma^-\right) \widetilde{H}^f_{T_0}(\cdot.
      \omega)\right\|^2_{L^2(\Gamma)}\,\d\omega\\
      &\,&&\le C \int_{-\infty}^\infty (1 + \omega^2)^{q+1} \left\|
      \widetilde{\psi}_{*,T_0}^f(\cdot, \omega)\right\|_{L^2(\Gamma)}^2\,\d\omega.
  \end{alignedat}
\end{equation}
\end{lemma}
\begin{proof}
  In order to obtain the desired bound we define the differential
  operators
\begin{equation}\label{ST_op_time}
  \mathcal{R} = \left(\gamma^-\partial_{\mathbf{n}} - \i\gamma^- \right)
  \enspace\mbox{and}\enspace\mathcal{T} = \left(-\i \frac{\partial}{\partial
    t}\right)^{q}\left(\gamma^-\partial_{\mathbf{n}} -
  \frac{\partial}{\partial t}\gamma^-\right),
\end{equation}
whose Fourier symbols are given by
\begin{equation}\label{ST_op_freq}
\widehat{\mathcal{R}} =
  \left(\gamma^-\partial_\mathbf{n} - \i \gamma^-\right)\enspace\mbox{and}\enspace \widehat{\mathcal{T}} =
    \omega^{q}\left(\gamma^-\partial_\mathbf{n} - \i\omega\gamma^- \right).
\end{equation}
Then, denoting by $\mathcal{Q}$ the sum of quantities on the left-hand
side of \Cref{conv_H_to_psi_eq} and using~\eqref{h_equiv} together
with Plancherel's theorem we obtain
\begin{alignat*}{2}
  &\begin{alignedat}{2}
      \mathcal{Q} =
      \int_0^{\omega_0} &\left\|\widehat{\mathcal{R}}\widetilde{H}^f_{T_0}(\cdot,
    \omega)\right\|^2_{L^2(\Gamma)}\,\d\omega
      &+ \int_{\omega_0}^\infty
      \left\|\widehat{\mathcal{T}}\widetilde{H}^f_{T_0}(\cdot,
      \omega)\right\|^2_{L^2(\Gamma)}\,\d\omega
    \end{alignedat}\\
    &\quad\quad\begin{alignedat}{2}
        \le \int_\Gamma \int_{-\infty}^\infty &\left|\widehat{\mathcal{R}} \widetilde{H}^f_{T_0}(\mathbf{r},
    \omega)\right|^2\,\d\omega\,\d\sigma(\mathbf{r}) + \int_\Gamma
        \int_{-\infty}^\infty \left|\widehat{\mathcal{T}} \widetilde{H}^f_{T_0}(\mathbf{r},
    \omega)\right|^2\,\d\omega\,\d\sigma(\mathbf{r})
    \end{alignedat}\\
    &\quad\quad\begin{alignedat}{2}
        = \int_\Gamma \int_{-\infty}^\infty &\left|\mathcal{R} \widetilde{h}_{T_0}(\mathbf{r},
    t')\right|^2\,\d t'\,\d\sigma(\mathbf{r}) + \int_\Gamma
        \int_{-\infty}^\infty \left|\mathcal{T} \widetilde{h}_{T_0}(\mathbf{r},
    t')\right|^2\,\d t'\,\d\sigma(\mathbf{r})
    \end{alignedat}\\
    &\quad\quad\begin{alignedat}{2}
        = \int_\Gamma \int_{T_0 - \tau}^\infty &\left|\mathcal{R} \widetilde{u}_{*,T_0}(\mathbf{r},
    t')\right|^2\,\d t'\,\d\sigma(\mathbf{r}) + \int_\Gamma
        \int_{T_0 - \tau}^\infty \left|\mathcal{T} \widetilde{u}_{*,T_0}(\mathbf{r},
    t')\right|^2\,\d t'\,\d\sigma(\mathbf{r}),\\
    \end{alignedat}
\end{alignat*}
where the last equality follows from~\eqref{h_equiv} by virtue of the
temporal locality of the operators $\mathcal{R}$ and $\mathcal{T}$. In
view of the non-negativity of the integrands on the right-hand side of
this estimate, it follows that
\begin{equation}
\begin{split}
    \mathcal{Q} \le \int_\Gamma
    \int_{-\infty}^\infty \left|\mathcal{R} \widetilde{u}_{*,T_0}(\mathbf{r}, t')\right|^2\,\d
    t'\,\d\sigma(\mathbf{r}) + \int_\Gamma
    \int_{-\infty}^\infty \left|\mathcal{T} \widetilde{u}_{*,T_0}(\mathbf{r}, t')\right|^2\,\d
    t'\,\d\sigma(\mathbf{r}).\label{HtoUbound0}
\end{split}
\end{equation}
Using once again Plancherel's theorem, \Cref{HtoUbound0} becomes
\begin{equation}\label{HtoUbound}
  \begin{split}
      \mathcal{Q} \le
      \int_{-\infty}^\infty \left\|\widehat{\mathcal{R}} \widetilde{U}_{*,T_0}^f(\cdot,
    \omega)\right\|^2_{L^2(\Gamma)}\,\d\omega + \int_{-\infty}^\infty \left\|\widehat{\mathcal{T}}
      \widetilde{U}_{*,T_0}^f(\cdot, \omega)\right\|^2_{L^2(\Gamma)}\,\d\omega,
  \end{split}
\end{equation}
where $\widetilde{U}_{*,T_0}^f$ denotes the temporal Fourier transform of the single layer
potential $\widetilde{u}_{*,T_0}$,
\begin{equation}\label{UstarT0}
    \widetilde{U}_{*,T_0}^f(\bfr, \omega) = \left(\mathscr{S}_\omega \widetilde{\psi}^f_{*,T_0}\right)(\bfr, \omega).
\end{equation}

Using \Cref{ST_op_freq}, \Cref{per_freq_rhs_oper_bound_i} and
\Cref{per_freq_rhs_oper_bound_ii} to bound the integrands on the
right-hand side of \Cref{HtoUbound}, and then using the fact that
$\omega^{2r} \le (1 + \omega^2)^{r}$ for nonnegative $r$, we obtain
\begin{equation*}
  \begin{split}
    &\begin{alignedat}{2}
    \mathcal{Q} \le
        C_1 \int_{-\infty}^\infty &\left(1 + \omega^2\right)
        \left\|\widetilde{\psi}_{*,T_0}^f(\cdot,
    \omega)\right\|^2_{L^2(\Gamma)}\,\d\omega\\
        &+ C_2 \int_{-\infty}^\infty \omega^{2q} \left(1 + \omega^2\right)
        \left\|\widetilde{\psi}_{*,T_0}^f(\cdot,
    \omega)\right\|^2_{L^2(\Gamma)}\,\d\omega
\end{alignedat}\\
&\quad\le C\int_{-\infty}^\infty \left(1 + \omega^2\right)^{q+1}
\left\|\widetilde{\psi}_{*,T_0}^f(\cdot,
  \omega)\right\|^2_{L^2(\Gamma)}\,\d\omega.
  \end{split}
\end{equation*}
for a suitable constant $C$, and the result follows.
\end{proof}

In order to establish the uniform-in-time
estimates~\eqref{density_unif_time_bound}
and~\eqref{density_unif_time_bound_Tn} we utilize a Bochner-space
version of the Sobolev lemma (cf. \Cref{def:sob_bochner}) over a
physical domain $\mathcal{U}$. This result, which is presented in what
follows, is primarily used in the surface-domain case
$\mathcal{U} = \Gamma$, but it is also used for volumetric domains in
\Cref{decay_corr}.  The proof results by merely incorporating the
Bochner space nomenclature in the classical arguments used for the
real-valued case~\cite[Lem.\ 6.5]{Folland}.
\begin{lemma}[Sobolev embedding in Bochner spaces]\label{sob_lemma}
  Let $k$ denote a nonnegative integer, let $s \ge 0$, and let
  $\mathcal{U}$ denote a subset of $\mathbb{R}^3$, where either
  $\mathcal{U} = \Gamma$ equals the Lipschitz boundary of an open and
  bounded domain $\Omega\subset\mathbb{R}^3$, or $\mathcal{U} = D$
  equals an open and bounded domain $D\subset\mathbb{R}^3$ with a Lipschitz
  boundary.  Then for $r > k + 1/2$,
  $H^r(\mathbb{R}; H^s(\mathcal{U}))$ is continuously embedded in
  $C^k(\mathbb{R}; H^s(\mathcal{U}))$:
  $H^r(\mathbb{R}; H^s(\mathcal{U})) \hookrightarrow C^k(\mathbb{R};
  H^s(\mathcal{U}))$. In particular, there exists a constant $C$ such
  that for all functions $a \in H^r(\mathbb{R}; H^s(\mathcal{U}))$ the
  bound
    \begin{equation}\label{eq:sob_lemma}
      \max_{0 \le \ell \le k} \sup_{t \in \mathbb{R}}
      \left\|\frac{\partial^\ell}{\partial t^\ell} a(t)\right\|_{H^s(\mathcal{U})} \le C \left\|a\right\|_{H^r(\mathbb{R}; H^s(\mathcal{U}))}
    \end{equation}
    holds, where $\frac{\partial^\ell}{\partial t^\ell} a(t)$ denotes
    the $\ell$-th classical derivative of $a$ with respect to $t$ (see
    \Cref{sob-boch}).
\end{lemma}
\begin{proof}
  Let $a \in H^r(\mathbb{R}; H^s(\mathcal{U}))$. We first show that
  $\widehat{\partial_t^\ell a} \in L^1(\mathbb{R}; H^s(\mathcal{U}))$ for all
  nonnegative integers $\ell \le k$.  Indeed, using the relation
  $\widehat{\partial_t^\ell a}(\omega) = (\i \omega)^\ell
  \widehat{a}(\omega)$ (in accordance with the Fourier transform
  convention~\cref{eq:fourier_transf}) we obtain
    \begin{equation}\label{eq:bochner_L1_estimate}
        \begin{split}
          &\int_{-\infty}^\infty \left\|\widehat{\partial_t^\ell
              a}\right.\left.  (\omega)\right\|_{H^s(\mathcal{U})}\,\d\omega =
          \int_{-\infty}^\infty \left\|(\i\omega)^\ell
            \widehat{a}(\omega)\right\|_{H^s(\mathcal{U})}\,\d\omega\\
          &\le \int_{-\infty}^\infty (1 + \omega^2)^{k/2}\left\|\widehat{a}(\omega)\right\|_{H^s(\mathcal{U})}\,\d\omega\\
          &= \int_{-\infty}^\infty (1 +
          \omega^2)^{(k-r)/2}\left\|\widehat{a}(\omega)\right\|_{H^s(\mathcal{U})} (1 +
          \omega^2)^{r/2}\,\d\omega\\
          &\le \left(\int_{-\infty}^\infty (1 +
            \omega^2)^{k-r}\,\d\omega \right)^{1/2}
          \left(\int_{-\infty}^\infty (1 +
            \omega^2)^r \left\|\widehat{a}(\omega)\right\|_{H^s(\mathcal{U})}^2\,\d\omega\right)^{1/2}\\
          &\le \widetilde{C} \left\|a\right\|_{H^r(\mathbb{R};\, H^s(\mathcal{U}))},
        \end{split}
    \end{equation}
    where we used the fact that $r > k + 1/2$ to bound the integral of
    $(1 + \omega^2)^{k-r}$ by a constant independent
    of $a$, and where we utilized the equivalent norm of the space
    $H^r(\mathbb{R}; H^s(\mathcal{U}))$ displayed
    in~\eqref{sob_bochner_fourier}. It follows that
    $\widehat{\partial_t^\ell a} \in L^1(\mathbb{R}; H^s(\mathcal{U}))$,
    as claimed.

    By the Bochner-specific version of the Riemann-Lebesgue
    lemma~\cite[Lem.\ 2.4.3]{Weis} (which ensures that the Fourier
    transform of an $L^1(\mathbb{R}; H^s(\mathcal{U}))$ function is a
    norm-continuous function of $\omega$ that tends to zero in norm as
    $\omega\to\infty$), in conjunction with the fact that, as
    established above,
    $\widehat{\partial_t^\ell a} \in L^1(\mathbb{R}; H^s(\mathcal{U}))$,
    it follows that
    \begin{equation}\label{fell_C}
      \widehat{\widehat{\partial_t^\ell a}} \in C(\mathbb{R}; H^s(\mathcal{U})), \quad 0 \le \ell \le k.
    \end{equation}

    Since by hypothesis $a\in H^r(\mathbb{R}; H^s(\mathcal{U}))$, further, it follows
    that for $0 \le \ell \le k$, $\partial_t^\ell a \in L^2(\mathbb{R}; H^s(\mathcal{U}))$, and
    thus by the Bochner Plancherel theorem~\cite[Thm.\ 2.20]{Karunakaran:98} we
    have, additionally,
    $\widehat{\partial_t^\ell a} \in L^1(\mathbb{R}; H^s(\mathcal{U})) \cap L^2(\mathbb{R};
    H^s(\mathcal{U}))$. Now, applying the Bochner Fourier inversion
    theorem~\cite{Karunakaran:98} (see also~\cite[\S 8.4]{Zemanian}),
    we obtain
    \begin{equation}\label{sob_lemma_partial_ell_f}
      \partial_t^\ell a(t) = \frac{1}{2\pi} \int_{-\infty}^\infty
      \widehat{\partial_t^\ell a}(\omega) \e^{\i\omega t}\,\d\omega,\quad 0 \le
      \ell \le k,
    \end{equation}
    which, in view of \Cref{fell_C} shows that
    $\partial_t^\ell a(t) = \frac{1}{2\pi}
    \widehat{\widehat{\partial_t^\ell a}}(-t) \in C(\mathbb{R};
    H^s(\mathcal{U}))$ for $0 \le \ell \le k$ and, therefore, that
    $H^r(\mathbb{R}; H^s(\mathcal{U})) \subset C^k(\mathbb{R};
    H^s(\mathcal{U}))$.  Finally, using \Cref{sob_lemma_partial_ell_f} and
    then \Cref{eq:bochner_L1_estimate} we see that, for
    $0 \le \ell \le k$,
    \begin{equation*}
      \begin{split}
        \sup_{t \in \mathbb{R}} \left\|\frac{\partial^\ell}{\partial t^\ell} a(t)\right\|_{H^s(\mathcal{U})} &= \sup_{t \in
      \mathbb{R}} \frac{1}{2\pi} \left\|\int_{-\infty}^\infty
        \widehat{\partial_t^\ell a}(\omega)
        \e^{\i \omega (-t)}\,\d\omega\right\|_{H^s(\mathcal{U})}\\
        &\le \frac{1}{2\pi} \int_{-\infty}^\infty
        \left\|\widehat{\partial_t^\ell a}(\omega) \right\|_{H^s(\mathcal{U})}\,\d\omega \le C
        \left\|a\right\|_{H^r(\mathbb{R}; H^s(\mathcal{U}))}.
      \end{split}
    \end{equation*}
    This establishes \Cref{eq:sob_lemma} and concludes the proof of
    the lemma.
\end{proof}

\begin{lemma}\label{int_eq_pderiv}
  Let $p,q$ denote non-negative integers, assume $\Gamma$ satisfies
  the $q$-growth condition, and assume $\widetilde{b}$ satisfies the
  $s$-regularity conditions~\eqref{eq:gamma_Hs_assump} with
  $s = p + q + 1$.  Then, letting $\widetilde{\psi}$ denote the
  solution of \Cref{eq:tdie_sl_generic}, we have
  $\widetilde{\psi} \in C^{p}(\mathbb{R}; L^2(\Gamma))$, and
  $\partial_t^p \widetilde{\psi}$ satisfies the integral equation
  \begin{equation}\label{eq:tdie_sl_generic_p}
        \left(S\partial_t^p \widetilde{\psi}\right)(\mathbf{r}, t) = \gamma^+ \partial_t^p \widetilde{b}(\mathbf{r}, t)\quad\mbox{for}\quad (\mathbf{r}, t) \in
    \Gamma\times\mathbb{R}.
  \end{equation}
\end{lemma}
\begin{proof}
  From \Cref{3d_decay_lemma_2ndkind_wellposed} we have
  $\widetilde{\psi} \in H^{p+1}(\mathbb{R};\, L^2(\Gamma))$, and,
  thus, by \Cref{sob_lemma} we obtain
  $\widetilde{\psi} \in C^{p}(\mathbb{R};\,L^2(\Gamma))$ as claimed.
  Similarly, by \Cref{sob_lemma}, $\widetilde{b}$ satisfies
  $\gamma^+ \widetilde{b} \in C^{p+q+1}(\mathbb{R};
  L^2(\Gamma))\subset C^{p}(\mathbb{R}; L^2(\Gamma))$. The proof is
  now completed inductively, by differentiation of
  \Cref{eq:tdie_sl_generic} under the integral sign in the
  expression~\eqref{eq:single_layer_op_time_conv} for the operator
  $S$.
\end{proof}
\noindent
The proof of \Cref{3d_decay_thm} is presented in what follows.
\vspace{0.5cm}
\begin{center}
{\bf Proof of \Cref{3d_decay_thm}}
\end{center}
The proof of~\Cref{density_Hp_time_bound} follows by first obtaining
an $L^2$-in-time estimate for the solution $\widetilde{\psi}$ of
\Cref{eq:tdie_sl_generic} with a generic right-hand side
$\widetilde{b}$ in suitable Sobolev-Bochner spaces, and then applying
that estimate to the particular cases $\widetilde{\psi} = \psi$ (for
$\widetilde{b} = b$) and $\widetilde{\psi} = \partial_t^p \psi$ (for
$\widetilde{b} = \partial_t^p b$, see~\Cref{def:sob_bochner}).

To obtain the necessary estimates for these functions
$\widetilde{\psi}$, in turn, we develop $L^2$-in-time estimates for
the quantities $\widetilde{\psi}_{+,T_0}$ that are related to the
solution $\widetilde{\psi}$ of \Cref{eq:tdie_sl_generic} via
\Cref{timewinddens} and \Cref{rem:tilde_1}.  Since, by hypothesis,
both aforementioned selections $\widetilde{b} = b$ and
$\widetilde{b} = \partial_t^p b$ of $\widetilde b$ satisfy
$\gamma^+\widetilde{b} \in H^{q+1}(\mathbb{R}; H^1(\Gamma))$ and
$\gamma^+\partial_\mathbf{n} \widetilde{b} \in H^{q}(\mathbb{R};
L^2(\Gamma))$, it follows that the conditions of
\Cref{3d_decay_lemma_2ndkind} are met, and, thus
\Cref{3d_decay_lemma_Ainv_freq_bounds} tells us that
$\widetilde{\psi}_{+,T_0}^f$ satisfies the
estimate~\cref{psiplus_estimate_3}.  Using this bound in conjunction
with the relation $\widetilde{\psi} = \widetilde{\psi}_{+,T_0}$ for
$t > T_0$ (see \Cref{psipm_def}) and Plancherel's identity we obtain
\begin{equation}\label{psi_to_psiplus_to_H}
  \begin{alignedat}{2}
      \left\|\widetilde{\psi}\right.\mkern-24mu &&&\left.\vphantom{\widetilde{\psi}}\right\|^2_{L^2([T_0, \infty);\,L^2(\Gamma))} \le
      \left\|\widetilde{\psi}_{+,T_0}\right\|^2_{L^2(\mathbb{R};\,L^2(\Gamma))} =
      \left\|\widetilde{\psi}_{+,T_0}^f\right\|^2_{L^2(\mathbb{R};\,L^2(\Gamma))} \\
    &\le\,&& C_1\int_0^{\omega_0}
      \left\|\gamma^-\partial_\mathbf{n} \widetilde{H}^f_{T_0}(\cdot,
      \omega) - \i \gamma^-\widetilde{H}^f_{T_0}(\cdot,
    \omega)\right\|^2_{L^2(\Gamma)}\,\d\omega\\
    & && + C_2\int_{\omega_0}^\infty \omega^{2q}
      \left\|\gamma^-\partial_\mathbf{n} \widetilde{H}^f_{T_0}(\cdot,
      \omega) - \i\omega \gamma^-\widetilde{H}^f_{T_0}(\cdot.
    \omega)\right\|^2_{L^2(\Gamma)}\,\d\omega.\\
  \end{alignedat}
\end{equation}
Using \Cref{conv_H_to_psi_eq} in \Cref{conv_H_to_psi}, further, to
estimate the integrals on the right-hand side of
\Cref{psi_to_psiplus_to_H} we obtain
\begin{equation}\label{final_psiplus_estimate}
    \left\|\widetilde{\psi}\right\|^2_{L^2([T_0, \infty);\,L^2(\Gamma))} \le \tilde{C} \int_{-\infty}^\infty (1 + \omega^2)^{q+1} \left\|
    \widetilde{\psi}_{*,T_0}^f(\cdot, \omega)\right\|_{L^2(\Gamma)}^2\,\d\omega,
\end{equation}
and, in view of the equivalence of the norms~\eqref{sob_bochner_norm}
and~\eqref{sob_bochner_fourier}, we re-express
\Cref{final_psiplus_estimate} as
\begin{equation}\label{final_psiplus_estimate_2}
    \left\|\widetilde{\psi}\right\|^2_{L^2([T_0, \infty);\,L^2(\Gamma))} \le \tilde{\tilde{C}}\left(\left\|
    \widetilde{\psi}_{*,T_0}\right\|_{L^2(\mathbb{R}; L^2(\Gamma))}^2 + \left\|
    \partial_t^{q+1} \widetilde{\psi}_{*,T_0}\right\|_{L^2(\mathbb{R}; L^2(\Gamma))}^2\right).
\end{equation}
Applying the Leibniz formula to the expression
\[
    \partial_t^{q+1} \widetilde{\psi}_{*,T_0}(\bfr, t) = \partial_t^{q+1} \left( w_+(t - T_0 + T_* + \tau)w_-(t - T_0)
    \widetilde{\psi}(\bfr, t)\right)
\]
and noting that for all $\bfr \in \Gamma$ and for $1 \le i \le p$ we
have
$\supp \partial_t^{i} \widetilde{\psi}_{*,T_0}(\bfr, \cdot) \subset
I_{T_0}$, using straightforward bounds on the functions $w_-$ and
$w_+$ (that do not depend on $T_0$---see \Cref{timewinddens}), we
obtain
\begin{equation}\label{final_psiplus_estimate_2a}
    \left\|\widetilde{\psi}\right\|^2_{L^2([T_0, \infty);\,L^2(\Gamma))} \le \sum_{i=0}^{q+1} C_i\left\|
    \partial_t^i \widetilde{\psi}\right\|_{L^2(I_{T_0}; L^2(\Gamma))}^2, \quad C_i = C_i(\Gamma, \tau).
\end{equation}
Using the continuity of the inclusion map in Sobolev spaces, it follows that
\begin{equation}\label{psiplus_estimate_L2_nontrapping}
    \lVert\widetilde{\psi}\rVert_{L^2([T_0, \infty);\,L^2(\Gamma))} \le
    C\left\|\widetilde{\psi}\right\|_{H^{q+1}(I_{T_0};\,L^2(\Gamma))} < \infty,\quad C = C(\Gamma, \tau),
\end{equation}
where the finiteness of the norm over $I_{T_0}$ in
\Cref{psiplus_estimate_L2_nontrapping} follows from
\Cref{3d_decay_lemma_2ndkind_wellposed}---since that lemma tells us
that $\widetilde{\psi} \in H^{q+1}(I_{T_0};L^2(\Gamma))$, in view of
the assumed hypotheses
$\gamma^+\widetilde{b} \in H^{2q+2}(\mathbb{R};L^2(\Gamma))$ and
$\gamma^+\partial_\mathbf{n} \widetilde{b} \in
H^{2q+1}(\mathbb{R};L^2(\Gamma))$ for each of $\widetilde{b} = b$ and
$\widetilde{b} = \partial_t^p b$. Applying
\Cref{psiplus_estimate_L2_nontrapping} with $\widetilde{b} = b$ yields
\Cref{density_Hp_time_bound} in the case $p = 0$. Furthermore,
\Cref{int_eq_pderiv} with $\widetilde{b} = b$ implies that
$\widetilde{\psi} = \psi$ satisfies \Cref{eq:tdie_sl_generic_p}. But
this equation can be expressed in the form \Cref{eq:tdie_sl_generic}
with $\widetilde{b} = \partial_t^p b$ and
$\widetilde \psi = \partial_t^p \psi$, for which the estimate
\Cref{psiplus_estimate_L2_nontrapping} becomes
\[
    \|\partial_t^p \psi\|_{L^2([T_0, \infty);\,L^2(\Gamma))} \le
    C\left\|\partial_t^p \psi\right\|_{H^{q+1}(I_{T_0};\,L^2(\Gamma))} < \infty,\quad C = C(\Gamma, \tau).
\]
which, together with \Cref{psiplus_estimate_L2_nontrapping} in the case $\widetilde{b} = b$ and $\widetilde{\psi} = \psi$, and using once again the continuity of the inclusion map in Sobolev spaces, implies \Cref{density_Hp_time_bound}.

Applying~\cref{density_Hp_time_bound} with $p = 1$ together with \Cref{sob_lemma}
yields
\begin{equation}
  \begin{split}
    \sup_{t>T_0} \left\|\psi(\cdot, t)\right\|_{L^2(\Gamma)} &\le \widetilde{C}
    \left\|\psi\right\|_{H^{1}([T_0, \infty);\,L^2(\Gamma))}\\
    &\le C(\Gamma, \tau)
    \left\|\psi\right\|_{H^{q+2}(I_{T_0};\,L^2(\Gamma))},
  \end{split}
\end{equation}
and, thus, \cref{density_unif_time_bound}. The proof is now
complete. $\qed$

\section{Super-algebraic decay of boundary densities and local
  energies}\label{sec:theory_part_ii}
This section extends the theoretical results of
\Cref{sec:theory_part_i}: it establishes that not only is it possible
to bound the density $\psi$ in the unbounded time interval
$[T_0, \infty)$ by its values on the preceding bounded subinterval
$I_{T_0}$, as shown in \Cref{3d_decay_thm}, but also, in
\Cref{3d_decay_thm_ii} below, the main theorem of this paper, that the
{\em temporal} Sobolev and maximum norms of the solution $\psi$ on
time intervals of the form $[T_0 + t, \infty)$, $t > 0$, each decay
rapidly (e.g., super-algebraically fast for temporally smooth incident
signals) as $t\to \infty$; see also \Cref{3d_decay_rmk_iii} where a
related but somewhat modified decay result and proof strategy are
suggested.  The statement and proof of \Cref{3d_decay_thm_ii} rely on
the nomenclature introduced in Sections~\ref{sec:prelim}
and~\ref{sec:theory_part_i}. Two important corollaries to this theorem,
namely Corollaries~\ref{decay_corr} and~\ref{decay_corr_energy},
relate \Cref{3d_decay_thm_ii} to rapid decay of solutions of the
wave equation. Following the statement and
proof of \Cref{decay_corr_energy} a brief discussion is presented that lays
out the main lines of the proof of \Cref{3d_decay_thm_ii}; a detailed
proof of this theorem is presented at the end of this
section, following a sequence of preparatory lemmas.

\begin{theorem}\label{3d_decay_thm_ii}
  Let $p$, $q$ and $n$ denote non-negative integers, $n>0$, let
  $T_0 > 0$ and $\tau > 0$ be given, and assume (i) $\Gamma$ satisfies
  the $q$-growth condition (\Cref{q-nontrapp}); (ii) The incident
  field $b(\bfr, t)$ satisfies the $s$-regularity
  conditions~\eqref{eq:gamma_Hs_assump} with
  $s = p + q + (n+1)(q+1)$; as well as, (iii) The
  incident field $b = b(\bfr, t)$ satisfies \Cref{all_t_b} and
  it vanishes for $(\bfr, t) \in \widebar{\Omega}\times\left \{I_{T_0} \cup [T_0, \infty)\right\}$, with $I_{T_0}$ as in
  \Cref{domainofdep}. Then for $t > T_0$ the solution $\psi$ of
  \Cref{eq:tdie_sl} satisfies the $t$-decay estimate
\begin{equation}\label{density_Hp_time_bound_Tn}
  \left\|\psi\right\|_{H^p([t, \infty);\,L^2(\Gamma))} \le
    C(\Gamma, \tau, p, n) (t - T_0)^{1/2-n} \left\|\psi\right\|_{H^{p + (n+1)(q +
    1)}(I_{T_0};\,L^2(\Gamma))} < \infty.
\end{equation}
If $p\geq 1$ then $\psi$ additionally satisfies the temporally
pointwise $t$-decay estimate
\begin{equation}\label{density_unif_time_bound_Tn}
  \left\|\psi(\cdot, t)\right\|_{L^2(\Gamma)} \le
  C(\Gamma, \tau, n) (t - T_0)^{1/2-n} \left\|\psi\right\|_{H^{(n+1)(q + 1) +
      1}(I_{T_0};\,L^2(\Gamma))} < \infty
\end{equation}
for all $t > T_0$.
\end{theorem}
\begin{corr}[Spatial-$L^2$/pointwise solution decay]\label{decay_corr}
  Let $q$ denote a non-negative integer, let $\Gamma$ satisfy the
  $q$-growth condition (\Cref{q-nontrapp}), and assume that the data
  $b$ for the problem~\cref{eq:w_eq}-\cref{eq:bdef} is such that
      \begin{equation}\label{b_smooth}
\gamma^+b \in C^\infty(\mathbb{R}; L^2(\Gamma))\quad\mbox{and}\quad
\gamma^+\partial_\mathbf{n} b \in C^\infty(\mathbb{R};L^2(\Gamma)).
\end{equation}
Further, assume that $b = b(\bfr, t)$ satisfies~\eqref{all_t_b}, and
that, for given $T_0 > 0$ and $\tau > 0$, $b(\bfr, t)$ vanishes for
$(\bfr, t) \in \widebar{\Omega} \times \left\lbrace I_{T_0} \cup [T_0,
  \infty)\right\rbrace$, where $I_{T_0} = I_{T_0, T_*, \tau}$ is
defined in \Cref{domainofdep} with $T_* = \diam(\Gamma)/c$ as
indicated in eq.~\Cref{Tstar_def}.  Then, letting
$D = \Omega^c \cap \left\lbrace |\mathbf{r}| < R\right\rbrace$, for
given $R > 0$ such that
$\widebar{\Omega} \subset \left\lbrace |\mathbf{r}| < R\right\rbrace$,
and defining
$r_{\mathrm{max}} = \sup_{\mathbf{r} \in D, \mathbf{r}' \in \Gamma}
|\mathbf{r} - \mathbf{r}'|$, for each pair of integers $n > 0$ and
$p \ge 0$ there exists a constant $C = C(p, n, \tau, R, \Gamma) > 0$
such that the solution $u$ to \Cref{eq:w_eq} satisfies the $t$-decay
estimate
\begin{equation}\label{uk_l2_sup_bound_decay}
    \left\|\partial_t^p u(\cdot, t)\right\|_{L^2(D)} \le C (t - T_0 -
    r_{\mathrm{max}}/c)^{1/2-n} \left\|\psi\right\|_{H^{p + (n+1)(q+1)+1}(I_{T_0};
    L^2(\Gamma))} < \infty,
\end{equation}
for all $t \in (T_0 + r_{\max}/c, \infty)$.

Further, $u$ also decays super-algebraically fast with increasing time
$t$ at each point $\mathbf{r}$ outside $\widebar{\Omega}$. More
precisely, given any compact set
$\mathcal{R} \subset (\widebar{\Omega})^c$ and defining
$r_{\mathrm{max}} = \max_{\mathbf{r} \in \mathcal{R}, \mathbf{r}' \in
  \Gamma} |\mathbf{r} - \mathbf{r}'|$, for each pair of integers
$n > 0$ and $p \ge 0$ there exists a constant
$C = C(p, n, \tau, \mathcal{R}, \Gamma) > 0$ such that $u$ satisfies
the $t$-decay estimate
\begin{equation}\label{u_ptwise_bound_decay}
      \left|\partial_t^p u(\mathbf{r}, t)\right| \le
      C(t - T_0 - r_{\mathrm{max}}/c)^{1/2-n} \left\|\psi\right\|_{H^{p + (n+1)(q+1)+1}(I_{T_0}; L^2(\Gamma))} < \infty,
\end{equation}
for all $(\mathbf{r}, t) \in \mathcal{R}\times (T_0 + r_{\mathrm{max}}/c,
\infty)$.

\end{corr}
\begin{proof}
  Except for hypothesis (ii) of \Cref{3d_decay_thm_ii}, all conditions
  of that Theorem follow immediately from hypotheses of the present
  corollary. Hypothesis (ii), in turn, follows since by
  hypothesis~\eqref{b_smooth} and the assumed compact support of $b$,
  this function satisfies $\gamma^+b \in H^p(\mathbb{R}; L^2(\Gamma))$
  and
  $\gamma^+ \partial_{\mathbf{n}} b \in H^p(\mathbb{R}; L^2(\Gamma))$
  for every integer $p\ge 0$. Thus, the conditions of
  \Cref{3d_decay_thm_ii} are satisfied for every integer $n > 0$ and
  every integer $p \ge 0$ and we thus have
\begin{equation}\label{psi_decay}
  \begin{split}
    \left\|\partial_t^p \psi\right\|_{H^1([t', \infty);\,L^2(\Gamma))} &\le
  \left\|\psi\right\|_{H^{p+1}([t', \infty);\,L^2(\Gamma))}\\
  &\le C_1(p, n, \Gamma, \tau) (t' - T_0)^{1/2-n} \left\|\psi\right\|_{H^{p + (n+1)(q +
    1) + 1}(I_{T_0};\,L^2(\Gamma))}
  \end{split}
\end{equation}
for $t' > T_0$.

The estimates \Cref{uk_l2_sup_bound_decay} and
\Cref{u_ptwise_bound_decay} are obtained in what follows by relying on
a corresponding estimate on $\partial_t^p u$ in the norm of
$H^1([T_0 + r_{\max}/c, \infty); L^2(\Gamma))$, on one hand, and an
estimate on the norm of $\partial_t^p u(\mathbf{r}, \cdot)$ in
$H^1([t + r_{\max}/c, \infty))$ for each $(\mathbf{r}, t) \in \mathcal{R}\times (T_0 + r_{\mathrm{max}}/c,
\infty)$,
on the other hand.  Estimate~\eqref{uk_l2_sup_bound_decay} is
established by first differentiating $s$ times under the integral sign
in the integral representation~\eqref{eq:kirchhoff_3d_soft} for the
solution $u$ (for certain values of the integer $s$) obtaining
    $\partial_t^s u(\mathbf{r}, t) = \left(\mathscr{S} \partial_t^s
  \psi\right)(\mathbf{r}, t)$, and then applying the Cauchy-Schwarz
inequality to the formula~\cref{eq:single_layer_pot_time_conv} for $\mathscr{S}$
to obtain, for all $t'$,
\begin{equation}
    \begin{split}
      \left\|\partial_t^s u(\cdot, t')\right.&\hspace{-1.5mm}\left.\vphantom{\partial_t^s u(\cdot, t')}\right\|^2_{L^2(D)} = \int_D
      \left|\partial_t^s u(\mathbf{r}, t')\right|^2\,\d V(\mathbf{r})\\
    &\le \frac{1}{16\pi^2} \int_D \left(\int_\Gamma
    \left|\partial_t^s \psi(\mathbf{r}', t' - |\mathbf{r} -
    \mathbf{r}'|/c)\right|^2\,\d\sigma(\mathbf{r}')\right)
    \left(\int_\Gamma \frac{\d\sigma(\mathbf{r}')}{|\mathbf{r} -
    \mathbf{r}'|^2}\right)\d V(\mathbf{r}).
    \end{split}
    \label{uk_L2_estimate}
\end{equation}
Integrating this bound in $t'$ for $t'\geq t$, for a given $t$, we
obtain
\begin{equation}
    \begin{split}
      \left\|\partial_t^s u\right\|^2_{L^2([t, \infty); L^2(D))} &\le
      \frac{1}{16\pi^2} \int_{D} \left[ \left(\int_t^\infty \int_\Gamma
      |\partial_t^s \psi(\mathbf{r}', t' - |\mathbf{r} - \mathbf{r}'|/c)|^2\,
      \d\sigma(\mathbf{r}')\,\d t'\right) \times\right.\\
      &\left.\quad\quad\quad\quad \times \left(\int_\Gamma \frac{\d\sigma(\mathbf{r}')}{|\mathbf{r} -
    \mathbf{r}'|^2}\right)\right]\,\d V(\mathbf{r}).
    \end{split}
    \label{uk_L2_estimate_integrated}
\end{equation}
  Letting $I_{s,1}(\mathbf{r}, t)$ and $I_2(\mathbf{r})$ denote, respectively, the first
  and second factors in the integrand of the  integral over
  $D$,
  \begin{equation}\label{I1_I2_decay_corr}
    I_{s,1}(\mathbf{r}, t) = \int_t^\infty \int_\Gamma
      |\partial_t^s\psi(\mathbf{r}', t' - |\mathbf{r} - \mathbf{r}'|/c)|^2\,
      \d\sigma(\mathbf{r}')\,\d t',\;\mbox{and}\; I_2(\mathbf{r}) = \int_\Gamma \frac{\d\sigma(\mathbf{r}')}{|\mathbf{r} -
    \mathbf{r}'|^2},
  \end{equation}
  we seek to bound each of $I_{s,1}(\mathbf{r}, t)$ and $I_2(\mathbf{r})$ by quantities independent of $\mathbf{r}$.
  Considering first $I_{s,1}(\mathbf{r}, t)$ for fixed $\mathbf{r} \in D$, a change in integration order and a slight extension of the integration domain in the $t'$ variable yields
  \begin{equation}\label{I1_decay_corr_estimate}
    I_{s,1}(\mathbf{r}, t) \le \int_\Gamma \int_{t - r_{\max}/c}^\infty |\partial_t^s \psi(\mathbf{r}', t')|^2\,\d t'\,\d\sigma(\mathbf{r}')= \left\|\partial_t^s\psi\right\|^2_{L^2([t - r_{\max}/c, \infty); L^2(\Gamma))}.
  \end{equation}
    The integral of $I_2(\mathbf{r})$ over $D$, in turn, is bounded by
an $R-$ and $\Gamma$-dependent constant since by switching the order of integration we obtain
    \begin{equation*}
      \int_{D} I_2(\mathbf{r}) \,\d V(\mathbf{r}) = \int_\Gamma \left(\int_{D} \frac{1}{|\mathbf{r} - \mathbf{r}'|^2}\,\d V(\mathbf{r})\right)\d\sigma(\mathbf{r}'),
    \end{equation*}
    where the inner integral is bounded for each
    $\mathbf{r}' \in \Gamma$ by an $R-$ and $\Gamma$-dependent
    constant. (The latter statement can easily be established by using
    a spherical coordinate system centered at each
    $\mathbf{r}' \in \Gamma$ for integration in $\mathbf{r} \in D$.)
    We have therefore established that
    \begin{equation}
      \label{u_l2_to_psi}
      \left\| \partial_t^s u\right\|^2_{L^2([t, \infty); L^2(D))} \le C_2\left\|\partial_t^s \psi\right\|^2_{L^2([t - r_{\max}/c, \infty); L^2(\Gamma))}.
    \end{equation}
    Taking separately $s = p$ and $s = p + 1$ in \Cref{u_l2_to_psi} and adding the results we obtain
    \begin{equation}\label{u_h1_to_psi}
      \left\| \partial_t^p u\right\|^2_{H^1([t, \infty); L^2(D))} \le C_3\left\|\partial_t^p\psi\right\|^2_{H^1([t - r_{\max}/c, \infty); L^2(\Gamma))},
    \end{equation}
    Using \Cref{u_h1_to_psi} in conjunction
    with \Cref{psi_decay} and the Sobolev embedding \Cref{sob_lemma} we obtain
    \begin{equation*}
      \begin{split}
        \left\|\partial_t^p u(\cdot, t)\right\|_{L^2(D)} &\le C_4 \left\|\partial_t^p u\right\|_{H^1([t, \infty); L^2(D))} \le \sqrt{C_3} C_4\left\|\partial_t^p\psi\right\|_{H^1([t - r_{\max}/c, \infty); L^2(\Gamma))}\\
        &\le C_5(\Gamma, \tau, n) (t - T_0 - r_{\max}/c)^{1/2-n} \left\|\psi\right\|_{H^{p + (n+1)(q +
  1) + 1}(I_{T_0};\,L^2(\Gamma))},
      \end{split}
    \end{equation*}
    and, thus, \Cref{uk_l2_sup_bound_decay} follows.

    To establish \Cref{u_ptwise_bound_decay}, in turn, we once again
    use $s$-times differentiation under the integral sign in the
    variable $t$ in the representation~\Cref{eq:kirchhoff_3d_soft}
    followed by the Cauchy-Schwarz inequality in the
    formula~\cref{eq:single_layer_pot_time_conv} for $\mathscr{S}$ to
    obtain
\begin{equation}\label{u_CS_to_psi_sup}
    \begin{split}
  |\partial_t^s u(\mathbf{r}, t')|^2 &\le \left( \int_\Gamma \frac{1}{|\mathbf{r} -
    \mathbf{r}'|^2}\,\d\sigma(\mathbf{r}')\right) \left(\int_\Gamma
    |\partial_t^s \psi(\mathbf{r}', t' - |\mathbf{r} - \mathbf{r}'|/c)|^2
    \,\d\sigma(\mathbf{r}')\right)\\
      &\le C_6^2 \int_\Gamma
    |\partial_t^s \psi(\mathbf{r}', t' - |\mathbf{r} - \mathbf{r}'|/c)|^2
    \,\d\sigma(\mathbf{r}'),
    \end{split}
\end{equation}
where
$C_6^2 = \sup_{\mathbf{r}\in \mathcal{R}} \int_\Gamma
\frac{1}{|\mathbf{r} - \mathbf{r}'|^2}\,\d\sigma(\mathbf{r}')$.
Integrating in $t'$ and using \Cref{I1_decay_corr_estimate} to estimate
the resulting quantity $I_{s,1}(\mathbf{r}, t)$, for each
$\mathbf{r} \in \mathcal{R}$ we obtain the $L^2$ bound
\begin{equation}\label{u_ptwise_l2_to_psi}
  \int_{t}^\infty |\partial_t^s u(\mathbf{r}, t')|^2\,\d t' \le C_6 \left\|\partial_t^s \psi\right\|^2_{L^2([t - r_{\max}/c, \infty); L^2(\Gamma))}.
\end{equation}
Taking separately $s = p$ and $s = p + 1$ in \Cref{u_ptwise_l2_to_psi}, adding the results, and using \Cref{psi_decay} in conjunction with the Sobolev embedding \Cref{sob_lemma}, we obtain, again for $(\mathbf{r}, t) \in \mathcal{R}\times (T_0 + r_{\mathrm{max}}/c,
\infty)$,
    \begin{equation*}
      \begin{split}
        \left|\partial_t^p u(\mathbf{r}, t)\right| &\le C_7 \left\|\partial_t^p u(\mathbf{r}, \cdot)\right\|_{H^1([t, \infty))} \le C_8\left\|\partial_t^p \psi\right\|_{H^1([t - r_{\max}/c, \infty); L^2(\Gamma))}\\
        &\le C_9(\Gamma, \tau, n) (t - T_0 - r_{\max}/c)^{1/2-n} \left\|\psi\right\|_{H^{p + (n+1)(q +
  1) + 1}(I_{T_0};\,L^2(\Gamma))} < \infty,
      \end{split}
    \end{equation*}
establishing \Cref{u_ptwise_bound_decay}. The proof is complete.
\end{proof}

While equation~\eqref{u_ptwise_bound_decay} provides a concrete
solution decay estimate at each spatial point $\mathbf{r}$ outside
$\widebar{\Omega}$, the associated constant $C$ does increase without
bound as $\mathbf{r}$ approaches $\Omega$ (cf.\
\Cref{u_CS_to_psi_sup}). The result presented next, in contrast, which
holds for arbitrary bounded subsets of $\widebar{\Omega}^c$,
establishes decay of the local energy expression~\eqref{local_energy}
considered in much of the literature---thus enriching the previous
estimate~\eqref{uk_l2_sup_bound_decay} by incorporating a norm
containing spatial derivatives. The proof below relies on the
estimates provided in \Cref{decay_corr} in conjunction with standard
elliptic regularity properties.
\begin{corr}[Local energy decay estimates]\label{decay_corr_energy}
  Let $\Gamma$ and $b$ satisfy the hypotheses of \Cref{decay_corr} and
  let $u$ denote the solution to the problem~\eqref{eq:w_eq}. Then,
  the local energy $E$ (equation~\eqref{local_energy}) associated with
  the solution $u$ decays super-algebraically fast as
  $t\to\infty$. More precisely, letting $R > 0$ be such that the
  radius-$R$ ball $B_R =\{ |\mathbf{r}| \le R\}$ contains
  $\widebar\Omega$, defining the compact region
  $D = \Omega^c \cap  B_R$,
  and letting
  $r_{\mathrm{max}} = \sup_{\mathbf{r} \in D, \mathbf{r}' \in \Gamma}
  |\mathbf{r} - \mathbf{r}'|$, for each integer $n > 0$ there exists a
  constant $C = C(\Gamma, R, \tau, n) > 0$ such that
  \begin{equation}\label{local_energy_superalg_decay}
    E(u, D, t) \le C (t - T_0 - r_{\max}/c)^{1 - 2n} \left\| \psi \right\|^2_{H^{(n+1)(q+1)+3}(I_{T_0}; L^2(\Gamma))} < \infty,
  \end{equation}
  for all $t \in (T_0 + r_{\max}/c, \infty)$.
\end{corr}
\begin{proof}
  In view of the bound on
  $\int_D |u_t(\mathbf{r}, t)|^2\,\d V(\mathbf{r})$ provided by
  \Cref{decay_corr}, it suffices to establish a corresponding estimate
  for $\int_D |\nabla u(\mathbf{r}, t)|^2\,\d V(\mathbf{r})$ for
  $t \in (T_0 + r_{\max}/c, \infty)$. To do this, we use the fact that
  $u$ is a solution of the wave equation~\eqref{eq:w_eq_a} as well as
  the hypothesis that $\gamma^+ b(\mathbf{r}', t') = 0$ for
  $(\mathbf{r}', t') \in \Gamma \times \left\{ I_{T_0} \cup [T_0,
    \infty)\right\}$, which together imply that, for each
  $t' \in I_{T_0} \cup [T_0, \infty)$, $u$ satisfies the exterior
  Dirichlet problem
\begin{subequations}\label{u_ellipt}
  \begin{align}
    \Delta u(\mathbf{r}, t') &= f(\mathbf{r}, t'), \quad\mbox{for}\quad \mathbf{r} \in \Omega^c,\label{u_ellipt_a}\\
    u(\mathbf{r}, t') &=  0,
                       \quad\mbox{for}\quad \mathbf{r}\in\Gamma = \partial\Omega,\label{u_ellipt_b}
    \end{align}
  \end{subequations}
  for the Poisson equation~\eqref{u_ellipt_a}, where, for
  $\mathbf{r} \in \Omega^c$,
  $f(\mathbf{r}, t') = \frac{1}{c^2} \partial_t^2 u(\mathbf{r},
  t')$. Exploiting, in addition, the relation
  \Cref{uk_l2_sup_bound_decay} with $p = 2$, we obtain, for
  $t > T_0 + r_{\max}/c$,
  \begin{equation}\label{f_l2_decay_estimate}
    \left\|f(\cdot, t)\right\|_{L^2(D)} \le \frac{C_1}{c^2}(t - T_0 -
    r_{\mathrm{max}}/c)^{1/2-n} \left\|\psi\right\|_{H^{(n+1)(q+1)+3}(I_{T_0};
    L^2(\Gamma))} < \infty,
  \end{equation}
  which, in particular, implies that
  \begin{equation}\label{f_poisson_def}
    f(\cdot, t) \in L^2(D) \quad\mbox{for}\quad t \in (T_0 + r_{\max}/c, \infty).
  \end{equation}

  To establish regularity for the Poisson solution we introduce the
  region
  $D_\varepsilon = \Omega^c \cap \left\lbrace |\mathbf{r}| \le R +
    \varepsilon\right\rbrace$ for some $\varepsilon > 0$ as well as,
  for each fixed $t\in (T_0 + r_{\max}/c, \infty)$, a function
  $\varphi: D_\varepsilon \to \mathbb{R}$ satisfying
  $\varphi(\bfr) = u(\bfr, t)$ in a neighborhood of
  $|\mathbf{r}| = R + \varepsilon$ as well as $\varphi = 0$ in both, a
  neighborhood of $\partial\Omega$ and in the region
  $|\mathbf{r}| > R + 2\varepsilon$---which can be easily constructed
  by multiplication of $u$ by a smooth cutoff function in the radial
  variable.  Moreover, using the
  representation~\eqref{eq:kirchhoff_3d_soft} for
  $\mathbf{r} \in D_\varepsilon$ bounded away from $\partial \Omega$,
  we see that $\varphi \in H^1(D_\varepsilon)$ for every $t$ since
  $\psi \in C^\infty(\mathbb{R}; L^2(\Gamma))$ per the assumptions in
  \Cref{decay_corr}.  Applying the regularity result~\cite[Thm.\
  8.9]{GilbargTrudinger} to the problem~\eqref{u_ellipt} on
  $D_\varepsilon$ with $f(\cdot, t) \in L^2(D_\varepsilon)$ and with
  boundary values given by $\varphi \in H^1(D_\varepsilon)$, we have
  $u(\cdot, t) \in H^1(D_\varepsilon)\cap
  H^2_{\mathrm{loc}}(D_\varepsilon)$. Since $D \subset D_\varepsilon$
  we also have $u(\cdot, t) \in H^1(D)\cap H^2_{\mathrm{loc}}(D)$. In
  particular, for each fixed $t $ we have $u(\cdot, t) \in H^1(D)$ and
  $\Delta u(\cdot, t) \in L^2(D)$. As a result, we can apply the
  version~\cite[Thm.\ 4.4]{McLean} of Green's first identity to
  obtain, for each $t \in (T_0 + r_{\max}/c, \infty)$,
  \begin{equation}\label{gradu_grns_identity}
    \int_D \left|\nabla u(\mathbf{r}, t)\right|^2\,\d V(\mathbf{r}) = \left\langle \gamma u(\cdot, t), \gamma \partial_\mathbf{n} u(\cdot, t) \right\rangle_{\partial D} - \int_D u(\mathbf{r}, t) \Delta u(\mathbf{r}, t)\,\d V(\mathbf{r}),
  \end{equation}
  where $\left\langle \cdot, \cdot \right\rangle$ denotes the duality
  pairing of $H^{-1/2}(\partial D)$ and $H^{1/2}(\partial D)$ and
  where $\gamma$ denotes the trace operator on $D$. In fact, since
  $\gamma u(\cdot, t) = 0$ on $\Gamma$, and since $u(\cdot, t)$ is
  smooth for $\mathbf{r}$ in a neighborhood of $\partial B_R$, we have
  $\left\langle \gamma u(\cdot, t), \gamma \partial_\mathbf{n}
    u(\cdot, t)\right\rangle_{\partial D} = \int_{\partial B_R}
  u(\mathbf{r}, t) \partial_\mathbf{n} u(\mathbf{r},
  t)\,\d\sigma(\mathbf{r})$.

  Using \Cref{gradu_grns_identity} we now obtain a bound
  on
  $E(u, D, t)$. Since $u(\cdot, t)$ satisfies \Cref{u_ellipt}, using the
  Cauchy-Schwarz inequality we obtain
  \begin{equation}\label{gradu_first_bound}
    \begin{split}
      \int_D \left|\nabla u(\mathbf{r}, t)\right|^2\,\d & V(\mathbf{r}) = \int_{\partial B_R} u(\mathbf{r}, t) \partial_{\mathbf{n}} u(\mathbf{r}, t) \,\d\sigma(\mathbf{r}) - \int_D u(\mathbf{r}, t) f(\mathbf{r}, t)\,\d V(\mathbf{r})\\
      &\le \left\| u(\cdot, t)\right\|_{L^2(\partial B_R)} \left\|\partial_{\mathbf{n}} u(\cdot, t)\right\|_{L^2(\partial B_R)} + \left\|u(\cdot, t)\right\|_{L^2(D)} \left\|f(\cdot, t)\right\|_{L^2(D)}.
    \end{split}
  \end{equation}
  The first volumetric term on the right-hand side in
  \Cref{gradu_first_bound} can be estimated using \Cref{decay_corr}
  with $p = 0$:
  \begin{equation}\label{u_volume_decay}
    \left\|u(\cdot, t)\right\|_{L^2(D)} \le C_2 (t - T_0 - r_{\max}/c)^{1/2-n} \left\|\psi\right\|_{H^{(n+1)(q+1)+1}(I_{T_0}; L^2(\Gamma))}.
  \end{equation}
  Then, using \Cref{f_l2_decay_estimate} and the continuity of the
  inclusion map in Sobolev spaces, \Cref{u_volume_decay} provides a bound for the second summand on the right-hand side of~\eqref{gradu_first_bound}:
  \begin{equation}\label{gradu_volume_bound}
    \left\|u(\cdot, t)\right\|_{L^2(D)} \left\|f(\cdot, t)\right\|_{L^2(D)} \le C_3 (t - T_0 - r_{\max}/c)^{1-2n} \left\|\psi\right\|_{H^{(n+1)(q+1)+3}(I_{T_0}; L^2(\Gamma))}^2.
  \end{equation}

  We now to turn to the first summand on the right-hand side
  of~\eqref{gradu_first_bound}, and we estimate each one of the
  corresponding boundary terms by relying on the representation
  formula~\cref{eq:kirchhoff_3d_soft}. Indeed, differentiating that
  formula once respect to the normal $\mathbf{n}$ and $s$-times with
  respect to time with $s = 0$, $1$, and adding the results, for
  $(\mathbf{r},t') \in \partial B_R \times \mathbb{R}$, we obtain the
  relation
  \begin{equation}
    \begin{split}
        \partial_{\mathbf{n}} \partial_t^s u(\mathbf{r}, t') = -\int_\Gamma &\left( \frac{\hat{\mathbf{r}}\cdot (\mathbf{r} - \mathbf{r}') \partial_t^s \psi(\mathbf{r}', t' - |\mathbf{r} - \mathbf{r}'|/c)}{|\mathbf{r} - \mathbf{r}'|^3}\right.\\&\left.+ \frac{\hat{\mathbf{r}} \cdot (\mathbf{r} - \mathbf{r}') \partial_t^{s+1} \psi(\mathbf{r}', t' - |\mathbf{r} - \mathbf{r}'|/c)}{c|\mathbf{r} - \mathbf{r}'|^2}\right)\,\d\sigma(\mathbf{r}'),
    \end{split}
  \end{equation}
  for which the Cauchy-Schwarz inequality implies, for
  $(\mathbf{r},t') \in \partial B_R \times \mathbb{R}$, that
  \begin{equation}\label{nablau_ptwise_to_l2_bound}
    \begin{split}
      |\partial_{\mathbf{n}} \partial_t^s u(\mathbf{r}, t')|^2 &\le C_4(\Gamma, R, \varepsilon) \int_\Gamma |\partial_t^s \psi(\mathbf{r}, t' - |\mathbf{r} - \mathbf{r}'|/c)|^2\,\d\sigma(\mathbf{r}')\\
        &+ C_5(\Gamma, R, \varepsilon) \int_\Gamma |\partial_t^{s+1} \psi(\mathbf{r}', t' - |\mathbf{r} - \mathbf{r}'|/c)|^2\,\d\sigma(\mathbf{r}').
    \end{split}
  \end{equation}
  The argument we use to bound $\partial_{\mathbf{n}} u$ uniformly in
  time is similar to the one used in the proof of \Cref{decay_corr}:
  we obtain bounds on $\partial_{\mathbf{n}} u(\mathbf{r}, \cdot)$ in
  the norm of $H^1([t, \infty))$ by taking $s = 0$ and $s = 1$ in
  \Cref{nablau_ptwise_to_l2_bound}, adding the results, and
  integrating the resulting inequality over $t'$ for $t' > t$; the
  uniform-in-time bound then follows from the Sobolev Lemma. Indeed,
  applying \Cref{nablau_ptwise_to_l2_bound} with $s=0$ and $s=1$ and
  using the definition~\eqref{I1_I2_decay_corr} of
  $I_{0, 1}(\mathbf{r}, t)$, $I_{1,1}(\mathbf{r}, t)$ and
  $I_{2,1}(\mathbf{r}, t)$, we obtain
  \begin{equation}\label{nablau_ptwise_to_l2_bound_integrated}
    \begin{split}
      \left\| \partial_{\mathbf{n}} u(\mathbf{r}, \cdot) \right\|^2_{H^1([t, \infty))} &\le C_4(\Gamma, R, \varepsilon) I_{0,1}(\mathbf{r}, t) + C_5(\Gamma, R, \varepsilon) I_{2,1}(\mathbf{r}, t)\\
      &\quad + (C_4(\Gamma, R, \varepsilon) + C_5(\Gamma, R, \varepsilon))I_{1,1}(\mathbf{r}, t).
    \end{split}
  \end{equation}
  Then, using the bound~\eqref{I1_decay_corr_estimate}, $I_{s,1}(\mathbf{r}, t)
  \le \left\|\partial_t^s \psi\right\|^2_{L^2([t - r_{\max}/c, \infty);
  L^2(\Gamma))}$, we obtain the relation
  \begin{equation}\label{nablau_H1_bound}
    \left\|\partial_{\mathbf{n}} u(\mathbf{r}, \cdot)\right\|_{H^1([t, \infty)} \le C_6(\Gamma, R, \varepsilon) \left\|\psi\right\|_{H^2([t - r_{\max}/c, \infty); L^2(\Gamma))}.
  \end{equation}
  In conjunction with the Sobolev embedding \Cref{sob_lemma} and the bound \Cref{density_Hp_time_bound_Tn} of \Cref{3d_decay_thm_ii}, the bound~\Cref{nablau_H1_bound} implies that
  \begin{equation}\label{nablau_ptwise_bound}
    \begin{split}
      \left|\partial_{\mathbf{n}} u(\mathbf{r}, t)\right| &\le C_7 \left\|\partial_{\mathbf{n}} u(\mathbf{r}, \cdot)\right\|_{H^1([t, \infty)} \le C_7 C_6(\Gamma, R, \varepsilon) \left\|\psi\right\|_{H^2([t - r_{\max}/c, \infty); L^2(\Gamma))}\\
      &\le C_8(\Gamma, R, \tau, n, \varepsilon) (t - T_0 - r_{\max}/c)^{1/2 - n} \left\|\psi\right\|_{H^{(n+1)(q+1)+2}(I_{T_0}; L^2(\Gamma))}
    \end{split}
  \end{equation}
  for
  $(\mathbf{r}, t) \in \partial B_R \times (T_0 + r_{\max}/c,
  \infty)$. Taking $L^2(\partial B_R)$ norms in both
  \Cref{u_ptwise_bound_decay} with $p=0$ and in
  \Cref{nablau_ptwise_bound} yields the desired estimate for the first
  summand on the right-hand side of~\eqref{gradu_first_bound}:
  \begin{equation}\label{unablau_ptwise_bound}
    \begin{split}
      \left\|u(\cdot, t)\right\|_{L^2(\partial B_R)} \left\|\partial_{\mathbf{n}} u(\cdot, t)\right\|_{L^2(\partial B_R)} \le C_9 &(t - T_0 - r_{\max}/c)^{1-2n} \times\\
    &\times \left\|\psi\right\|^2_{H^{(n+1)(q+1)+2}(I_{T_0}; L^2(\Gamma))}
    \end{split}
  \end{equation}
  for $t\geq T_0 + r_{\max}/c$, where we once again used the
  continuity of the inclusion map in Sobolev spaces.

  As suggested above, to obtain decay estimates for the second term
  in~\eqref{local_energy} we use the bound
  \Cref{uk_l2_sup_bound_decay} with $p = 1$, which tells us that
  \begin{equation}\label{dt2u_decay}
    \left\|\partial_t u(\cdot, t)\right\|_{L^2(D)} \le C_{10} (t - T_0 -
    r_{\mathrm{max}}/c)^{1/2-n} \left\|\psi\right\|_{H^{(n+1)(q+1)+2}(I_{T_0};
      L^2(\Gamma))}
\end{equation}
for $t \in (T_0 + r_{\max}/c, \infty)$.

To complete the proof we now utilize~\Cref{gradu_first_bound} with
right-hand side terms substituted by~\eqref{gradu_volume_bound}
and~\eqref{unablau_ptwise_bound}, together with the
bound~\eqref{dt2u_decay}, and obtain the bound
  \begin{equation}
    \begin{split}
      E(u, D, t) &\le \left\| u(\cdot, t)\right\|_{L^2(\partial B_R)} \left\|\partial_{\mathbf{n}} u(\cdot, t)\right\|_{L^2(\partial B_R)}\\
      &\quad + \left\|u(\cdot, t)\right\|_{L^2(D)} \left\|f(\cdot, t)\right\|_{L^2(D)} + \left\|\partial_t u(\cdot, t)\right\|^2_{L^2(D)}\\
      &\le C_9 (t - T_0 - r_{\max}/c)^{1-2n} \left\|\psi\right\|^2_{H^{(n+1)(q+1)+2}(I_{T_0}; L^2(\Gamma))}\\
      &\quad + C_3^2 (t - T_0 - r_{\max}/c)^{1-2n} \left\|\psi\right\|_{H^{(n+1)(q+1)+3}(I_{T_0}; L^2(\Gamma))}^2\\
      &\quad + C_{10}^2 (t - T_0 -
    r_{\mathrm{max}}/c)^{1-2n} \left\|\psi\right\|^2_{H^{(n+1)(q+1)+2}(I_{T_0};
    L^2(\Gamma))}.
    \end{split}
  \end{equation}
  The desired estimate~\eqref{local_energy_superalg_decay} now
  follows, once again, by virtue of the continuity of the inclusion
  map in Sobolev spaces.
\end{proof}

\begin{rem}\label{ralston}
  The well known contribution~\cite{Ralston:69} shows that, for a
  trapping obstacle $\Omega$ and for an arbitrarily large time
  $\mathcal{T}$, initial conditions can be selected so that the local
  energy $E(u, D, t)$ in~\eqref{energy_decay} with
  $D =\Omega^c\cap \{|\bfr| < R\}$ remains arbitrarily close to the
  initial energy value $E(u, D, 0)$ for $0\leq t\leq \mathcal{T}$. In
  particular, the result~\cite{Ralston:69} implies that a decay bound
  of the form~\eqref{energy_decay} that is uniformly valid for all
  admissible incident initial conditions and for all compact sets
  $D\subset \Omega^c$ cannot hold for a trapping
  obstacle. References~\cite{Ikawa:82,Ikawa:88} do present, however,
  uniformly valid decay estimates relative to higher order Sobolev
  norms over the complete exterior domain, which are valid for
  trapping structures consisting of certain unions of convex
  obstacles, for which the trapping rays span spatial regions of zero
  measure (indeed, a single ray in the case of the structure
  consisting of two convex connected obstacles in~\cite{Ikawa:82}, and
  more generally, as implied by Assumption (H.2)
  in~\cite{Ikawa:88}). In the same spirit, \Cref{3d_decay_thm_ii} and
  Corollaries~\ref{decay_corr} and~\ref{decay_corr_energy} present
  decay results for trapping obstacles satisfying the $q$-growth
  condition (including a result of decay for the local energy
  $E(u, D, t)$) in terms of higher-order {\em surface} Sobolev norms,
  that are uniformly valid for all admissible incident fields. In
  particular, these results apply to obstacles such as those depicted
  in \Cref{fig:3d_connected_trapping} (for which the trapping rays
  span volumetric regions of positive measure), in addition to the
  examples~\cite{Ikawa:82,Ikawa:88} which, per
  reference~\cite{Spence:20}, satisfy the $q$-growth condition with
  $q=1$.
\end{rem}

The overall approach to the proof of \Cref{3d_decay_thm_ii} relies
critically on the time-history domain-of-dependence ideas described in
\Cref{rem:h_equiv_rem_2} (see also \Cref{3d_decay_rmk_iii} where a
related but somewhat modified decay result and proof strategy are
suggested). Technically, the proof of \Cref{3d_decay_thm_ii} proceeds
on the basis of an argument resulting from integration by parts with
respect to the temporal frequency $\omega$, which requires
$\omega$-differentiation of a certain function
$\breve{\psi}^f_{+,0}(\bfr, \omega)$ closely related to the boundary
integral density $\psi^f(\bfr, \omega)$.  Certain necessary results
concerning differentiability of boundary integral operators and
associated integral densities are established in a series of lemmas
presented in the Appendix.  Some of the main elements of the proof of
\Cref{3d_decay_thm_ii}, in turn, are presented in
Lemmas~\ref{3d_decay_thm_recursive_bound_general}
through~\ref{decay_estimate_L2} below. Thus, with reference to the
Appendix, \Cref{3d_decay_thm_recursive_bound_general} provides
pointwise bounds on density derivatives. Then, the technical
\Cref{Rderiv_sobolev_bound} (of a similar character to
\Cref{conv_H_to_psi}) establishes bounds on the integrals of certain
quantities in~\Cref{3d_decay_thm_recursive_bound_general}, and
\Cref{decay_estimate_L2} provides the primary decay estimate used in
the proof of \Cref{3d_decay_thm_ii}.

\begin{rem}\label{3d_decay_rmk_iii}
  Before proceeding with the proof of \Cref{3d_decay_thm_ii} we note
  that a related but somewhat less informative decay result and proof
  strategy, which do not depend on the bootstrap DoD concept, can be
  contemplated. In the alternative approach the decay proof results
  once again from an argument based on integration by parts (in
  frequency-domain) in the Fourier integral that represents the
  time-domain solution. The proof of such a result proceeds as
  follows: starting from the given time-domain data $ b(r,t)$ defined
  in the neighborhood $\Omega^\textit{inc}$ of $\Omega$, a Fourier
  transform is performed to obtain the right-hand side of the integral
  equation~\eqref{CFIE_direct}. Since $\Gamma$ satisfies the
  $q$-growth condition, it follows that the solution
  $\psi^f(\bfr,\omega)$ of this equation admits an upper bound that
  grows polynomially as a function of $\omega$. Using an argument
  similar to the one presented in
  \Cref{3d_decay_thm_recursive_bound_general} below, corresponding
  polynomially growing bounds are obtained for the $\omega$
  derivatives of $\psi^f(\bfr,\omega)$. Thus, the decay result can be
  obtained by an integration by parts argument followed by a bound on
  the integral of the resulting integrands. Such a bound can be
  obtained by relying on smooth partitioning of the integral together
  with a Young inequality-based estimate in an argument similar to the
  one in \Cref{decay_estimate_L2}. The resulting time-decay bound
  resembles the bound~\eqref{density_Hp_time_bound_Tn}. But, unlike
  equation~\eqref{density_Hp_time_bound_Tn}, which expresses decay in
  terms of the norm of the solution $\psi$ itself over the bootstrap
  DoD $I_{T_0}$, the alternative bound expresses decay in terms of the
  norm, {\em over all time}, of the right-hand side function
  $b(\bfr,t)$ and its normal derivative on $\Gamma$. Thus
  \Cref{3d_decay_thm_ii} provides a significantly more precise decay
  estimate, but it does so at the cost of certain added complexity in
  the proof---which is mainly confined to \Cref{Rderiv_sobolev_bound}.
\end{rem}

Following the aforementioned proof plan for
Theorem~\ref{3d_decay_thm_ii} we first present, in
\Cref{3d_decay_thm_recursive_bound_general}, estimates on
frequency-derivatives of certain frequency-domain density solutions.

\begin{lemma}\label{3d_decay_thm_recursive_bound_general}
  Let $q$ denote a non-negative integer, assume $\Omega$ satisfies the
  $q$-growth condition (\Cref{q-nontrapp}), and let $\widetilde{b}$
  satisfy the assumptions of Lemma~\ref{3d_decay_thm_h_equiv} (so
  that, in particular, $\widetilde{b} (\mathbf{r}, t)$ is $C^2$ and
  temporally compactly supported in the time interval
  $[\alpha, T-T_*-2\tau]$ for all $\bfr \in \widebar{\Omega}$) as well
  as the $s$-regularity conditions~\eqref{eq:gamma_Hs_assump} with
  $s=q$.  Further, let $\widetilde{\psi}$ and
  $\mu = \widetilde{\psi}_{+,T}^f$ denote the solutions to
  \Cref{eq:tdie_sl_generic} and \cref{CFIE_proof_generic},
  respectively, and, with reference to \cref{CFIE_proof_generic},
  define $\widetilde{R}_T(\mathbf{r}, \omega)$ by
  \begin{equation}\label{RT_def}
    \widetilde{R}_T(\mathbf{r}, \omega)=   \gamma^-\partial_\mathbf{n} \widetilde{H}^f_T(\mathbf{r}, \omega)
    - \i  \eta_0(\omega) \gamma^- \widetilde{H}^f_T(\mathbf{r}, \omega), \quad\mathbf{r} \in \Gamma,
  \end{equation}
  for $\omega \ge 0$, and by Hermitian symmetry for $\omega < 0$:
$\widetilde{R}_T(\mathbf{r},
\omega)=\widebar{\widetilde{R}_T(\mathbf{r}, -\omega)}$.  Then
  $\mu = \widetilde{\psi}_{+,T}^f \in C^\infty(\mathbb{R};
  L^2(\Gamma))$ and
  $\widetilde{R}_T \in C^\infty(\mathbb{R}\setminus \pm \omega_0;
  L^2(\Gamma))$, and for all nonnegative integers $p$ and all
  $\omega \neq \pm \omega_0$ (cf. \Cref{Aop_def}) we have
  \begin{equation}\label{recursion_hypothesis_squared_general}
    \begin{split}
      \left\| \partial_\omega^p \mu(\cdot, \omega)\right\|_{L^2(\Gamma)}^2
      \le \sum_{i=0}^{p} &\left(\sum_{j=0}^{(i+1)(q+1)-1} d_{ij}^p \omega^{2j}
      \left\|\partial_\omega^{p-i} \widetilde{R}_T(\cdot,
      |\omega|)\right\|^2_{L^2(\Gamma)}\right),
    \end{split}
  \end{equation}
  where $d_{ij}^p > 0$ denote certain non-negative constants
  independent of $T$. Additionally, for each nonnegative integer $p$
  there exists a constant $C$ (dependent on $p$, $\alpha$, $T$, and on certain norms of $\widetilde{b}$),
  such that
  \begin{equation}\label{psit_nongrowth}
    \left\|\partial_\omega^p\mu(\cdot, \omega)\right\|_{L^2(\Gamma)} \le C |\omega|^{(p+1)(q+1)}
  \end{equation}
  for all sufficiently large values of $|\omega|$.
\end{lemma}
\begin{proof}
  Let $p$ denote a nonnegative integer and let
  $\omega \ne \pm \omega_0$.  \Cref{int_eq_freq_regularity} tells us
  that
  $\widetilde{R}_T \in C^\infty(\mathbb{R}\setminus \pm \omega_0;
  L^2(\Gamma))$ and
  $\widetilde{\psi}_{+,T}^f \in C^\infty(\mathbb{R}; L^2(\Gamma))$, as
  claimed. We restrict the remainder of the proof to the case
  $\omega > 0$; the full result then follows by the property of
  Hermitian symmetry satisfied by $\mu = \widetilde{\psi}^f_{+,T}$
  (\Cref{negative_freq}).

  In order to establish \Cref{recursion_hypothesis_squared_general} we
  first show that
  \begin{equation}\label{recursion_hypothesis_general}
      \left\| \partial_\omega^p \mu(\cdot, \omega)\right\|_{L^2(\Gamma)}
      \le \sum_{i=0}^p \left(\sum_{j=0}^{(i+1)(q+1)-1} b_{ij}^p \omega^j
      \left\|\partial_\omega^{p-i} \widetilde{R}_T(\cdot,
      \omega)\right\|_{L^2(\Gamma)}\right),
  \end{equation}
  for certain constants $b_{ij}^p \geq 0$.  The proof of
  \Cref{recursion_hypothesis_general} proceeds inductively: assuming
  that there exist constants $b_{ij}^s \geq 0$ ($0\leq s\leq p$) such
  that the relation
  \begin{equation}\label{recursion_hypothesis_general_s}
      \left\| \partial_\omega^s \mu(\cdot, \omega)\right\|_{L^2(\Gamma)}
      \le \sum_{i=0}^s \left(\sum_{j=0}^{(i+1)(q+1)-1} b_{ij}^s \omega^j
      \left\|\partial_\omega^{s-i} \widetilde{R}_T(\cdot,
      \omega)\right\|_{L^2(\Gamma)}\right)
\end{equation}
holds for all nonnegative integers $s \le p$, we show that there exist
constants $b_{ij}^{p+1} \geq 0$ such that
\Cref{recursion_hypothesis_general_s} holds for $s = p + 1$. (The base
case $s = 0$ follows from \Cref{CFIE_proof_generic} on account of the
$q$-growth condition.)  To carry out the inductive step we
use~\eqref{mu_freq_deriv} in \Cref{int_eq_freq_regularity}, which
tells us that
  \begin{equation}\label{leibniz}
    \left(\partial_\omega^{p+1} \mu \right)(\mathbf{r}, \omega) =
    A_\omega^{-1}\left(\partial_\omega^{p+1} \widetilde{R}_T(\mathbf{r}, \omega) -
    \sum_{k=1}^{p+1} a_k^{p+1}(\partial_\omega^k A_\omega)
    (\partial_\omega^{p+1-k} \mu)(\mathbf{r}, \omega) \right)
  \end{equation}
  for certain integers $a_k^{p+1}$; $k=1, \ldots, p+1$.  But the
  $q$-growth condition tells us that there exist
  positive $C_1, C_2$ such that
  $\left\|A_\omega^{-1}\right\|_{L^2(\Gamma)\to L^2(\Gamma)}\le C_1 +
  C_2 \omega^q$. Further, the operator-norm bound
  \Cref{Aop_deriv_bound} in \Cref{omega_explicit_norms_deriv_int_op}
  tells us that, for certain constants $\alpha_{0k}$ and
  $\alpha_{1k}$, we have
  $\left\|\partial_\omega^k A_\omega\right\|_{L^2(\Gamma)\to
    L^2(\Gamma)} \le \alpha_{0k} + \alpha_{1k} \omega$ for all
  $\omega \in \mathbb{R}^+ \setminus \omega_0$. It thus follows
  from~\eqref{leibniz} that
  \begin{equation}\label{partial_p_mu_deriv_first_bound}
    \begin{split}
      \left\|\vphantom{\partial_\omega^{p+1} \mu(\cdot, \omega)}
      \partial_\omega^{p+1} \mu(\cdot, \omega)\right.&
      \left.\vphantom{\partial_\omega^{p+1} \mu(\cdot, \omega)}\right\|_{L^2(\Gamma)}
      \le
      (C_1 + C_2\omega^q) \left(\vphantom{\sum_{k=1}^{p+1}}
      \left\|\partial_\omega^{p+1}
      \widetilde{R}_T(\cdot, \omega)\right\|_{L^2(\Gamma)} \right.\\
      &+ \left.\sum_{k=1}^{p+1} |a_k^{p+1}|
      (\alpha_{0k} + \alpha_{1k} \omega)\left\|\partial_\omega^{p+1-k}
      \mu(\cdot, \omega)\right\|_{L^2(\Gamma)}\right).
    \end{split}
  \end{equation}
  Substituting \Cref{recursion_hypothesis_general_s} with
  $s = p + 1 - k $ for $k = 1, \ldots, p + 1$ into
  \Cref{partial_p_mu_deriv_first_bound} we obtain
  \begin{alignat*}{2}
    \left\|\partial_\omega^{p+1}\mu(\cdot, \omega)\right\|_{L^2(\Gamma)} &\le
    (C_1 + C_2\omega^q)\left\|\partial_\omega^{p+1} \widetilde{R}_T(\cdot,
    \omega)\right\|_{L^2(\Gamma)}\\
    &\begin{alignedat}[t]{2}
      + \sum_{k=1}^{p+1} &|a_k^{p+1}|\left(C_1\alpha_{0k} +
      C_2\alpha_{0k}\omega^q + C_1\alpha_{1k}\omega +
      C_2\alpha_{1k}\omega^{q+1}\right)\times\\
      &\begin{alignedat}[t]{2}
        \negthickspace\negthickspace\times\left[\vphantom{\sum_{i=0}^{p+1-k}\sum_{j=0}^i}\right.&\left.
        \sum_{i=0}^{p+1-k}\sum_{j=0}^{(i+1)(q+1)-1} b_{ij}^{p+1-k} \omega^j
        \left\|\partial_\omega^{p+1-k-i} \widetilde{R}_T(\cdot,
    \omega)\right\|_{L^2(\Gamma)}\right],
    \end{alignedat}
    \end{alignedat}
  \end{alignat*}
  from which, expanding the products, we obtain
  \begin{alignat}{2}\label{final_inequality}
    \left\|\partial_\omega^{p+1}\mu(\cdot, \omega)\right\|_{L^2(\Gamma)} &
    \le (C_1 + C_2\omega^q)\left\|\partial_\omega^{p+1} \widetilde{R}_T(\cdot,
    \omega)\right\|_{L^2(\Gamma)} + (A),
  \end{alignat}
where
\begin{alignat*}{2}
  (A) = \sum_{k=1}^{p+1}&\sum_{i=0}^{p+1-k}\sum_{j=0}^{(i+1)(q+1)-1}
      |a_k^{p+1}b_{ij}^{p+1-k}|\left(C_1\alpha_{0k}\omega^j +
      C_2\alpha_{0k}\omega^{q+j}\vphantom{+ C_1\alpha_{1k}\omega^{j+1} +
      C_2\alpha_{1k}\omega^{q+j+1}}\right.\\
      &\left.+\;C_1\alpha_{1k}\omega^{j+1} +
      C_2\alpha_{1k}\omega^{q+j+1}\right)\left\|\partial_\omega^{p+1-k-i}
      \widetilde{R}_T(\cdot, \omega)\right\|_{L^2(\Gamma)}.
\end{alignat*}
It is easy to check that \Cref{final_inequality} implies that there
exist constants $b_{ij}^{p+1} \geq 0$ such that the relation
\Cref{recursion_hypothesis_general_s} with $s=p+1$ holds. Indeed, the
first term on the right-hand side of \Cref{final_inequality} and every
term arising from the summations in $(A)$ can be expressed as a
constant multiplied by a term of the form
$\omega^\ell \left\|\partial_\omega^m \widetilde{R}_T(\cdot,
  \omega)\right\|_{L^2(\Gamma)}$ for some $m \le p$ and for some
$\ell\leq (p+1)(q+1) + q$---all of which match corresponding terms in
\Cref{recursion_hypothesis_general_s} with $s = p + 1$. This
concludes the inductive proof, showing that for each nonnegative
integer $p$ there exist constants $b_{ij}^p\geq 0$ such that the
inequality~\eqref{recursion_hypothesis_general} holds. The desired
inequality~\Cref{recursion_hypothesis_squared_general} follows
directly from \Cref{recursion_hypothesis_general} using the relation
$\left\|\sum_{i=1}^p a_i\right\|^2 \le p\sum_{i=1}^p
\left\|a_i\right\|^2$.

In view of~\eqref{recursion_hypothesis_general}, in order to
establish the bound \Cref{psit_nongrowth} we estimate the expression
$\|\partial^{\ell}_\omega \widetilde{R}_T(\cdot,
\omega)\|_{L^2(\Gamma)}$ in~\eqref{recursion_hypothesis_general} (with
$\ell = p - i$, $0\leq i,\ell\leq p$). To do this we note that, in
view of~\eqref{RT_def} and~\eqref{Aomega_def_eqn}, for
$\omega \ge \omega_0$ we have
$\widetilde{R}_T = \left(\gamma^- \partial_\mathbf{n} -
  \i\omega\gamma^- \right) \widetilde{H}^f_T$. Thus,
using~\eqref{Ht_untilded_def} and~\eqref{jump_cond}), we obtain the
relation
\begin{equation}\label{Rtderiv_def}
    \partial^{\ell}_\omega \widetilde{R}_T = \partial^{\ell}_\omega \gamma^- \partial_\mathbf{n} \widetilde{B}^f - \partial^{\ell}_\omega \left( \i\omega \gamma^- \widetilde{B}^f\right) - \partial^{\ell}_\omega \left(\frac{1}{2}I + K_\omega^* -
    \i\omega S_\omega\right) \widetilde{\psi}_{-,T},\, \, (\omega > \omega_0),
\end{equation}
whose right-hand terms we estimate in what follows.  In view
of~\eqref{freq_u_b_tilde}, the function $\gamma^-\widetilde{B}^f$
equals the temporal Fourier transform of the compactly-supported
function $\gamma^- \widetilde{b} = \gamma^+ \widetilde{b}$. But, in
view of the $s$-regularity hypotheses and other assumptions in force,
the function
$\gamma^- \widetilde{b} = \gamma^- \widetilde{b}(\mathbf{r}, t)$ is an
element of $L^2(\mathbb{R}; L^2(\Gamma))$) that is compactly-supported
as a function of $t$. We may thus differentiate under the
Fourier-transform integral sign in~\eqref{freq_u_b_tilde}, which shows
that $\partial^{\ell}_\omega \gamma^- \widetilde{B}^f$ equals the
temporal Fourier transform of the compactly-supported function
$(-\i t)^\ell\gamma^- \widetilde{b}\in L^2(\mathbb{R};
L^2(\Gamma))$. Using the Cauchy-Schwarz inequality it follows that for
a certain constant $\widetilde{C}_0>0$ (that depends on $p$, $\alpha$,
$T$, and $\|\gamma^- \widetilde{b}\|_{L^2(\mathbb{R}; L^2(\Gamma))}$,
but which does not depend on $\omega$) we have
$\|\partial^{\ell}_\omega \gamma^- \widetilde{B}^f\|_{L^2(\Gamma)}\leq
\widetilde{C}_0$ for $0\leq \ell\leq p$ and for all
$\omega \in \mathbb{R}$.  In view of the triangle inequality we
conclude that, for each $0 \le \ell \le p$ and for all
$\omega \in \mathbb{R}$ we have
  \begin{equation}
    \label{paley_der_b}
    \left\|\partial^{\ell}_\omega \left(\omega \gamma^-\widetilde{B}^f(\cdot,
      \omega)\right)\right\|_{L^2(\Gamma)} \le \widetilde{C}_1(1 + \omega)
  \end{equation}
  where $ \widetilde{C}_1$ denotes a constant which, once again,
  depends on $p$, $\alpha$, $T$, and on
  $\|\gamma^- \widetilde{b}\|_{L^2(\mathbb{R};
    L^2(\Gamma))}$ but does not depend on $\omega$, which provides the
  needed estimate of the second right-hand term. A similar argument
  applied to
  $\partial_\omega^\ell \gamma^- \partial_\mathbf{n} \widetilde{B}^f$
  leads to the desired estimate
  \begin{equation}
    \label{paley_der_c}
    \left\|\partial^{\ell}_\omega \gamma^- \partial_\mathbf{n} \widetilde{B}^f(\cdot,
      \omega)\right\|_{L^2(\Gamma)} \le  \widetilde{C}_2,
  \end{equation}
  for the first right-hand term in \Cref{Rtderiv_def}, where
  $ \widetilde{C}_2$ denotes a constant with dependencies analogous to
  those found for $ \widetilde{C}_1$.

  In order to obtain a bound for the last term on the right-hand
  side of~\Cref{Rtderiv_def}, in turn, we note that 
  $\widetilde{\psi}_{-,T}$ is compactly-supported
  (cf.~\Cref{psipm_def}) and satisfies
  $\widetilde{\psi}_{-,T} \in L^2(\mathbb{R}; L^2(\Gamma))$
  (cf. \Cref{3d_decay_lemma_2ndkind_wellposed}), and, thus, an argument similar to the one leading
  to~\eqref{paley_der_b} yields the bound
  \begin{equation}\label{eq:dpi}
    \left\|\partial_\omega^{\ell} \widetilde{\psi}^f_{-,T}(\cdot, \omega)\right\|_{L^2(\Gamma)} \le  \widetilde{C}_3
  \end{equation}
  for all $\ell$, $0 \le \ell \le p$, and all $\omega > 0$, where
  $ \widetilde{C}_3$ depends on the integer $p$, $\alpha$, $T$, and,
  via \Cref{3d_decay_lemma_2ndkind_wellposed} with $p=0$, on
  $\|\gamma^+ \widetilde{b}\|_{H^{q+1}(\mathbb{R}; L^2(\Gamma))}$ and
  $\|\gamma^+ \partial_\mathbf{n} \widetilde{b}\|_{H^q(\mathbb{R};
    L^2(\Gamma))}$. Then, employing the operator norm bounds presented
  in \Cref{omega_explicit_norms_deriv_SK} for
  $\partial_\omega^\ell S_\omega$ and
  $\partial_\omega^\ell K_\omega^*$ ($\ell = 0, \ldots, p$) together
  with \Cref{eq:dpi} we obtain the bound
\begin{equation}\label{eq:dpii}
  \left\|\partial^{\ell}_\omega \left(\left(\frac{1}{2}I + K_\omega^* -
    \i\omega S_\omega\right) \widetilde{\psi}_{-,T}(\cdot, \omega)\right)\right\|_{L^2(\Gamma)} \le  \widetilde{C}_4(1 + \omega)
\end{equation}
for all integers $\ell$, $0 \le \ell \le p$ and for all $\omega > 0$,
where $\widetilde{C}_4$ denotes a constant dependent on
$\widetilde{C}_3$ and the constants in
\Cref{omega_explicit_norms_deriv_SK} but which is independent of
$\omega$.

Estimating the norm of~\eqref{Rtderiv_def} by means of the triangle
inequality, and bounding the right-hand side of the resulting
inequality by means of \Cref{paley_der_b}, \Cref{paley_der_c} and
\Cref{eq:dpii}, we obtain
\begin{equation}\label{eq:dpiii}
  \left\|\partial_\omega^\ell \widetilde{R}_T(\cdot, \omega)\right\|_{L^2(\Gamma)} \le  \widetilde{C}_5\omega \quad (\omega > \omega_0),
\end{equation}
for all integers $\ell$, $0 \le \ell \le p$, where $\widetilde{C}_5$ denotes a constant
independent of $\omega$ but dependent on $\omega_0$, $p$, $\alpha$ and
$T$, and on the norm values
$\|\gamma^+ \widetilde{b}\|_{H^{q+1}(\mathbb{R}; L^2(\Gamma))}$ and
$\|\gamma^+ \partial_\mathbf{n} \widetilde{b}\|_{H^q(\mathbb{R};
  L^2(\Gamma))}$. The estimate \Cref{psit_nongrowth} follows directly
from \Cref{recursion_hypothesis_general} and \Cref{eq:dpiii}.
\end{proof}

\begin{rem}\label{rem:breve}
  In order to obtain the estimates in \Cref{3d_decay_thm_ii} it will
  be necessary to perform time-recentering on the data $\widetilde{b}$
  and the solution $\widetilde{\psi}$---see Remarks \ref{ITvsI0}
  and~\ref{rem:tilde_1}.  Given a real number $T_0$ we define for a
  given $\widetilde{b}$ the time-shifted ``breve'' quantities
  \begin{equation}\label{def:breve}
    \breve{b}(\bfr, t) = \widetilde{b}(\bfr, t + T_0),\quad\mbox{and}\quad \breve{\psi}(\bfr, t) = \widetilde{\psi}(\bfr, t + T_0).
  \end{equation}
  With reference to \Cref{rem:tilde_1}, note that
  $\breve{\psi}(\bfr, t)$ equals the solution
  $\widetilde{\psi}(\bfr, t)$ of \Cref{eq:tdie_sl_generic} with
  $\widetilde{b}$ substituted by $\breve{b}$. Note also that if
  $\widetilde{b}$ satisfies \Cref{all_t_b_generic} for some $\alpha$
  then $\breve{b}$ satisfies \Cref{all_t_b_generic} with $\alpha$
  substituted by $\alpha - T_0$.  We will consider the $T_0$-dependent
  density $\breve{\psi}$ in conjunction with the interval $I_0$ equal
  to $I_T$ with $T = 0$ (see~\Cref{ITvsI0}), and we thus define
  $\breve{\psi}_{\pm,0}$ and $\breve{\psi}_{*,0}$ by analogy with
  $\psi_{\pm,T}$ and $\psi_{*,T}$ in \Cref{timewinddens}, but with
  $\breve{\psi}$ in lieu of $\psi$ and with $T=0$.
\end{rem}

  Consistent with the conventions laid out in
  \Cref{timewinddens} and \Cref{rem:tilde_1} and letting
    \begin{equation}\label{breveupmstardef}
        \breve{u}_{\pm,0}(\bfr, t) = (\mathscr{S} \breve{\psi}_{\pm,0})(\bfr, t)\quad\mbox{and}\quad \breve{u}_{*,0}(\bfr, t) = (\mathscr{S} \breve{\psi}_{*,0})(\bfr, t),\quad (\bfr, t) \in \mathbb{R}^3 \times \mathbb{R},
    \end{equation}
    we define also the function $\breve{h}_0$ (cf. \Cref{htildedef}),
    \begin{equation}\label{brevehdef}
        \breve{h}_0(\bfr, t) = \breve{b}(\bfr, t) - \breve{u}_{-,0}(\bfr, t),\quad (\mathbf{r}, t) \in \Omega^{inc} \times \mathbb{R},
    \end{equation}
    and, using~\eqref{potent_time_freq}, its Fourier transform
    \begin{equation}\label{Ht_def_breve}
        \breve{H}^f_0(\mathbf{r}, \omega) = \breve{B}^f(\mathbf{r}, \omega) - \left(\mathscr{S}_\omega \breve{\psi}_{-,0}^f\right)(\mathbf{r}, \omega),\quad \mathbf{r}\in\Omega^{inc} \times \mathbb{R}.
    \end{equation}
    Similarly, the right-hand side of
    \Cref{CFIE_proof_generic} becomes
    \begin{equation}\label{Rt_def_breve}
      \breve{R}_0(\bfr, \omega) = \gamma^- \partial_\mathbf{n} \breve{H}^f_0(\bfr, \omega) - \i\eta_0(\omega)\gamma^- \breve{H}^f_0(\bfr, \omega), \quad (\bfr,\omega) \in \Gamma\times \mathbb{R}.
    \end{equation}
\begin{lemma}\label{Rderiv_sobolev_bound}
  Let $T_0$ and $\tau$ denote given real numbers, let $q$ denote a
  non-negative integer, let $\breve{b}$ be defined as in
  \Cref{rem:breve}, and assume that $\breve{b}(\bfr, t)$ vanishes for
  all
  $(\bfr, t) \in \widebar\Omega\times \left\lbrace I_0\cup \left[0,
      \infty\right) \right\rbrace$ (where $I_0 = I_T$ with $T=0$, see
  \Cref{domainofdep}). Additionally, let $\Omega$ satisfy the
  $q$-growth condition and assume $\breve{b}$ satisfies the
  $s$-regularity conditions \Cref{eq:gamma_Hs_assump} with $s = q$.
  Finally, assume that for a given nonnegative integer $n$,
  $\breve{\psi}_{*,0}$ satisfies
  $\breve{\psi}_{*,0} \in H^{n+1}(I_0;L^2(\Gamma))$.  Then for all
  integers $m \ge 0$ and all integers $j$ such that $0 \le j \le n$ we
  have
  \[
    \int_0^\infty \omega^{2j} \left\| \partial_\omega^m \breve{R}_0(\cdot,
    \omega)\right\|_{L^2(\Gamma)}^2\,\d\omega \le C \left\| \breve{\psi}_{*,0}
    \right\|_{H^{j+1}(I_0;\,L^2(\Gamma))}^2,
  \]
  where $C$ is a constant independent of $T_0$ and $b$.
\end{lemma}
\begin{rem}\label{breve{h}_0_impact_rem}
  In order to obtain an upper bound on integrals containing powers of
  the temporal variable $t$, the estimate~\eqref{breve{h}_0_impact} in
  the proof of \Cref{Rderiv_sobolev_bound} relies in an essential
  manner on the domain-of-dependence relations~\eqref{h_equiv}
  and~\eqref{h_compactsupp} for the function $\breve{h}_0$ (cf.
  \Cref{rem:h_equiv_rem_2})---which, limiting the integration of the
  aforementioned powers of $t$ to a bounded interval, yields
  meaningful integral estimates necessary for the proof of the lemma.
\end{rem}
\begin{proof}[Proof of \Cref{Rderiv_sobolev_bound}.]
  Since $\breve{b}$ satisfies the $s$-regularity conditions with
  $s = q$, using \Cref{3d_decay_lemma_2ndkind_wellposed}, with
  reference to \Cref{rem:breve}, shows that the quantity
  $\breve{\psi}_{-,0}^f$ in \eqref{Ht_def_breve} (which also enters in
  the definition~\eqref{Rt_def_breve}) satisfies
  $\breve{\psi}_{-,0}^f \in L^2(\mathbb{R}; L^2(\Gamma))$.  In view
  of~\eqref{Rt_def_breve} we have
  \begin{equation}\label{w2j_proof_first_eqn}
      \int_0^\infty \omega^{2j} \left\| \partial_\omega^m \breve{R}_0(\cdot,
      \omega)\right\|_{L^2(\Gamma)}^2\,\d\omega = \int_0^\infty \int_\Gamma \left|\widehat{S}_{jm}
    \breve{H}^f_0(\mathbf{r}, \omega)\right|^2\,\d\sigma(\mathbf{r})\,\d\omega,
  \end{equation}
  where the operator $\widehat{S}_{jm}$ is defined as
  \begin{equation}\label{Sdef_def}
    \widehat{S}_{jm} = \omega^{j}
    \partial_\omega^m\left(\gamma^-\partial_{\mathbf{n}} - \i \eta_0 \gamma^-\right).
  \end{equation}
  In view of the definition~\Cref{Aomega_def_eqn} of the function
  $\eta_0(\omega)$ in~\eqref{Sdef_def}, which depends on whether
  $0\leq\omega< \omega_0$ or
  $\omega>\omega_0$,~\cref{w2j_proof_first_eqn} is re-expressed in the
  form
  \begin{equation}\label{Rderiv_split}
    \begin{split}
        \int_0^\infty \omega^{2j} &\left\| \partial_\omega^m \breve{R}_0(\cdot,
    \omega)\right\|_{L^2(\Gamma)}^2\,\d\omega\\
    &= \int_\Gamma\left(\int_0^{\omega_0} + \int_{\omega_0}^\infty\right)\left|\widehat{S}_{jm}
    \breve{H}^f_0(\mathbf{r}, \omega)\right|^2\,\d\omega\,\d\sigma(\mathbf{r}).
    \end{split}
  \end{equation}
  Then, defining the operators
  \begin{equation}\label{S_split_omega}
    \widehat{S}_{jm}^1 = \omega^j \partial_\omega^{m}
    (\gamma^-\partial_{\mathbf{n}} - \i\gamma^-),\quad\mbox{and}\quad
    \widehat{S}_{jm}^2 = \omega^j \partial_\omega^{m}
    (\gamma^-\partial_{\mathbf{n}} - \i\omega\gamma^-),
  \end{equation}
  which clearly satisfy $\widehat{S}_{jm} = \widehat{S}_{jm}^1$ for
  $\omega < \omega_0$ and $\widehat{S}_{jm} = \widehat{S}_{jm}^2$ for
  $\omega > \omega_0$,~\eqref{Rderiv_split} yields
  \begin{align*}
    \int_0^\infty \omega^{2j} &\left\| \partial_\omega^m \breve{R}_0(\cdot,
                                \omega)\right\|_{L^2(\Gamma)}^2\,\d\omega \\
    &= \int_\Gamma
    \int_{0}^{\omega_0} \left|\widehat{S}_{jm}^1 \breve{H}^f_0(\mathbf{r},
    \omega)\right|^2\,\d\omega\,\d\sigma(\mathbf{r})
    +\int_\Gamma\int_{\omega_0}^\infty \left|\widehat{S}_{jm}^2 \breve{H}^f_0(\mathbf{r},
    \omega)\right|^2\,\d\omega\,\d\sigma(\mathbf{r})\\
    &\le \int_\Gamma
    \int_{-\infty}^\infty \left|\widehat{S}_{jm}^1 \breve{H}^f_0(\mathbf{r},
    \omega)\right|^2\,\d\omega\,\d\sigma(\mathbf{r})
    +\int_\Gamma\int_{-\infty}^\infty \left|\widehat{S}_{jm}^2 \breve{H}^f_0(\mathbf{r},
    \omega)\right|^2\,\d\omega\,\d\sigma(\mathbf{r}).
  \end{align*}
  Thus, utilizing the time-domain operators
  \begin{equation}\label{S_split_t}
    S_{jm}^1 = (-\i\frac{\partial}{\partial t})^j (-it)^{m}
    (\gamma^-\partial_{\mathbf{n}} - \i \gamma^-) \quad\mbox{and}\quad
    S_{jm}^2 = (-\i\frac{\partial}{\partial t})^j (-it)^{m}
    (\gamma^-\partial_{\mathbf{n}} - \frac{\partial}{\partial t}\gamma^-)
  \end{equation}
  corresponding to~\eqref{S_split_omega}, together with Plancherel's
  theorem and equations~\eqref{brevehdef}--\eqref{Rt_def_breve}, we
  obtain
  \begin{equation} \label{breve{h}_0_impact}
 \begin{split}
      &\int_0^\infty \omega^{2j} \left\| \partial_\omega^m \breve{R}_0(\cdot,
                                  \omega)\right\|_{L^2(\Gamma)}^2\,\d\omega \\
    &\le \int_\Gamma
    \int_{-\infty}^\infty \left|S_{jm}^1 \breve{h}_0(\mathbf{r},
    t')\right|^2\,\d t'\,\d\sigma(\mathbf{r})
    +\int_\Gamma\int_{-\infty}^\infty \left|S_{jm}^2 \breve{h}_0(\mathbf{r},
    t')\right|^2\,\d t'\,\d\sigma(\mathbf{r})\\
    &=\int_\Gamma
        \int_{-\tau}^{T_*} \left|S_{jm}^1 \breve{u}_{*,0}(\mathbf{r},
    t')\right|^2\,\d t'\,\d\sigma(\mathbf{r})
        +\int_\Gamma\int_{-\tau}^{T_*} \left|S_{jm}^2 \breve{u}_{*,0}(\mathbf{r},
    t')\right|^2\,\d t'\,\d\sigma(\mathbf{r}),
    \end{split}
    \end{equation}
    where, since $\Gamma = \partial\Omega \subset \widebar\Omega$, the
    last equality follows by using~\Cref{h_equiv}
    and~\Cref{h_compactsupp} with $T=0$ (see \Cref{rem:breve}).
    Letting
  \begin{equation}\label{v1v2_def}
  v_1 = (\gamma^-\partial_{\mathbf{n}} -
    \i\gamma^-)\breve{u}_{*,0}, \quad\mbox{and}\quad
  v_2 = (\gamma^-\partial_{\mathbf{n}} -
  \frac{\partial}{\partial t}\gamma^-)\breve{u}_{*,0},
  \end{equation}
  and calling $a_\ell = {j \choose \ell}$, by Leibniz's product rule we
  then obtain
  \begin{alignat*}{2}
    \int_0^\infty \omega^{2j} &\left\| \partial_\omega^m \breve{R}_0(\cdot,
    \omega)\right\|_{L^2(\Gamma)}^2\,\d\omega\\
    &\begin{alignedat}{2}\le \int_\Gamma&
      \int_{-\tau}^{T_*} \left| (-\i \frac{\partial}{\partial t'})^j (-\i t')^m
    v_1(\mathbf{r}, t')\right|^2\,\d t'\,\d\sigma(\mathbf{r})\\
      &+ \int_\Gamma \int_{-\tau}^{T_*} \left| (-\i \frac{\partial}{\partial
      t'})^j (-\i t')^m
    v_2(\mathbf{r}, t')\right|^2\,\d t'\,\d\sigma(\mathbf{r})
    \end{alignedat}\\
    &\begin{alignedat}{2}=
      \int_\Gamma&
      \int_{-\tau}^{T_*} \left|\sum_{\ell=0}^j a_\ell \left(
      \left(\frac{\partial}{\partial t'}\right)^{\ell} (t')^m\right)
      \left(\left(\frac{\partial}{\partial t'}\right)^{j-\ell}
      v_1(\mathbf{r}, t')\right)\right|^2\,\d t'\,\d\sigma(\mathbf{r})\\
      &+ \int_\Gamma
      \int_{-\tau}^{T_*} \left|\sum_{\ell=0}^j a_\ell \left(
      \left(\frac{\partial}{\partial t'}\right)^\ell (t')^m\right)
      \left(\left(\frac{\partial}{\partial t'}\right)^{j-\ell}
      v_2(\mathbf{r}, t')\right)\right|^2\,\d t'\,\d\sigma(\mathbf{r}).
    \end{alignedat}
  \end{alignat*}
  Substituting the exact value of the derivative
  $\left(\frac{\partial}{\partial t'}\right)^\ell (\i t')^m$ in these
  expressions, we further obtain
  \begin{alignat*}{2}
      \int_0^\infty \omega^{2j} &\left\| \partial_\omega^m \breve{R}_0(\cdot,
    \omega)\right\|_{L^2(\Gamma)}^2\,\d\omega\\
    &\begin{alignedat}{2}\le
      \int_\Gamma&
      \int_{-\tau}^{T_*} \left|\sum_{\ell=0}^j \tilde{a}_\ell (t')^{m-\ell}
      \left(\left(\frac{\partial}{\partial t'}\right)^{j-\ell}
      v_1(\mathbf{r}, t')\right)\right|^2\,\d t'\,\d\sigma(\mathbf{r})\\
      &+ \int_\Gamma
      \int_{-\tau}^{T_*} \left|\sum_{\ell=0}^j \tilde{a}_\ell (t')^{m-\ell}
      \left(\left(\frac{\partial}{\partial t'}\right)^{j-\ell}
      v_2(\mathbf{r}, t')\right)\right|^2\,\d t'\,\d\sigma(\mathbf{r}),
    \end{alignedat}
  \end{alignat*}
  where $\tilde{a}_\ell = \frac{m!}{(m-\ell)!} a_\ell$ for
  $m - \ell \ge 0$ and $\tilde{a}_\ell = 0$ for $m - \ell < 0$.  Since
  the $t'$-integration is limited to the bounded region
  $[- \tau, T_*]$ the quantities $|t'|^{m-\ell}$ are bounded by a
  constant (which, importantly, is independent of $T_0$---see \Cref{ITvsI0}),
  and thus
  \begin{alignat*}{2}
      \int_0^\infty \omega^{2j} &\left\| \partial_\omega^m \breve{R}_0(\cdot,
    \omega)\right\|_{L^2(\Gamma)}^2\,\d\omega\\
    &\le C_1\sum_{\ell=0}^j \int_\Gamma \int_{-\tau}^{T_*}
    \left(\left|\left(\frac{\partial}{\partial t'}\right)^\ell v_1(\mathbf{r},
      t')\right|^2 +\left|\left(\frac{\partial}{\partial t'}\right)^\ell v_2(\mathbf{r},
    t')\right|^2 \right)\,\d t'\,\d\sigma(\mathbf{r})\\
    &\le C_1\sum_{\ell=0}^j \int_\Gamma \int_{-\infty}^{\infty}
    \left(\left|\left(\frac{\partial}{\partial t'}\right)^\ell v_1(\mathbf{r},
      t')\right|^2 +\left|\left(\frac{\partial}{\partial t'}\right)^\ell v_2(\mathbf{r},
    t')\right|^2 \right)\,\d t'\,\d\sigma(\mathbf{r}),
  \end{alignat*}
  where the last inequality simply bounds the space-time norm on the
  finite temporal region $[- \tau, T_*]$ by the full time integral on
  $\mathbb{R}$. In view of~\Cref{v1v2_def}, using once again the Plancherel
  theorem, and denoting, per \Cref{FT_conv},
  $\breve{U}_{*,0}^f$ the Fourier transform of $\breve{u}_{*,0}$, we
  estimate
  \begin{alignat*}{2}
      \int_0^\infty \omega^{2j} &\left\| \partial_\omega^m \breve{R}_0(\cdot,
    \omega)\right\|_{L^2(\Gamma)}^2\,\d\omega\\
    &\begin{alignedat}{2}%
      \le C_1\sum_{\ell=0}^j &\left(\int_\Gamma \int_{-\infty}^{\infty}
      \left|\omega^\ell(\gamma^-\partial_\mathbf{n} -
        \i\gamma^-)\breve{U}_{*,0}^f(\mathbf{r},
    \omega)\right|^2\,\d\omega\,\d\sigma(\mathbf{r})\right.\\
      &\left.+ \int_\Gamma \int_{-\infty}^{\infty}
      \left|\omega^\ell(\gamma^-\partial_\mathbf{n} - \i\omega\gamma^-)\breve{U}_{*,0}^f(\mathbf{r},
    \omega)\right|^2\,\d\omega\,\d\sigma(\mathbf{r})\right)
    \end{alignedat}\\
    &\begin{alignedat}{2}%
      \le \tilde{C}_1&\int_\Gamma \int_{-\infty}^{\infty}
      \left|(1 + \omega^2)^{j/2}(\gamma^-\partial_\mathbf{n} -
        \i\gamma^-)\breve{U}_{*,0}^f(\mathbf{r},
    \omega)\right|^2\,\d\omega\,\d\sigma(\mathbf{r})\\
      &+ \tilde{C}_1\int_\Gamma \int_{-\infty}^{\infty}
      \left|(1 + \omega^2)^{j/2}(\gamma^-\partial_\mathbf{n} -
        \i\omega\gamma^-)\breve{U}_{*,0}^f(\mathbf{r},
    \omega)\right|^2\,\d\omega\,\d\sigma(\mathbf{r}).
    \end{alignedat}\\
  \end{alignat*}

  We thus have established that
  \begin{equation}\label{Rderiv_sobolev_norm_estimate}
  \begin{split}
      \int_0^\infty \omega^{2j} &\left\| \partial_\omega^m \breve{R}_0(\cdot,
    \omega)\right\|_{L^2(\Gamma)}^2\,\d\omega\\
    &\le \tilde{C}_1\int_{-\infty}^\infty (1 + \omega^2)^j \left\| (\gamma^- \partial_\mathbf{n} -
      \i\gamma^-) \breve{U}_{*,0}^f(\cdot, \omega)\right\|_{L^2(\Gamma)}^2\,\d\omega\\
    &+ \tilde{C}_1\int_{-\infty}^\infty (1 + \omega^2)^j \left\| (\gamma^-\partial_\mathbf{n} -
      \i\omega\gamma^-) \breve{U}_{*,0}^f(\cdot, \omega)\right\|_{L^2(\Gamma)}^2\,\d\omega.
  \end{split}
  \end{equation}
  To complete the proof we now use \Cref{per_freq_rhs_oper_bounds},
  which provides the frequency-wise bounds
  \begin{equation}\label{Rderiv_normbound_1}
      \left\| (\gamma^-\partial_\mathbf{n} - \i\gamma^-) \breve{U}_{*,0}^f(\cdot, \omega)\right\|_{L^2(\Gamma)}
    \le D(1 + \omega^2)^{1/2}\left\|\breve{\psi}_{*,0}^f(\cdot,
    \omega)\right\|_{L^2(\Gamma)},
  \end{equation}
  and
  \begin{equation}\label{Rderiv_normbound_2}
    \left\| (\gamma^-\partial_\mathbf{n} - \i \omega\gamma^-) \breve{U}_{*,0}^f(\cdot, \omega)\right\|_{L^2(\Gamma)}
      \le E (1 + \omega^2)^{1/2}\left\|\breve{\psi}_{*,0}^f(\cdot, \omega)\right\|_{L^2(\Gamma)},
  \end{equation}
  where $D, E > 0$ are constants independent of $\omega$, $\breve{b}$
  and $\breve{\psi}$.  Substituting~\cref{Rderiv_normbound_1}
  and~\cref{Rderiv_normbound_2} in~\Cref{Rderiv_sobolev_norm_estimate}
  we conclude that
  \begin{alignat*}{2}
      \int_0^\infty \omega^{2j} &\left\| \partial_\omega^m \breve{R}_0(\cdot,
      \omega)\right\|_{L^2(\Gamma)}^2\,\d\omega\\
      &\begin{alignedat}{2}
        \le \int_{-\infty}^\infty C_2(1 + &\omega^2)^{j+1}
          \left\|\breve{\psi}_{*,0}^f(\cdot,
          \omega)\right\|_{L^2(\Gamma)}^2\,\d\omega\\
          &+ \int_{-\infty}^\infty C_3(1 + \omega^2)^{j+1}\left\|\breve{\psi}_{*,0}^f(\cdot,
        \omega)\right\|_{L^2(\Gamma)}^2\,\d\omega.
    \end{alignedat}\\
    &\le
    C\left\|\breve{\psi}_{*,0}\right\|_{H^{j+1}(\mathbb{R};\,L^2(\Gamma))}^2
    =
    C\left\|\breve{\psi}_{*,0}\right\|_{H^{j+1}(I_0;\,L^2(\Gamma))}^2,
    \end{alignat*}
    as desired.
\end{proof}

\Cref{Rderiv_sobolev_bound} is used in what follows to establish the
main building block in the proof of \Cref{3d_decay_thm_ii}, namely,
\Cref{decay_estimate_L2}. The proof of \Cref{decay_estimate_L2}, in
turn, incorporates a smooth windowing procedure which relies on use of
compactly-supported time-domain window functions. For definiteness, we
utilize the time-window functions introduced in what follows.
\begin{defi}\label{def_w}
  Letting $v(u) = \exp(\frac{2e^{-1/u}}{u-1})$, we define
\begin{equation}\label{wTdef}
  w(s) =
    \begin{cases}
      1 - v(\frac{s + 2s_0}{s_0}), & -2s_0 \le s \le -s_0\\
      1, & -s_0 < s < s_0\\
      v(\frac{s - s_0}{s_0}), & s_0 \le s \le 2s_0\\
      0, & |s| > 2s_0,
    \end{cases}
    \qquad\mbox{and}\qquad w_\varphi(s) = w(s - \varphi),
\end{equation}
where $\varphi\in\mathbb{R}$ denotes an important ``time-shift''
parameter that enacts the time-recentering approach alluded to in
\Cref{ITvsI0}, and where $s_0 > 0$ is a fixed parameter that, in the
context of this paper, can be selected arbitrarily.
\end{defi}

Clearly, the functions $w$ and $w_\varphi$ (i)~Satisfy
$w,w_\varphi\in C^\infty(\mathbb{R})$; (ii)~Equal $1$ in an interval
of length $2s_0$; (iii)~Increase (decrease) from $0$ to $1$ ($1$ to
$0$) in intervals of length $s_0$; (iv)~Satisfy $0\leq w(s) \le 1$ and
$0\leq w_\varphi(s) \le 1$ for all $s\in\mathbb{R}$.  It is easy to
check that for every $\varphi \in\mathbb{R}$ we have
\begin{equation}
  \label{eq:POU2}
  \begin{split}
   & w_{\varphi + s_0}(s)  + w_{\varphi + 4 s_0}(s) =1\\
   & w_{\varphi + s_0 + 3\ell s_0}(s) = 0\quad (\ell \not\in \{0,1 \})
  \end{split}
  \qquad\mbox{for} \quad s \in [\varphi, \varphi + 3s_0].
\end{equation}
In particular, the functions $w_{3\ell s_0}(s)$ with $\ell \in\mathbb{Z}$
form a partition of unity, wherein at most two functions in the family do not vanish at any given $s\in\mathbb{R}$.

\begin{lemma}\label{decay_estimate_L2}
  Let $n$ and $q$ denote non-negative integers, $n>0$, let $T_0>0$,
  let $\widetilde{b}$ be defined as in \Cref{rem:tilde_1},  and assume that
  $\widetilde{b}$ vanishes for all $(\bfr, t) \in \widebar{\Omega}\times\left \{I_{T_0} \cup [T_0, \infty)\right\}$, with $I_{T_0}$ as in
  \Cref{domainofdep}. Additionally, let $\Omega$ satisfy the
  $q$-growth condition and assume $\widetilde{b}$ satisfies
  the $s$-regularity conditions~\eqref{eq:gamma_Hs_assump} with $s = (n+1)(q+1) + q$.
Then, the functions $\breve{\psi}_{+,0}$ and $\breve{\psi}$ defined in~\Cref{rem:breve} satisfy
\begin{equation}\label{L2_rchidecay}
  \left\|w_\varphi \breve{\psi}_{+,0}\right\|_{L^2(\mathbb{R};\,L^2(\Gamma)}^2 \le
    C(\Gamma, \tau, n, s_0) \varphi^{-2n}
    \left\|\breve{\psi}\right\|_{H^{(n+1)(q+1)}(I_0;\,L^2(\Gamma))}^2,
\end{equation}
for all $\varphi > 0$, where $C(\Gamma, \tau, n, s_0)$ denotes a
constant independent of $\varphi, T_0,$ and $\widetilde{b}$.
\end{lemma}
\begin{proof}
  We first note that since by hypothesis $\widetilde{b}(\bfr, t)$
  vanishes for
  $(\bfr,t) \in \widebar\Omega \times \left\lbrace I_{T_0}\cup
    \left[T_0, \infty\right)\right\rbrace$, it follows that
  $\breve{b}(\bfr, t)$ vanishes for all
  $(\bfr,t) \in \widebar\Omega \times \left\lbrace I_0\cup \left[0,
      \infty\right)\right\rbrace$.  Similarly, the $s$-regularity
  condition hypotheses are satisfied for $\breve{b}$.

  In order to establish the desired decay
  estimate~\eqref{L2_rchidecay} for $w_\varphi \breve{\psi}_{+,0}$, by
  Plancherel's theorem we may instead provide an estimate for the
  $L^2$ norm of its Fourier transform:
  \begin{equation}\label{conv_w_2}
   \left\|w_\varphi \breve{\psi}_{+,0}\right\|_{L^2(\mathbb{R};\,L^2(\Gamma)}^2 =   \left\|\widehat{w_\varphi \breve{\psi}_{+,0}}\right\|_{L^2(\mathbb{R};\,L^2(\Gamma)}^2 =
    \left\|\widehat{w}_\varphi \ast
    \breve{\psi}_{+,0}^f\right\|_{L^2(\mathbb{R};\,L^2(\Gamma)}^2.
  \end{equation}
  Using the relation
  $\widehat{w}_\varphi(\omega) = \e^{-\i\omega \varphi} \widehat{w}(\omega)$ we obtain
  \begin{equation}\label{conv_w}
    \left(\widehat{w}_\varphi \ast
    \breve{\psi}_{+,0}^f\right)(\bfr, \omega)   =  \int_{-\infty}^\infty \e^{-\i \tau \varphi} \widehat{w}(\tau) \breve{\psi}_{+,0}^f(\bfr, \omega - \tau)\,\d \tau,
   \end{equation}
   and we proceed to integrate by parts this integral $n$ times with
   respect to $\tau$. To do this we note that, (i)~Both
   $\widehat{w}_\varphi(\omega)$ and
   $\breve{\psi}_{+,0}^f (\mathbf{r}, \omega)$ are infinitely
   differentiable functions of $\omega$ (in view of \Cref{def_w},
   \Cref{rem:breve} and \Cref{int_eq_freq_regularity}); (ii)~The
   Fourier transform $\widehat w(\tau)$ and all of its derivatives
   tends to zero faster than any negative power of $\tau$ as
   $\tau\to\pm\infty$, as it befits the Fourier transform of a
   compactly supported function; and (iii)~For each non-negative
   integer $m$ there exist an integer $N_m > 0$ and a constant
   $C_m >0$ such that
   $\left\|\partial_\omega^m \breve{\psi}_{+,0}^f(\cdot, \omega -
     \tau)\right\|_{L^2(\Gamma)} \le C_m|\omega - \tau|^{N_m}$ as
   $|\tau| \to \infty$, as it follows directly from
   \Cref{psit_nongrowth} in
   \Cref{3d_decay_thm_recursive_bound_general}. Thus, integrating by
   parts~\eqref{conv_w}, using Leibniz's product rule, and noting that
   all boundary terms at $\tau = \pm\infty$ vanish, \eqref{conv_w}
   becomes
   \begin{equation}\label{nIntByParts}
\begin{split}
  \left( \widehat{w}_\varphi \ast \breve{\psi}_{+,0}^f \right)(\mathbf{r}, \omega) =
    \left(\frac{1}{i\varphi}\right)^n \int_{-\infty}^\infty
    \e^{-i\tau \varphi}  \left(\vphantom{\sum_{m=0}^n} \right. & \sum_{m=0}^n a_m \left( \partial^{n-m}_\tau
    \widehat{w}(\tau)\right) \times\\
    &\left.\times \left( \vphantom{\sum_{m=0}^n} \partial^m_\omega \breve{\psi}_{+,0}^f(\mathbf{r}, \omega -
    \tau)\right)\right)d\tau,
\end{split}
\end{equation}
where $a_m = {n \choose m}$.

In view of~\eqref{conv_w_2} and \Cref{nIntByParts}, calling, for $0 \le m \le n$,
\begin{equation}\label{w_psi}
    F_{nm}^\varphi(\tau) = \e^{-\i \tau \varphi}\partial^{n-m}_\tau \widehat{w}(\tau) \quad\mbox{and}\quad P_m(\bfr,\tau) = \partial^m_\tau \breve{\psi}_{+,0}^f(\bfr, \tau),
\end{equation}
and using the relation $\| \sum_{i=1}^n a_i \|^2 \le n \sum_{i=1}^n \| a_i\|^2$,
we obtain
\begin{equation}
\begin{split}
  \left\| w_\varphi \breve {\psi}_{+,0} \right.&\negmedspace\left.\vphantom{w_\varphi
  \breve{\psi}_{+,0}}\right\|_{L^2(\mathbb{R};\,L^2(\Gamma))}^2
  \le (n+1)\varphi^{-2n} \sum_{m=0}^n |a_m|^2 \left\|
  F_{nm}^\varphi \ast P_m\right\|_{L^2(\mathbb{R}; L^2(\Gamma))}^2.\label{pre_young}
\end{split}
\end{equation}
In order to obtain a bound on the norms on the right-hand side
of~\eqref{pre_young} we rely on the Bochner version~\cite[Lem.\
1.2.30]{Weis} of Young's convolution inequality, and we thus first
establish that the required hypotheses are satisfied, namely, that
\begin{equation}\label{w_psi_spaces}
    F_{nm}^\varphi \in L^1(\mathbb{R})\quad\mbox{and}\quad P_m \in L^2(\mathbb{R}; L^2(\Gamma)) \quad\mbox{for}\quad  0 \le m \le n.
\end{equation}
The first of these relations is easily established: since $w$ is
smooth and compactly supported, it is in the Schwartz class, and
thus~\cite{Folland}, its Fourier transform is also in the Schwartz
class---and, in particular, $\widehat{w}$ and all of its derivatives are
elements of $L^1(\mathbb{R})$.  To verify the second relation
in~\eqref{w_psi_spaces}, on the other hand, we first note from
\Cref{3d_decay_lemma_2ndkind} that, for $\omega\geq 0$,
$\breve{\psi}_{+,0}^f$ satisfies the integral equation
\[
    \left(A_\omega \breve{\psi}_{+,0}^f\right)(\mathbf{r}, \omega) = \breve{R}_0(\mathbf{r}, \omega),
\]
where $\breve{R}_0(\mathbf{r}, \omega)$ is defined by
\Cref{Rt_def_breve}.  It then follows
from~\cref{recursion_hypothesis_squared_general} in
\Cref{3d_decay_thm_recursive_bound_general} that, for $m \le n$, there
exist constants $d_{ij}^m > 0$ such that for all
$\omega \in \mathbb{R}$, $\omega \neq \pm \omega_0$,
\begin{equation}\label{bound_on_deriv_psiplus}
    \begin{split}
      \left\| P_m(\cdot, \omega)\right\|_{L^2(\Gamma)}^2
      \le \sum_{i=0}^{m} &\sum_{j=0}^{(i+1)(q+1)-1} d_{ij}^m \omega^{2j}
      \left\|\partial_\omega^{m-i} \breve{R}_0(\cdot,
      |\omega|)\right\|^2_{L^2(\Gamma)}.
    \end{split}
\end{equation}
Integrating, we obtain
\begin{equation}\label{fixmeplease}
\begin{split}
    \int_{-\infty}^\infty \left\|\vphantom{P_m(\cdot,
    \omega)}\right.&\hspace{-1.4mm}\left.P_m(\cdot,
    \omega)\right\|_{L^2(\Gamma)}^2 \,\d\omega\\
    &\le \sum_{i=0}^{m} \sum_{j=0}^{(i+1)(q+1)-1} 2d_{ij}^m \int_0^\infty \omega^{2j}
  \left\|\partial_\omega^{m-i} \breve{R}_0(\cdot,
    \omega)\right\|_{L^2(\Gamma)}^2\,\d\omega.
\end{split}
\end{equation}
We now use \Cref{Rderiv_sobolev_bound} to estimate each term on the
right-hand side of \Cref{fixmeplease}. The main hypothesis of that
lemma, which, in the present context amounts to the requirement that
$\breve{\psi}_{*,0} \in H^{(n+1)(q+1)}(I_0;L^2(\Gamma))$, follows from
\Cref{3d_decay_lemma_2ndkind_wellposed} since $\breve{b}$ satisfies
the $s$-regularity conditions~\eqref{eq:gamma_Hs_assump} with
$s = (n+1)(q+1)+q$ (which hold in view of the $s$-regularity condition
hypotheses for $\widetilde{b}$ and the definition~\cref{def:breve} of
$\breve{b}$).  For $m \le n$ we may therefore write
\begin{equation}\label{partialomega_psibreve}
\begin{split}
  \int_{-\infty}^\infty \left\| P_m\right.&\hspace{-1.4mm}\left.
  (\cdot,
  \omega)\right\|_{L^2(\Gamma)}^2\,\d\omega\\
    &\le C_1\left\| \breve{\psi}_{*,0}\right\|_{H^{(m+1)(q+1)}(I_0;\,L^2(\Gamma))}^2 \le C_2\left\|
    \breve{\psi}_{*,0}\right\|_{H^{(n+1)(q+1)}(I_0;\,L^2(\Gamma))}^2,
\end{split}
\end{equation}
with $C_1$ and $C_2 > 0$ independent of $\widetilde{b}$ and $T_0$, which, in
particular, establishes the second relation in \Cref{w_psi_spaces}.

Having established \Cref{w_psi_spaces} we may apply Young's
convolution inequality~\cite[Lemma\ 1.2.30]{Weis} to each of the terms
in the right-hand sum in \Cref{pre_young} and obtain
\begin{equation}\label{3d_decay_thm_main_estimate}
  \begin{split}
    \left\| w_\varphi \breve{\psi}_{+,0} \right\|_{L^2(\mathbb{R};\,L^2(\Gamma))}^2
      &\le C(n) \varphi^{-2n} \sum_{m=0}^n |a_m|^2
  \left\|F_{nm}^\varphi \right\|_{L^1}^2
  \left\|P_m \right\|_{L^2(\mathbb{R}; L^2(\Gamma))}^2\\
  &\le C(n, s_0) \varphi^{-2n}
  \sum_{m=0}^n \left\|P_m \right\|^2_{L^2(\mathbb{R};\,L^2(\Gamma))}
  \end{split}
\end{equation}
(since, of course,
$\left\| F_{nm}^\varphi\right\|_{L^1(\mathbb{R}} =
\left\|\partial_\omega^{n-m}\widehat{w}\right\|_{L^1(\mathbb{R})}$), which using
\Cref{partialomega_psibreve} yields
\begin{equation*}
    \left\|w_\varphi \breve{\psi}_{+,0}\right\|_{L^2(\mathbb{R};\,L^2(\Gamma)}^2 \le C(\Gamma, \tau, n, s_0) \varphi^{-2n}
    \left\|\breve{\psi}_{*,0}\right\|_{H^{(n+1)(q+1)}(I_0;\,L^2(\Gamma))}^2.
\end{equation*}
The claimed relation \Cref{L2_rchidecay} follows from this inequality
since, according to \Cref{rem:breve},
$\breve{\psi}_{*,0}(\bfr', t') = w_-(t' + T_* +
\tau)w_+(t')\breve{\psi}(\bfr',t')$, and since each derivative of each
one of the functions $w_-$ and $w_+$ is uniformly bounded.
\end{proof}

On the basis of the preparatory lemmas established above in this
section we now present the proof of \Cref{3d_decay_thm_ii}.
\begin{proof}[Proof of \Cref{3d_decay_thm_ii}]
  Letting (cf.\ \Cref{def:breve})
  \begin{equation}\label{psi_check}
    \breve{b}(\bfr, t) = b(\bfr, t + T_0)\quad\mbox{and}\quad \breve{\psi}(\bfr', \theta) = \psi(\bfr', \theta + T_0),
  \end{equation}
  instead of \Cref{density_Hp_time_bound_Tn}
  we establish the equivalent $\theta$-decay ($\theta = t - T_0 > 0$) estimate
\begin{equation}\label{density_Hp_time_bound_Tn_shifted}
  \left\|\breve{\psi}\right\|_{H^p([\theta, \infty); L^2(\Gamma))} \le C(\Gamma,
  \tau, p, n) \theta^{1/2-n} \left\|\breve{\psi}\right\|_{H^{p + n(q+1)}(I_0;
    L^2(\Gamma))} < \infty,
\end{equation}
where $I_0$ (i.e. $I_T$ for $T=0$) was introduced in
\Cref{domainofdep}. Clearly, $\breve{\psi}$ in \Cref{psi_check}
satisfies the integral equation
\begin{equation}\label{brevepsi_int_eq}
  \left(S\breve{\psi}\right)(\mathbf{r}, t) = \gamma^+ \breve{b}(\mathbf{r}, t)\quad\mbox{for}\quad (\mathbf{r}, t) \in
  \Gamma\times\mathbb{R},
\end{equation}
which coincides with \Cref{eq:tdie_sl_generic} for
$\widetilde{b}(\bfr, t) = \breve{b}(\bfr, t) = b(\bfr, t + T_0)$
and $\widetilde{\psi} = \breve{\psi}$.

To establish \Cref{density_Hp_time_bound_Tn_shifted} we first obtain
certain decay results for each element of a sequence of bounded time
intervals and then produce the final estimate by summing over the
sequence. Using \Cref{eq:POU2} with $\varphi = \theta + 3\ell s_0$
together with the identity
$\breve{\psi}(\cdot, t') = \breve{\psi}_{+,0}(\cdot, t')$ for $t' > 0$
(see \Cref{rem:breve}) and the relation $(A +B)^2\leq 2A^2 +2 B^2$ we
obtain
\begin{equation}\label{ptwise_tprime_bound}
  \begin{split}
    \left\|\breve{\psi}\right\|^2_{L^2([\theta, \infty); L^2(\Gamma))} &=  \int_\theta^\infty\hspace{-2mm} \int_\Gamma \left( \sum_{\ell =0}^\infty (w_{\theta + 3\ell s_0 + s_0} + w_{\theta + 3(\ell+1) s_0 + s_0}) \breve{\psi}_{+,0} (\bfr', t')\hspace{-1mm}\right)^2 \hspace{-2mm}\d\sigma(\bfr') \d t'\\
    &\le 4 \sum_{\ell=0}^\infty \left\|w_{\theta +
    3\ell s_0 + s_0} \breve{\psi}_{+,0}\right\|^2_{L^2(\mathbb{R}; L^2(\Gamma))}.
    \end{split}
\end{equation}
The same argument applied to $\partial_t^p \breve{\psi}$ tells us that
\begin{equation}\label{Wt_to_sum_ii}
    \left\|\partial_t^p \breve{\psi}\right\|^2_{L^2([\theta, \infty); L^2(\Gamma))} \le 4
    \sum_{\ell=0}^\infty \left\|w_{\theta +
    3\ell s_0 + s_0} \breve{\psi}_{p,+,0}\right\|^2_{L^2(\mathbb{R}; L^2(\Gamma))},
\end{equation}
where
\begin{equation}\label{check_psi_p}
  \breve{\psi}_{p,+,0} = w_+(t) \partial_t^p \breve{\psi}.
\end{equation}
Combining \Cref{ptwise_tprime_bound} and \Cref{Wt_to_sum_ii} and using
\Cref{def:sob_bochner} we obtain
\begin{equation}\label{psik_to_wT}
\begin{split}
\left\|\breve{\psi}\right\|^2_{H^p([\theta,
        \infty);\,L^2(\Gamma))} &\le  4 \sum_{\ell = 0}^\infty \left(\left\| w_{\theta + 3\ell s_0 + s_0} \breve{\psi}_{+,0} \right\|_{L^2(\mathbb{R};
        L^2(\Gamma))}^2\right.\\
        &\quad\quad\quad\quad\quad\left. + \left\| w_{\theta + 3\ell s_0 + s_0} \breve{\psi}_{p,+,0} \right\|_{L^2(\mathbb{R};
        L^2(\Gamma))}^2\right).
\end{split}
\end{equation}
To complete the proof we now estimate the $L^2$ norms on the
right-hand side of~\eqref{psik_to_wT}, that is to say, the norms
$\left\| w_{\varphi} \breve{\psi}_{+,0} \right\|_{L^2(\mathbb{R};
  L^2(\Gamma))}$ and
$\left\| w_{\varphi} \breve{\psi}_{p,+,0} \right\|_{L^2(\mathbb{R};
  L^2(\Gamma))}$ for $\varphi = \theta + 3\ell s_0 + s_0$,
$\ell = 0, 1, \cdots$; the desired result then follows by addition of
the resulting estimates.

\paragraph{\bf Decay estimate for
  $\left\| w_{\varphi} \breve{\psi}_{+,0} \right\|_{L^2(\mathbb{R};
    L^2(\Gamma))}^2$.}

Since by assumption $b$ satisfies the $s$-regularity conditions~\eqref{eq:gamma_Hs_assump}
with $s=(n+1)(q+1)+q$, and since, by hypothesis, $\breve{b}(\bfr', t')$ vanishes for
$(\bfr',t') \in \widebar{\Omega} \times \left\lbrace I_0\cup[0, \infty)\right\rbrace$,
\Cref{decay_estimate_L2} with $\widetilde{b} = b$ yields
\begin{equation}\label{rchipsi_L2_estimate}
  \left\|w_\varphi \breve{\psi}_{+,0}\right\|_{L^2(\mathbb{R};\,L^2(\Gamma)}^2 \le
    C(\Gamma, \tau, n, s_0) \varphi^{-2n}
    \left\|\breve{\psi}\right\|_{H^{(n+1)(q+1)}(I_0;\,L^2(\Gamma))}^2
\end{equation}
for arbitrary $\varphi > 0$, where $C$ denotes a constant independent
of $\varphi$, $T_0$, and $b$.

\paragraph{\bf Decay estimate for
  $\left\| w_{\varphi} \breve{\psi}_{p,+,0} \right\|_{L^2(\mathbb{R};
    L^2(\Gamma))}$ ($p > 0$).}
\Cref{int_eq_pderiv} tells us that
\begin{equation}\label{eq:pderiv_brevepsi}
    \left(S \partial_t^p \psi\right)(\bfr, t) = \gamma^+ \partial_t^p b(\bfr, t),\quad\mbox{for}\quad
    (\bfr, t) \in \Gamma \times \mathbb{R},
\end{equation}
and we may thus apply \Cref{decay_estimate_L2} to obtain decay
estimates for $w_\varphi \breve{\psi}_{p,+,0}$. Certainly,
$\partial_t^p b$ satisfies the hypotheses of that lemma: (i)
$\partial_t^p b$ vanishes for
$(\bfr',t') \in \widebar{\Omega} \times \left\lbrace I_{T_0} \cup
  [T_0, \infty)\right\rbrace$ since, by hypothesis, $b$ vanishes in
that set; and (ii) $\partial_t^p b$ satisfies the $s$-regularity
conditions~\eqref{eq:gamma_Hs_assump} with $s=(n+1)(q+1)+q$, as it
follows from the present hypotheses on $b$. Applying
\Cref{decay_estimate_L2} with $\widetilde{b} = \partial_t^p b$ we then
obtain the estimate
\begin{equation}\label{L2_deriv_rchidecay}
    \left\| w_\varphi \breve{\psi}_{p,+,0} \right\|_{L^2(\mathbb{R}; L^2(\Gamma))}^2 \le
    C(\Gamma, \tau, n, s_0) \varphi^{-2n} \left\|\partial_t^p
    \breve{\psi}\right\|_{H^{(n+1)(q+1)}(I_0;\,L^2(\Gamma))}^2,
\end{equation}
for arbitrary $\varphi > 0$, where $C$ is again a constant independent of
$\varphi$, $T_0$, and $b$.

\paragraph{\bf Combined decay estimate.}

Using  \Cref{psik_to_wT}, \Cref{rchipsi_L2_estimate}, and \Cref{L2_deriv_rchidecay} we obtain
\begin{equation}\label{Hp_tdecay_i}
    \begin{split}
    \left\|\breve{\psi}\right\|^2_{H^p([\theta,
        \infty);\,L^2(\Gamma))} &\le 4 \sum_{\ell = 0}^\infty \left(\left\| w_{\theta + 3\ell s_0 + s_0} \breve{\psi}_{+,0} \right\|_{L^2(\mathbb{R};
        L^2(\Gamma))}^2\right.\\
        &\quad\quad\quad\left. + \left\| w_{\theta + 3\ell s_0 + s_0} \breve{\psi}_{p,+,0} \right\|_{L^2(\mathbb{R};
        L^2(\Gamma))}^2\right)\\
        &\le C_1 \sum_{\ell = 0}^\infty  \left( (\theta + 3\ell s_0 + s_0)^{-2n}
        \left\|\breve{\psi}\right\|^2_{H^{(n+1)(q+1)}(I_0;\,L^2(\Gamma))}\right.\\
        &\quad\quad\quad\left. + (\theta + 3\ell s_0 + s_0)^{-2n}
        \left\|\partial_t^p \breve{\psi}\right\|^2_{H^{(n+1)(q+1)}(I_0;\,L^2(\Gamma))} \right),
    \end{split}
\end{equation}
where $C_1$ is a constant dependent only on $\Gamma$,
$\tau$, $n$ and $s_0$.  It follows that
\begin{equation}\label{Hp_tdecay_ii}
    \begin{split}
        \left\|\breve{\psi}\right\|^2_{H^p([\theta, \infty);\,L^2(\Gamma))} &\le C \left\|\breve{\psi}\right\|^2_{H^{p + (n+1)(q+1)}(I_0;\,L^2(\Gamma))}
        \sum_{\ell = 0}^\infty (\theta + 3\ell s_0 + s_0)^{-2n},
    \end{split}
\end{equation}
where, again, $C = C(\Gamma, \tau, n, s_0) > 0$.  Since $\theta > 0$,
the term $(\theta + 3\ell s_0 + s_0)^{-n}$ is a strictly decreasing and
positive function of $\ell$, and, thus, estimating the sum by an
integral it is easy to check that
\[
    \sum_{\ell=0}^\infty (\theta + 3\ell s_0 + s_0)^{-2n} \le \widetilde{C} \theta^{-2n+1},\quad \widetilde{C} > 0,
\]
which in conjunction with \Cref{Hp_tdecay_ii} establishes
\Cref{density_Hp_time_bound_Tn_shifted}, and, therefore, its
equivalent decay estimate~\cref{density_Hp_time_bound_Tn}.  The
result~\cref{density_unif_time_bound_Tn}, finally, follows from use
of~\cref{density_Hp_time_bound_Tn} with $p=1$ together with
\Cref{sob_lemma}. The proof is now complete.
\end{proof}

\appendix
\section{Sobolev-Bochner spaces\label{sob-boch}}
\renewcommand{\thesection}{A}
This appendix introduces the class of Bochner spaces and the
associated trace operators utilized in this paper.
\begin{defi}\label{trace_def}
  We denote by $\gamma^+$ and $\gamma^-$ the well-known trace operators
 \[
     \gamma^+: H^s_{\mathrm{loc}}(\mathbb{R}^3 \setminus \widebar{\Omega}) \to H^{s - 1/2}(\Gamma),
  \]
  and
  \[
    \gamma^-: H^s(\Omega) \to H^{s - 1/2}(\Gamma),
  \]
  with $1/2 < s \le 1$~\cite[Thm.\ 3.37]{McLean}, that are associated with a
  Lipschitz obstacle $\Omega$ and its boundary $\Gamma$.
\end{defi}

The Sobolev-Bochner spaces and norms introduced in this section, whose
elements are functions defined on the complete real line, are similar
to, but different from, the corresponding Sobolev-Bochner spaces and
norms on finite and semi-infinite intervals that have been used in the
context of parabolic PDEs, cf.~\cite[\S 5.9]{Evans}. The corresponding
definitions for bounded and semi-infinite intervals, which are also
used in this paper, are defined analogously.
\begin{defi}[Sobolev-Bochner spaces]\label{def:sob_bochner}
  For an integer $r \ge 0$, $\partial^r_\xi u(\mathbf{r}, \xi)$
  denotes the $r$-th distributional derivative of the function $u$
  with respect to the real variable $\xi$. Further, for given integers
  $r,s \ge 0$ and a given set $\mathcal{U}\subset\mathbb{R}^3$, where
  either $\mathcal{U} = \Gamma$ equals the Lipschitz boundary of an
  open and bounded domain $\Omega\subset\mathbb{R}^3$, or
  $\mathcal{U} = D$ equals an open and bounded domain
  $D\subset\mathbb{R}^3$ with a Lipschitz boundary, we define the
  Sobolev-Bochner spaces
\begin{equation}\label{sob_bochner}
  H^r(\mathbb{R}; H^s(\mathcal{U})) = \left\lbrace u: \mathbb{R} \to H^s(\mathcal{U})\ |\
  \left\|u\right\|_{H^r(\mathbb{R};
  H^s(\mathcal{U}))} < \infty\right\rbrace,
\end{equation}
 with the norm
\begin{equation}\label{sob_bochner_norm}
  \left\|u\right\|_{H^r(\mathbb{R}; H^s(\mathcal{U}))} = \left[\left\|u\right\|_{L^2(\mathbb{R}; H^s(\mathcal{U}))}^2 +
    \left\|\partial_\xi^r u\right\|^2_{L^2(\mathbb{R}; H^s(\mathcal{U}))}\right]^{1/2},
\end{equation}
where the
  $L^2(\mathbb{R}; H^s(\mathcal{U}))$-norm of a function
  $v: \mathbb{R} \to H^s(\mathcal{U})$ is given by
\begin{equation}
  \left\|v\right\|_{L^2(\mathbb{R}; H^s(\mathcal{U}))} = \left[ \int_{-\infty}^\infty \left\|
    v(\cdot, \xi)\right\|_{H^s(\mathcal{U})}^2\,\d \xi\right]^{1/2}.
\end{equation}
Note that the integrals inherent in the Sobolev-Bochner norms may be
interpreted in the sense of Bochner~\cite{Hille,DunfordSchwartz}, or,
equivalently, and more simply, as double integrals with respect to
$\bfr\in\mathcal{U}$ and $\xi\in \mathbb{R}$.
\end{defi}

An equivalent norm for
$H^r(\mathbb{R}; H^s(\mathcal{U}))$ is given by
\begin{equation}\label{sob_bochner_fourier}
  \vertiii{u}_{H^r(\mathbb{R}; H^s(\mathcal{U}))} =
\left[\int_{-\infty}^\infty (1 + \omega^2)^{r/2} \left\|U^f(\cdot,
  \omega)\right\|^2_{H^s(\mathcal{U})}\,\d\omega\right]^{1/2},
\end{equation}
where $U^f$ denotes the Fourier transform of $u$ with respect to
$\xi$. (The equivalence of this norm to~\eqref{sob_bochner_norm}
results from a simple application of the classical Plancherel
theorem~\cite[Thm.\ 6.1]{Folland} for real-valued functions together
with Fubini's theorem; cf.~\cite{Karunakaran:98} where the general
Bochner case is considered for $r=0$.)  For integers $k, s \ge 0$ we
also utilize the spaces
\begin{equation}\label{Ck_def}
  C^k(\mathbb{R}; H^s(\mathcal{U})) = \left\lbrace u\in H^k_{{\mathrm{loc}}}( \mathbb{R};H^q(\mathcal{U})\ \lvert\ \partial_\xi^k u
    \in C(\mathbb{R}; H^q(\mathcal{U}))\right\rbrace,
\end{equation}
where $C(\mathbb{R}; H^s(\mathcal{U}))$ denotes the set of all
strongly continuous functions $u: \mathbb{R} \to H^s(\mathcal{U})$.
Note that this notation is consistent with the classical
definition~\cite{ErnGuermond} of Bochner $C^k$ spaces: any element
$u = u(\bfr, \xi)$ in the space of $C^k(\mathbb{R}; H^s(\mathcal{U}))$
introduced in~\eqref{Ck_def} is indeed $k$-times (\emph{strongly})
continuously differentiable as a function of $\xi$---since its weak
and strong derivatives coincide almost everywhere~\cite[Cor.\
64.32]{ErnGuermond}, and since, by definition, the weak derivative is
continuous.  Finally we define
$C^\infty(\mathbb{R}; H^s(\mathcal{U})) = \cap_{k=1}^\infty
C^k(\mathbb{R}; H^s(\mathcal{U}))$.

\renewcommand{\thesection}{B}
\section{Proof of \Cref{td_sl_lemma}\label{app_c}}

\begin{proof}
  We first establish the $L^2$-norm continuity of the
  operator~\eqref{S_oper_bochner} restricted to
  $\mathcal{S}(\mathbb{R},L^2(\Gamma))$. To do this, given
  $\mu \in \mathcal{S}(\mathbb{R},L^2(\Gamma))$, we use the change of
  variables
  \begin{equation}\label{tau_chvar}
    \tau = t-|\mathbf{r} - \mathbf{r}'|/c
  \end{equation}
  to re-express the norm of $S\mu$ in the form
\begin{equation}\label{eq:td_sl_l2_bound}
\|S\mu\|^2_{L^2(\mathbb{R}; L^2(\Gamma))}=  \int_{-\infty}^\infty \int_\Gamma \left|\int_\Gamma \frac{\mu(\mathbf{r}',
    \tau) }{4\pi |\mathbf{r} -
    \mathbf{r}'|}\,\d\sigma(\mathbf{r}')\right|^2\,\d\sigma(\mathbf{r})\,\d \tau.
\end{equation}
Since
$\mu \in \mathcal{S}(\mathbb{R},L^2(\Gamma))\subset
L^2(\mathbb{R},L^2(\Gamma))$, Tonelli's theorem tells us that
$\mu(\mathbf{r}',\tau)$ is an element of $L^2(\Gamma)$ for all fixed
$\tau\in\mathbb{R}$ outside a set of measure zero. We can therefore
invoke \cref{single_layer_uniform} with $\omega =0$ to obtain
  \begin{equation}
    \int_\Gamma \left|\int_\Gamma \frac{\mu(\mathbf{r}', \tau) }{4\pi
        |\mathbf{r} -
        \mathbf{r}'|}\,\d\sigma(\mathbf{r}')\right|^2\,\d\sigma(\mathbf{r})
    \leq C \left\| \mu(\cdot, \tau)\right\|_{L^2(\Gamma)}^2
  \end{equation}
  for all $\tau$ outside a set of measure zero. Integrating with
  respect to $\tau$ we then obtain
  \begin{equation}\label{S_in_L2_2}
    \left\| S\mu\right\|^2_{L^2(\mathbb{R}; L^2(\Gamma))}  \le C \int_{-\infty}^\infty \left\| \mu(\cdot, \tau)\right\|_{L^2(\Gamma)}^2\,\d \tau = C\left\| \mu\right\|_{L^2(\mathbb{R}; L^2(\Gamma))}^2.
  \end{equation}
  The extension of the operator $S$ to an operator as indicated
  in~\eqref{S_oper_bochner}, as well as the continuity of that
  operator, follow from~\eqref{S_in_L2_2}, by using suitable Cauchy
  sequences and the Plancherel theorem, in view of the
  density~\cite[Prop.\ 2.4.23]{Weis} of the space
  $\mathcal{S}(\mathbb{R}; L^2(\Gamma))$ in
  $L^2(\mathbb{R}; L^2(\Gamma))$.

  In view of the continuity of the operators $\widetilde S$, $S$, and
  the Fourier transform $\mathcal{F}$ in the space
  $L^2(\mathbb{R}; L^2(\Gamma))$ (cf.,
  respectively,~\eqref{SL_cont_Bochner}, \eqref{S_oper_bochner}, and
  the Bochner-Plancherel theorem~\cite{Karunakaran:98}), together with
  the aforementioned density of the space
  $\mathcal{S}(\mathbb{R}; L^2(\Gamma))$ in
  $L^2(\mathbb{R}; L^2(\Gamma))$, to
  establish~\eqref{eq:single_layer_op_ii} it suffices to show that
  this equality holds for for each
  $\mu\in \mathcal{S}(\mathbb{R}; L^2(\Gamma))$. But a function
  $\mu\in \mathcal{S}(\mathbb{R}; L^2(\Gamma))$, $\mu = \mu(\bfr, t)$,
  is necessarily integrable for
  $(\mathbf{r},t)\in \Gamma\times\mathbb{R}$, and we may thus use the
  change of variables~\eqref{tau_chvar} followed by the Fubini theorem
  to obtain
\begin{equation}\label{fubini_1}
  \begin{split}
    \int_{-\infty}^\infty \e^{-\i\omega t}\int_\Gamma \frac{\mu(\mathbf{r}', t - |\mathbf{r} - \mathbf{r}'|/c)}{4\pi |\mathbf{r} - \mathbf{r}'|} &   \d\sigma(\mathbf{r}') \,\d t 
    \\&= \int_\Gamma \frac{\e^{\i\frac{\omega}{c} |\mathbf{r} - \mathbf{r}'|}}{4\pi |\mathbf{r} - \mathbf{r}'|}\,\d\sigma(\mathbf{r}') \int_{-\infty}^\infty \e^{-\i\omega \tau}  \mu(\mathbf{r}', \tau)\,\d \tau,
  \end{split}
\end{equation}
establishing~\eqref{eq:single_layer_op_ii} and thus completing the proof of the lemma.
\end{proof}
\renewcommand{\thesection}{C}
\section{Frequency-differentiated integral equation
  solutions\label{app}}
This appendix presents necessary technical lemmas concerning the
differentiability of frequency domain solutions with respect to the
temporal frequency $\omega$. In detail, \Cref{op_deriv_define}
introduces relevant boundary integral operators with
frequency-differentiated kernels;
Lemmas~\ref{omega_explicit_norms_deriv_SK}
and~\ref{omega_explicit_norms_deriv_int_op} show that certain standard
integral operators in scattering theory are strongly differentiable
with respect to frequency, and they present frequency-explicit bounds
on the norms of the frequency-derivatives of these operators; and
\Cref{int_eq_freq_regularity} establishes that, for temporally
compactly supported data $\widetilde{b}$, the frequency-domain
integral equation solution is infinitely differentiable as a function
of frequency, and its derivatives can be produced by means of
Leibniz's formula.
\begin{defi}\label{op_deriv_define}
  With reference to~\Cref{Aop_def}, for $m \in \mathbb{N}_0$ define
  the operators $S_{\omega, m}$ and $K_{\omega,m}^*$ by
  $S_{\omega,0}= S_{\omega}$, $K^*_{\omega,0} = K^*_\omega$, and, for
  $m\geq 1$,
    \begin{equation}\label{Somegaop_deriv_def}
      \left(\left(S_{\omega,m}\right)\mu\right)(\mathbf{r}) \coloneqq
    \int_\Gamma \left(\partial_\omega^m G_\omega(\mathbf{r}, \mathbf{r}')\right)
    \mu(\mathbf{r}')\,\d\sigma(\mathbf{r}'),\quad\mathbf{r} \in \Gamma,
  \end{equation}
and
  \begin{equation}\label{Komegaop_deriv_def}
      \left( \left(K^*_{\omega,m}\right)\mu\right)(\mathbf{r})
    \coloneqq
  \int_\Gamma \left(\partial_\omega^m \frac{\partial G_\omega(\mathbf{r},
    \mathbf{r}')}{\partial
  \mathbf{n}(\mathbf{r})}\right) \mu(\mathbf{r}')\,\d\sigma(\mathbf{r}'),\quad\mathbf{r} \in \Gamma.
  \end{equation}
  Also, for a given $\omega_0 > 0$, define $A^{(m)}_\omega$ as
\begin{equation}\label{Aomegaop_deriv_def}
  \hspace{1.cm} A^{(m)}_\omega =
  \begin{cases} K_{\omega,m}^* - \i
    S_{\omega,m},&\quad\mbox{for}\quad 0 \le |\omega| < \omega_0\\
    K_{\omega,m}^* - \i m S_{\omega,m-1} -
    \i\omega S_{\omega,m},&\quad\mbox{for}\quad |\omega| > \omega_0.\hspace{1.cm}\qedhere
    \end{cases}
\end{equation}
\end{defi}

\begin{lemma}\label{omega_explicit_norms_deriv_SK}
  The operators $S_{\omega, m}$ and $K^*_{\omega, m}$ are strongly
  differentiable, and for all positive integers $m$ we have
\begin{equation}\label{Sderiv_freq}
    S_{\omega, m} = \lim_{\Delta \omega \to 0}
  \frac{S_{\widetilde{\omega}, m- 1} -
    S_{\omega, m - 1}}{\Delta \omega}, \quad \widetilde{\omega} = \omega + \Delta \omega,
\end{equation}
and
\begin{equation}\label{Kderiv_freq}
    K^*_{\omega, m} = \lim_{\Delta \omega \to 0}
  \frac{K^*_{\widetilde{\omega}, m-1} -
    K^*_{\omega, m-1}}{\Delta \omega}, \quad \widetilde{\omega} = \omega + \Delta \omega,
\end{equation}
where the limit is understood in the sense of the $L^2(\Gamma)$
operator norm.  In other words, using the symbol $\partial_\omega^m$
to denote the $m$-th derivative of an operator with respect to
frequency (defined as the limit of quotients of increments, with
convergence according to the $L^2(\Gamma)$ operator norm), we have
\[
    \partial_\omega^m S_\omega = S_{\omega, m}
\]
and
\[
    \partial_\omega^m K_\omega^* = K^*_{\omega, m},
\]
for all non-negative integers $m$. Further,
$\partial_\omega^m S_\omega$ and $\partial_\omega^m K_\omega^*$
satisfy
\begin{equation}\label{SK_deriv_bounds}
  \left\|\partial_\omega^m S_\omega\right\|_{L^2(\Gamma)\to L^2(\Gamma)} \le
  C_1,\quad\mbox{and}\quad\left\|\partial_\omega^m K_\omega^*
  \right\|_{L^2(\Gamma)\to L^2(\Gamma)} \le C_2 + C_3|\omega|,
\end{equation}
for all $\omega\in\mathbb{R}$, for certain $\omega$-independent finite
constants $C_j = C_j(\Gamma, m)$, $j = 1, 2, 3$.
\end{lemma}
\begin{proof}[Proof of \Cref{omega_explicit_norms_deriv_SK}.]
  We establish \Cref{Sderiv_freq} and \Cref{Kderiv_freq} by showing
  that there exist positive constants $\widetilde{C}_1$ and
  $\widetilde{C}_2$ dependent only on $\Gamma$ and $m$ such that for all sufficiently small
  $|\Delta \omega| > 0$ we have
\begin{equation}\label{Sop_freq_deriv_bound}
    \left\|\frac{S_{\widetilde\omega, m - 1} -
      S_{\omega,m-1}}{\Delta \omega} -
      S_{\omega,m}\right\|_{L^2(\Gamma) \to L^2(\Gamma)} \le \widetilde{C}_1|\Delta
      \omega|,
\end{equation}
and
\begin{equation}\label{Kop_freq_deriv_bound}
  \left\|\frac{K^*_{\widetilde\omega,m-1} -
      K^*_{\omega,m-1}}{\Delta \omega} -
    K^*_{\omega,m}\right\|_{L^2(\Gamma) \to L^2(\Gamma)} \le \widetilde{C}_2|\Delta
  \omega|.
\end{equation}
To do this we rely on the expressions
\begin{equation}\label{omeg_deriv_G}
    \partial_\omega^m G_\omega(\mathbf{r}, \mathbf{r}') = \left(\i\frac{|\mathbf{r} - \mathbf{r}'|}{c}\right)^m
    \frac{\e^{\i\frac{\omega}{c}|\mathbf{r}
    - \mathbf{r}'|}}{4\pi|\mathbf{r} - \mathbf{r}'|},
\end{equation}
and
\begin{equation}\label{omeg_deriv_dGdn}
    \begin{split}
        \partial_\omega^m &\frac{\partial G_\omega(\mathbf{r}, \mathbf{r}')}{\partial
        \mathbf{n}(\mathbf{r})} =\\
        &= \frac{(\mathbf{r} -
    \mathbf{r}')\cdot\mathbf{n}(\mathbf{r})}{4\pi|\mathbf{r} - \mathbf{r}'|^3}
    \e^{\i\frac{\omega}{c}|\mathbf{r} - \mathbf{r}'|} \left(\i \frac{|\mathbf{r} -
    \mathbf{r}'|}{c}\right)^m \left(\i\omega \frac{|\mathbf{r} - \mathbf{r}'|}{c} + m -
    1\right)\\
    &= \frac{\i}{c} \frac{(\mathbf{r} -
    \mathbf{r}')\cdot\mathbf{n}(\mathbf{r})}{4\pi|\mathbf{r} - \mathbf{r}'|^2}
    \e^{\i\frac{\omega}{c}|\mathbf{r} - \mathbf{r}'|} \left(\i \frac{|\mathbf{r} -
        \mathbf{r}'|}{c}\right)^{m-1} \left(\i\omega \frac{|\mathbf{r} - \mathbf{r}'|}{c} + m -
    1\right),
    \end{split}
\end{equation}
that are valid for all $m\in \mathbb{N}_0$.

In order to establish \Cref{Sop_freq_deriv_bound} and
\Cref{Kop_freq_deriv_bound}, we let $a \in L^2(\Gamma)$ and we note that the
difference between $S_{\omega,m}$ and the associated quotient of
increments, when applied to $a$, reads
\begin{equation}\label{Sbound_0}
  \begin{split}
      \int_\Gamma \left[ \frac{1}{\Delta \omega} \right.&\left.\vphantom{\frac{1}{\Delta \omega}}\left(\partial_\omega^{m-1}
    G_{\widetilde{\omega}} - \partial_\omega^{m-1} G_{\omega}\right) -
    \partial_\omega^m G_{\omega}\right] a(\mathbf{r}')\,\d\sigma(\mathbf{r}')\\
    &= \int_\Gamma \frac{1}{4\pi|\mathbf{r} - \mathbf{r}'|}\left[\frac{1}{\Delta \omega}\left\lbrace
      \left(\i\frac{|\mathbf{r} - \mathbf{r}'|}{c}\right)^{m-1}\left( \e^{\i\frac{\widetilde{\omega}}{c}|\mathbf{r} -
      \mathbf{r}'|} - \e^{\i\frac{\omega}{c}|\mathbf{r} - \mathbf{r}'|}\right)
        \right\rbrace\right.\\
        &\quad \quad \left.-
        \e^{\i\frac{\omega}{c}|\mathbf{r} - \mathbf{r}'|}\left(\i\frac{|\mathbf{r} - \mathbf{r}'|}{c}\right)^m \vphantom{\frac{1}{\Delta \omega}} \right]a(\mathbf{r}')\,\d\sigma(\mathbf{r}')\\
    &= \int_\Gamma \frac{1}{4\pi|\mathbf{r} - \mathbf{r}'|} \mathcal{Q}^1_{\Delta \omega}(\mathbf{r}, \mathbf{r}')
        a(\mathbf{r}')\,\d\sigma(\mathbf{r}'),
  \end{split}
\end{equation}
where $\mathcal{Q}_{\Delta \omega}^1$ is the term in square brackets
in the next to last integral.  Similarly, the difference between
$K^*_{\omega,m}$ and the associated quotient of increments, applied to
$a$, equals
\begin{equation}\label{Kbound_0}
  \begin{split}
      &\int_\Gamma \left[ \vphantom{\frac{1}{\Delta \omega}}\frac{1}{\Delta \omega} \left(\partial_\omega^{m-1} \frac{\partial
    G_{\widetilde{\omega}}}{\partial \mathbf{n}(\mathbf{r})} - \partial_\omega^{m-1} \frac{\partial G_{\omega}}{\partial
    \mathbf{n}(\mathbf{r})}\right) - \partial_\omega^m \frac{\partial G_{\omega}}{\partial
    \mathbf{n}(\mathbf{r})}\right] a(\mathbf{r}')\,\d\sigma(\mathbf{r}')\\
    &= \int_\Gamma\frac{(\mathbf{r} - \mathbf{r}')\cdot
        \mathbf{n}(\mathbf{r})}{4\pi|\mathbf{r} - \mathbf{r}'|^2}\left[\frac{1}{\Delta \omega}\left\lbrace
        \left(\i\frac{|\mathbf{r} - \mathbf{r}'|}{c}\right)^{m-1}\left(
        \i\frac{\widetilde{\omega}}{c}\e^{\i\frac{\widetilde{\omega}}{c}|\mathbf{r} - \mathbf{r}'|}
        - \i\frac{\omega}{c}\e^{\i\frac{\omega}{c}|\mathbf{r} - \mathbf{r}'|}\right)\right.\right.\\
        &\quad\quad\left.\left.+ \frac{\i}{c}(m -
        2)\left(\i\frac{|\mathbf{r} - \mathbf{r}'|}{c}\right)^{m-2}\left(\e^{\i\frac{\widetilde{\omega}}{c}|\mathbf{r} -
        \mathbf{r}'|} - \e^{\i\frac{\omega}{c}|\mathbf{r} -
        \mathbf{r}'|}\right)\right\rbrace\right.\\
        &\quad\quad \left.-
        \frac{\i}{c}\e^{\i\frac{\omega}{c}|\mathbf{r} - \mathbf{r}'|}\left(\i\frac{|\mathbf{r} -
        \mathbf{r}'|}{c}\right)^{m-1}\left(\i\frac{\omega}{c}|\mathbf{r} -
        \mathbf{r}'| + m - 1\right) \vphantom{\frac{1}{\Delta \omega}} \right]f(\mathbf{r}')\,\d\sigma(\mathbf{r}')\\
        &= \int_\Gamma \frac{(\mathbf{r} - \mathbf{r}')\cdot
        \mathbf{n}(\mathbf{r})}{4\pi|\mathbf{r} - \mathbf{r}'|^2} \mathcal{Q}^2_{\Delta \omega}(\mathbf{r}, \mathbf{r}')
        f(\mathbf{r}')\,\d\sigma(\mathbf{r}'),
  \end{split}
\end{equation}
where we have used the second expression in \Cref{omeg_deriv_dGdn}, and
where $\mathcal{Q}_{\Delta \omega}^2$ is the quantity in square
brackets in the next to last integral.

Since, for $m\in\mathbb{N}$ and for all
$\mathbf{r}, \mathbf{r}'\in\Gamma$, $\mathcal{Q}_{\Delta \omega}^1$
and $\mathcal{Q}_{\Delta \omega}^2$ amount to difference quotients of
smooth functions of $\omega$ minus the corresponding derivatives, use
of second-order Taylor expansions of $\mathcal{Q}_{\Delta \omega}^1$
and $\mathcal{Q}_{\Delta \omega}^2$ in the variable $\omega$ ensures
that, for certain constants $D_1$ and $D_2$,
$|\mathcal{Q}_{\Delta \omega}^1(\mathbf{r}, \mathbf{r}')| \le
D_1|\Delta \omega|$ and
$|\mathcal{Q}_{\Delta \omega}^2(\mathbf{r}, \mathbf{r}')| \le
D_2|\Delta \omega|$ for all $\mathbf{r}, \mathbf{r}' \in \Gamma$, and
for all $\Delta \omega$ in a bounded interval around
$\Delta \omega = 0$.

Using the bound on $\mathcal{Q}_{\Delta \omega}^1$ together with
\Cref{Sbound_0} we obtain
\begin{equation}\label{S_bound}
    \left|\left(\frac{S_{\widetilde{\omega},m-1} -
    S_{\omega,m-1}}{\Delta \omega} - S_{\omega,m}\right)[a](\mathbf{r})\right| \le D_1|\Delta\omega| \int_\Gamma
    \frac{1}{4\pi|\mathbf{r} - \mathbf{r}'|} |a(\mathbf{r}')|\,\d\sigma(\mathbf{r}').
\end{equation}
Similarly, using the bound for $\mathcal{Q}_{\Delta \omega}^2$
together with \Cref{Kbound_0} we obtain
\begin{equation}\label{K_bound}
    \left| \left(\frac{K^*_{\widetilde{\omega},m-1} -
    K^*_{\omega,m-1}}{\Delta \omega} - K^*_{\omega,m}\right)[a](\mathbf{r})\right| \le D_2|\Delta\omega| \int_\Gamma
    \frac{1}{4\pi|\mathbf{r} - \mathbf{r}'|} |a(\mathbf{r}')|\,\d\sigma(\mathbf{r}'),
\end{equation}
where the relation
$|(\mathbf{r} - \mathbf{r}')\cdot \mathbf{n}(\mathbf{r}) / |\mathbf{r}
- \mathbf{r}'|| \le 1$ was used.  But the integral expressions on the
right-hand sides of \Cref{S_bound} and \Cref{K_bound} are Laplace
single-layer operators, equal to $S_\omega$ with $\omega = 0$,
acting on $|a|$, and, thus, by~\eqref{single_layer_uniform} we obtain
\[
    \left\| \left(\frac{S_{\widetilde{\omega},m-1} -
    S_{\omega,m-1}}{\Delta \omega} - S_{\omega,m}\right)a\right\|_{L^2(\Gamma)} \le C D_1 |\Delta \omega|\left\|a\right\|_{L^2(\Gamma)}
\]
and
\[
    \left\| \left(\frac{K^*_{\widetilde{\omega},m-1} -
    K^*_{\omega,m}}{\Delta \omega} - K^*_{\omega,m}\right)a\right\|_{L^2(\Gamma)} \le C D_2 |\Delta \omega|\left\|a\right\|_{L^2(\Gamma)},
\]
where $C = C(\Gamma)$ denotes the operator norm
$\left\|S_0\right\|_{L^2(\Gamma)\to L^2(\Gamma)}$ of the Laplace
single-layer operator $S_0:L^2(\Gamma)\to L^2(\Gamma)$. The relations
\Cref{Sop_freq_deriv_bound} and \Cref {Kop_freq_deriv_bound}, and thus
\Cref{Sderiv_freq} and \Cref{Kderiv_freq}, follow directly.

In order to establish the estimates in \Cref{SK_deriv_bounds},
finally, let $a \in L^2(\Gamma)$ and $m\in\mathbb{N}$. Then,
using~\eqref{omeg_deriv_G} and the Cauchy-Schwarz inequality we obtain
\begin{equation*}
    \begin{split}
        \left\|\left(\partial_\omega^m S_\omega\right) a\right\|_{L^2(\Gamma)}^2 &= \int_\Gamma\left| \int_\Gamma
        \left(\i
    \frac{|\mathbf{r} - \mathbf{r}'|}{c}\right)^m \frac{\e^{\i\frac{\omega}{c}|\mathbf{r} - \mathbf{r}'|}}{4\pi|\mathbf{r} -
    \mathbf{r}'|} a(\mathbf{r}')\,\d\sigma(\mathbf{r}')\right|^2\,\d\sigma(\mathbf{r})\\
        &\le \left(\frac{\diam(\Gamma)^{m-1} (\mathrm{area}(\Gamma))}{4\pi c^m}\right)^2 \int_\Gamma
        \left|a(\mathbf{r}')\right|^2\,\d\sigma(\mathbf{r}')\\
        &= C_1^2(\Gamma, m) \left\|a\right\|^2_{L^2(\Gamma)}
    \end{split}
\end{equation*}
for some constant $C_1 = C_1(\Gamma, m)$ independent of $\omega$, from
which the first inequality in \Cref{SK_deriv_bounds}
follows. Similarly, using the second expression in
\Cref{omeg_deriv_dGdn} and the fact that
$|(\mathbf{r} - \mathbf{r}')\cdot \mathbf{n}(\mathbf{r}) / |\mathbf{r}
- \mathbf{r}'|| \le 1$ we obtain
\begin{equation*}
    \begin{split}
    \left\|\left(\partial_\omega^m K^*_\omega\right) a\right\|_{L^2(\Gamma)}^2 &=
    \int_\Gamma\left| \int_\Gamma \frac{\i}{c} \frac{(\mathbf{r} -
    \mathbf{r}')\cdot\mathbf{n}(\mathbf{r})}{4\pi|\mathbf{r} - \mathbf{r}'|^2}
    \e^{\i\frac{\omega}{c}|\mathbf{r} - \mathbf{r}'|} \left(\i \frac{|\mathbf{r} -
        \mathbf{r}'|}{c}\right)^{m-1}\right.\times \\
        &\quad\quad\times \left.\left(\i\omega \frac{|\mathbf{r} - \mathbf{r}'|}{c} + m - 1\right)a(\mathbf{r}')\,\d\sigma(\mathbf{r}')\right|^2\,\d\sigma(\mathbf{r})\\
        &\le \left[\frac{1}{c}\left(\frac{\diam(\Gamma)}{c}\right)^{m-1} \left(|\omega| \frac{\diam(\Gamma)}{c} + m -
        1\right)\right]^2\times\\
        &\quad\quad\times \int_\Gamma \left(\int_\Gamma \frac{|a(\mathbf{r}')|}{4\pi|\mathbf{r} -
        \mathbf{r}'|}\d\sigma(\mathbf{r}')\right)^2\,\d\sigma(\mathbf{r})\\
        &\le C'_1(\Gamma, m) \left(C'_2(\Gamma)|\omega| +C'_3(m)\right)^2 \left\|S_0 |a|\right\|^2_{L^2(\Gamma)}\\
        &\le  \left(C_2|\omega| + C_3\right)^2 \left\|a\right\|^2_{L^2(\Gamma)}
    \end{split}
\end{equation*}
where $C_j = C_j(\Gamma,m)$, $j=2,3$, are positive constants
independent of $\omega$, and where in the last step we used once again
the boundedness~\eqref{single_layer_uniform} in norm enjoyed by the
single-layer operator $S_0$. This establishes the second inequality in
\Cref{SK_deriv_bounds} and it completes the proof of the Lemma.
\end{proof}

\begin{lemma}\label{omega_explicit_norms_deriv_int_op}
  The operator $A_\omega^{(m)}$ in~\eqref{Aomegaop_deriv_def} is
  strongly infinitely differentiable at all $\omega \ne \pm \omega_0$,
  and for all $m\in\mathbb{N}$ we have
\begin{equation}\label{Adiv_freq}
  A_\omega^{(m)} = \lim_{\Delta \omega \to 0}
      \frac{A_{\widetilde\omega}^{(m-1)} -
      A_{\omega}^{(m-1)}}{\Delta \omega},\quad \widetilde{\omega} = \omega + \Delta \omega,
  \end{equation}
  in the sense of the $L^2(\Gamma)$ operator norm.  Further, for all
  $m\in\mathbb{N}$ and $\omega \ne \pm \omega_0$ we have
\begin{equation}\label{Aop_deriv_def}
  \partial_\omega^m A_\omega = A_\omega^{(m)},\;\;  A_\omega^{(m)} =
  \begin{cases} \partial_\omega^m K_\omega^* - \i\partial_\omega^m
    S_\omega,&\quad\mbox{for}\quad 0 \le |\omega| < \omega_0,\\
    \partial_\omega^m K_\omega^* - \i m \partial_\omega^{m-1}S_\omega -
    \i\omega\partial_\omega^m S_\omega,&\quad\mbox{for}\quad |\omega| >
    \omega_0,
  \end{cases}
\end{equation}
and $\partial_\omega^m A_\omega$ satisfies
\begin{equation}\label{Aop_deriv_bound}
  \norm{\partial_\omega^m A_{\omega}}_{L^2(\Gamma)\to L^2(\Gamma)} \le \alpha_0 +
  \alpha_1\omega
\end{equation}
for some $\omega$-independent positive constants
$\alpha_j = \alpha_j(\Gamma, m)$, $j=0,1$.
\end{lemma}
\begin{proof}
  Equations \Cref{Adiv_freq} and \Cref{Aop_deriv_def} follow directly
  from \Cref{omega_explicit_norms_deriv_SK} and the product
  differentiation rule. The frequency-explicit operator norm
  bound~\cref{Aop_deriv_bound} follows immediately from the estimates
  \Cref{SK_deriv_bounds} and~\Cref{Aop_deriv_def}. Indeed, for
  $0 \le |\omega| < \omega_0$, using the first case
  in~\Cref{Aop_deriv_def} we obtain
  \begin{equation}\label{Aomega_lowfreq_op_bound}
    \norm{\partial_\omega^m A_{\omega}}_{L^2(\Gamma)\to L^2(\Gamma)} \le
    \norm{\partial_\omega^m K_\omega^*}_{L^2(\Gamma)\to L^2(\Gamma)}
    + \norm{\partial_\omega^m S_\omega}_{L^2(\Gamma)\to L^2(\Gamma)},
  \end{equation}
  while for $|\omega| > \omega_0$, from the second case
  in~\Cref{Aop_deriv_def} we have
  \begin{equation}\label{Aomega_op_bound}
    \begin{split}
      \norm{\partial_\omega^m A_{\omega}}_{L^2(\Gamma)\to L^2(\Gamma)} &\le
      \norm{\partial_\omega^m K_\omega^*}_{L^2(\Gamma)\to L^2(\Gamma)}\\
      &+ m\norm{\partial_\omega^{m-1} S_\omega}_{L^2(\Gamma)\to L^2(\Gamma)} +
      \omega \norm{\partial_\omega^m S_\omega}_{L^2(\Gamma)\to L^2(\Gamma)}.
    \end{split}
  \end{equation}
  In either case~\eqref{Aop_deriv_bound} follows for some constants
  $\alpha_j = \alpha_j(\Gamma, m)$, and the proof is complete.
\end{proof}

\begin{lemma}\label{int_eq_freq_regularity}
  Let $\widetilde{b}$ satisfy the assumptions of
  Lemmas~\ref{3d_decay_lemma_2ndkind} and~\ref{3d_decay_thm_h_equiv}
  and let $\widetilde{\psi}$ denote the solution to
  \Cref{eq:tdie_sl_generic}. Further, using the expressions for
  $\eta_0(\omega)$ and $\widetilde{H}^f_T$ given in
  \Cref{Aomega_def_eqn} and \Cref{Ht_def_generic}, respectively,
  define $\widetilde{R}_T(\mathbf{r}, \omega)$ by
\begin{equation}\label{Rdef}
    \widetilde{R}_T(\mathbf{r}, \omega) = \gamma^- \partial_\mathbf{n} \widetilde{H}^f_T(\mathbf{r}, \omega) - \i\eta_0(\omega) \gamma^-
    \widetilde{H}^f_T(\mathbf{r}, \omega),
\end{equation}
for $\omega \ge 0$, and by Hermitian symmetry for $\omega < 0$:
$\widetilde{R}_T(\mathbf{r},
\omega)=\widebar{\widetilde{R}_T(\mathbf{r}, -\omega)}$.  Then
$\widetilde{R}_T \in C^\infty(\mathbb{R}\setminus \pm \omega_0;
L^2(\Gamma))$, the solution $\widetilde{\psi}^f_{+,T}$ to
\Cref{CFIE_proof_generic} satisfies
$\widetilde{\psi}_{+,T}^f \in C^\infty(\mathbb{R}; L^2(\Gamma))$, and,
letting $\mu = \widetilde{\psi}_{+,T}^f$ and $a_k^p = {{p}\choose{k}}$
($k=1,\dots,p$, $p\in\{\mathbb{N}\cup\{0\})$, for all
$\omega \in \mathbb{R}^+ \setminus \{ \omega_0\}$ and for all
non-negative integers $p$ we have
\begin{equation}\label{mu_freq_deriv}
  \left(A_\omega \left( \partial_\omega^p \mu\right)\right)(\mathbf{r}, \omega) =
    \partial_\omega^p \widetilde{R}_T(\mathbf{r}, \omega) - \sum_{k=1}^p a_k^p \left(\partial_\omega^k
  A_\omega\right)\left( \partial_\omega^{p - k} \mu\right)(\mathbf{r}, \omega).
\end{equation}
 \end{lemma}

\begin{proof}
  According to \Cref{h_compactsupp}, $\widetilde{h}_T$ is compactly
  supported as a function of time within the temporal interval
  $[\alpha,T-T_*-2\tau]$. Thus, by differentiation under the
  Fourier-transform integral sign we see that the Fourier transform
  $\widetilde{H}^f_T$ of $\widetilde{h}_T$ is infinitely
  differentiable in the sense introduced at the end of
  \Cref{sob-boch}:
  $\widetilde{H}^f_T \in C^\infty(\mathbb{R}; L^2(\Gamma))$. (All necessary
  differentiations under the integral sign are easily justified using
  the dominated convergence theorem, since $\widetilde{h}_T$ is
  compactly supported.) It follows that the function $\widetilde{R}_T$
  defined by \Cref{Rdef} is infinitely differentiable with respect to
  $\omega$ for each $\omega \in \mathbb{R}^+ \setminus \omega_0$ (see
  also \Cref{negative_freq} concerning negative $\omega$).

  It remains to show that
  $\mu = \widetilde{\psi}_{+,T}^f \in C^\infty(\mathbb{R};
  L^2(\Gamma))$ and that \Cref{mu_freq_deriv} holds. Suppressing $\mathbf{r}$-dependence, consider the equation
  \[
    A_\omega \mu(\omega) = R_T(\omega)\quad\mbox{or, equivalently}\quad \mu(\omega) = A^{-1}_\omega R_T(\omega).
  \]
  To establish the differentiability of $\mu$ we first show that the
  operator $A^{-1}_\omega$ is strongly infinitely differentiable (the quotients of increments converge strongly, as in Lemmas~\ref{omega_explicit_norms_deriv_SK} and~\ref{omega_explicit_norms_deriv_int_op}) with
  respect to $\omega$, and that
\begin{equation}\label{inv_der_expl}
   \partial_\omega A_{\omega}^{-1} = -A_{\omega}^{-1}\left( \partial_\omega A_{\omega}\right) A_{\omega}^{-1}.
 \end{equation}
 To do this we first note that $A_\omega^{-1}$ is strongly continuous
 with respect to $\omega$ (cf.~\cite{Ramm:04}), as it follows
 directly from the relations
 \begin{equation*}
   \begin{split}
     \left\|A^{-1}_{\omega_* + \Delta \omega} - A^{-1}_{\omega_*}\right.&\hspace{-1.5mm}\left.\right\|_{L^2(\Gamma) \to L^2(\Gamma)} = \left\|A^{-1}{\omega_* + \Delta \omega} \left( A_{\omega_* + \Delta \omega} - A_{\omega_*}\right) A^{-1}_{\omega^*}\right\|_{L^2(\Gamma) \to L^2(\Gamma)}\\
     &\le C \left\|A_{\omega_* + \Delta \omega} - A_{\omega_*}\right\|_{L^2(\Gamma) \to L^2(\Gamma)} \to  0\quad\mbox{ as }\quad \Delta \omega \to 0.
   \end{split}
   \end{equation*}
Then, using \Cref{omega_explicit_norms_deriv_int_op} together with the
identity
\[
\frac{A^{-1}_{\omega_* + \Delta \omega} - A^{-1}_{\omega_*}}{\Delta \omega} = -A^{-1}_{\omega_*} \frac{A_{\omega_* + \Delta \omega} - A_{\omega_*}}{\Delta \omega} A^{-1}_{\omega_* + \Delta \omega},
\]
in view of the strong continuity of $A_\omega$ with respect to
$\omega$, it follows that $A_\omega^{-1}$ is strongly differentiable,
and that for every $\omega \in \mathbb{R}^+ \setminus \omega_0$
equation \Cref{inv_der_expl} holds.  Utilizing the easily established
rule for differentiation of an operator product of the form
$B_{\omega}C_{\omega}$ together with~\eqref{inv_der_expl}, further,
shows that $A_{\omega}^{-1}$ is infinitely differentiable.  The
differentiability of $\mu = \widetilde{\psi}_{+,T}^f$ to all orders
and for all real frequencies $\omega$ then follows
directly---including at $\omega = \omega_0$, since, by
\Cref{3d_decay_lemma_2ndkind_wellposed}, $\widetilde{\psi}_{+,T}^f$
does not depend on the specific choice of $\omega_0$. The
expression~\eqref{mu_freq_deriv}, finally, results from an application
of Leibniz's differentiation rule to the quantity
$A_\omega\mu_\omega$. The proof is complete.
\end{proof}

\bibliographystyle{amsplain}
\bibliography{tdie}
\end{document}